\documentclass[11pt]{article}
\usepackage{amsmath,amsthm,amssymb,mathrsfs}
\usepackage{mathptmx}
\usepackage{graphicx,mdframed}
\usepackage{bbm}
\usepackage{subcaption}
\usepackage{float}
\usepackage{url}
\usepackage{psfrag}
\usepackage{xcolor}
\usepackage{tabularx, caption}
\usepackage{lscape}

\usepackage[colorlinks]{hyperref}
\hypersetup{
    colorlinks=true,
    linkcolor=blue,
    filecolor=magenta,      
    urlcolor=cyan,
}

\makeatletter

\newcommand{\opnorm}{\@ifstar\@opnorms\@opnorm}
\newcommand{\@opnorms}[1]{%
	\left|\mkern-1.5mu\left|\mkern-1.5mu\left|
	#1
	\right|\mkern-1.5mu\right|\mkern-1.5mu\right|
}
\newcommand{\@opnorm}[2][]{%
	\mathopen{#1|\mkern-1.5mu#1|\mkern-1.5mu#1|}
	#2
	\mathclose{#1|\mkern-1.5mu#1|\mkern-1.5mu#1|}
}
\makeatother

\font\tencmmib=cmmib10 \skewchar\tencmmib '60
\newfam\cmmibfam
\textfont\cmmibfam=\tencmmib

\def\lessim{\ \lower4pt\hbox{$
		\buildrel{\displaystyle <}\over\sim$}\ }
\def\gessim{\ \lower4pt\hbox{$\buildrel{\displaystyle >}
		\over\sim$}\ }
\def\n{\noindent}

\def\eps{{\varepsilon}}

\def\n{\noindent}

\newcommand{\la}{\langle}
\newcommand{\ra}{\rangle}

\newcommand{\e}{\mathbb{E}}
\newcommand{\p}{\mathbb{P}}

\newcommand{\gp}{\mathrm{GP}}
\newcommand{\bd}{\underline{ \delta}}

\newtheorem{theorem}{\bf Theorem}
\newtheorem{lemma}[theorem]{\bf Lemma}

\newtheorem{remark}[theorem]{\bf Remark}

\newtheorem{proposition}[theorem]{\bf Proposition}

\newenvironment{Proof of lemma}{\noindent{\bf Proof of Lemma}}{\hfill$\Box$\newline}
\newenvironment{Proof of theorem}{\noindent{\it Proof of Theorem}}{\hfill\scriptsize{$\Box$}\newline}
\newenvironment{Proof of theorems}{\noindent{\bf Proof of Theorems}}{\hfill$\Box$\newline}
\newenvironment{Proof of proposition}{\noindent{\bf Proof of Proposition}}{\hfill$\Box$\newline}
\newenvironment{Proof of propositions}{\noindent{\bf Proof of Propositions}}{\hfill$\Box$\newline}
\newenvironment{Proof of exercise}{\noindent{\it Proof of Exercise:}}{\hfill$\Box$}
\renewcommand{\r}{\mathbb{R}}

\def\gl{\lambda}

\numberwithin{equation}{section}
\numberwithin{theorem}{section}
\begin{document}

	\title{On $\ell_p$-Gaussian-Grothendieck problem}
	
	\author{Wei-Kuo Chen\thanks{University of Minnesota. Email: wkchen@umn.edu.  Partly supported by NSF grant DMS-17-52184} \and  Arnab Sen\thanks{University of Minnesota. Email: arnab@umn.edu}}
	
	\maketitle
	
	\begin{abstract}	
		For $p\geq 1$ and $(g_{ij})_{1\leq i,j\leq n}$ being a matrix of i.i.d. standard Gaussian entries, we study the $n$-limit of the $\ell_p$-Gaussian-Grothendieck problem defined as
		$$
		\max\Bigl\{\sum_{i,j=1}^n g_{ij}x_ix_j: x\in \mathbb{R}^n,\sum_{i=1}^n |x_i|^p=1\Bigr\}.
		$$
		The case $p=2$  corresponds to the top eigenvalue of the Gaussian Orthogonal Ensemble; when $p=\infty$, the maximum value is essentially the ground state energy of the Sherrington-Kirkpatrick mean-field spin glass model and its limit can be expressed by the famous Parisi formula. In the present work, we focus on the cases $1\leq p<2$ and $2<p<\infty.$ For the former, we compute the limit of the $\ell_p$-Gaussian-Grothendieck problem and investigate the structure of the set of all near optimizers along with stability estimates. In the latter case, we show that this problem admits a Parisi-type variational representation and the corresponding optimizer is weakly delocalized in the sense that its entries vanish uniformly in a polynomial order. 
	\end{abstract}
	
	{
  \hypersetup{linkcolor=black}
  \tableofcontents
}
	

	\section{Introduction and main results}
	
	For $1 \le p \le \infty$,  the $\ell_p$-Grothendieck problem (see \cite{KN12}) seeks to study the following optimization problem
	\begin{equation}\label{eq:Gp} 
	\max_{x\in \mathbb{R}^n:\|x\|_p \le1}\la A x,x\ra,
	\end{equation} 
	where $A$ is an $n \times n$ input matrix. 
	Here, for $1 \le p<\infty$,  the standard $\ell_p$-norm of $x \in \r^n$ is defined as
	$
	\|x\|_p:=(\sum_{i\in [n]} |x_i|^p)^{1/p}.
	$
	Additionally, the $\ell_\infty$-norm of $x$ is given by $ \|x\|_\infty = \max_{i \in [n]}  |x_i|$.

	For $p=2$,  the spectral theory tells us that  the optimal value of \eqref{eq:Gp}  is just the maximum eigenvalue of the symmetric matrix $(A+ A^T)/2$. For $p=\infty$, under a mild assumption that the diagonal entries of $A$ vanish,   the optimal value of \eqref{eq:Gp}  remains unchanged if  the quadratic form is maximized over the discrete hypercube $\{-1, 1\}^n$ instead (see, e.g., \cite[Lemma 3.5]{montanari2019optimization}). This optimization problem is known to be NP-hard \cite{MS65}. The optimal value can be approximated within a $O(\log n)$ factor in polynomial time \cite{CW04,KS03,M01,NRT99}. On the other hand,  it is computationally hard to approximate the optimal value within a $O( (\log n)^\gamma)$ factor for every $\gamma \in (0, 1/6)$ \cite{KS11}.
	
	The special cases of the $\ell_p$-Grothendieck problem appear in clustering algorithms.  The cases $p=2$  and $p=\infty$ are related  to the spectral partitioning \cite{fiedler73} and  the correlation clustering \cite{CW04}, respectively.  For $2< p < \infty$,   the $\ell_p$-Grothendieck problem  can be viewed as an interpolation between these two clustering criteria. 
	In this case, approximating \eqref{eq:Gp}  within a factor of $\xi_p^2 - \eps$ is NP-hard for any $\eps >0$, while there exists a polynomial time algorithm that approximates \eqref{eq:Gp} within a factor of $\xi_p^2$ \cite{GRSW12,KNS10}, where $\xi_p  = ( \e |z|^p )^{1/p}$ is the $p$-th norm of a standard Gaussian random variable $z$.
	
	The case $p=1$ is also believed to be computationally hard (see the discussion in \cite{KN12} and references therein).  Let us mention that if we optimize over a simplex instead, then \eqref{eq:Gp} includes, as a special case, finding the cardinality of the maximum independent set of a graph \cite{MS65}, which is a well-known NP-hard problem. The case $1< p<2$ seems to remain unexplored in the literature. 
	
	While finding an efficient algorithm to compute the optimal value in \eqref{eq:Gp}  is generally difficult except when $p=2$, it is natural to study the $\ell_p$-Grothendieck problem for random input matrices first. This leads to the following optimization problem that takes a random Gaussian matrix as input, which we will  refer to as the $\ell_p$-Gaussian-Grothendieck problem or simply the $\ell_p$-Grothendieck problem, 
	\begin{align} \label{eq:GGp} 
	\max_{x\in \mathbb{R}^n:\|x\|_p \le 1}\la G_nx,x\ra = \max_{x\in \mathbb{R}^n:\|x\|_p=1}\la G_nx,x\ra.
	\end{align}
	Here $G_n$ is an $n\times n$ matrix with entries $g_{ij}$ being i.i.d.\ standard Gaussians for all $i,j\in [n]$. 
	
	Another motivation for investigating the optimization problem \eqref{eq:GGp} arises from the study of spin glass models, which, roughly speaking, are disordered spin systems invented in the 70's in order to understand some unusual behaviors of certain alloys, such as CuMn. The quadratic form, $\la G_nx,x\ra$ in \eqref{eq:GGp}, reflecting the pairwise interactions across sites, can be viewed as (the negative of) the Hamiltonian or energy of the Sherrington-Kirkpatrick (SK) mean-field spin glass model on the unit $\ell_p$-ball. In the terminology of statistical physics, the maximum value of $\la G_nx,x\ra$  over all possible spin configurations $x \in \mathbb{R}^n$ with $\|x\|_p=1$ is called the ground state energy of the model. When $p=\infty$, this quantity is essentially the ground state energy of the classical SK model \cite{SK72} with Ising spins. In the case $p=2$,  the spin configurations lie on the standard $\ell_2$-sphere and this model is called the spherical SK model. For physicists' treatments of mean-field spin glass models, we refer the readers to check \cite{MPV87}. See also the books \cite{Pan13,Tal111,Tal112} for the mathematical progress in the past decade.
	
	In the present work, we aim to study the $\ell_p$-Gaussian-Grothendieck problems in two major directions, $(i)$ their limits as $n$ tends to infinity and $(ii)$ the structure of the set of $(1-\varepsilon)$-optimizers, $\mathcal{M}_\eps = \mathcal{M}_\eps (p, n)$, defined as the collection of all $x\in \mathbb{R}^n$ satisfying  $ \|x\|_p = 1$ and  
	\[ \la G_nx,x\ra  \ge (1-\eps) \max_{y\in \mathbb{R}^n:\|y\|_p = 1}\la G_ny,y \ra.\]
	We  divide our main results into two parts $1\leq p<2$ and $p>2$,  as they exhibit significantly different behaviors. They are presented in the following two subsections. Before we turn to their statements, we introduce some notations that will be used throughout this paper.
	
	\medskip
	
	{\noindent \bf General Notation.}
	\rm  For $1\leq p\leq \infty,$ let $1\leq p^* \leq \infty$ be the H\"older conjugate of $p$, i.e., $1/p+1/p^*=1.$ For $x\in \mathbb{R}^n,$ we define the normalized $\ell_p$-norm of $x$ as 
	$$
	\opnorm{x}_p=\Bigl(\frac{1}{n}\sum_{i\in [n]}|x_i|^p\Bigr)^{1/p}\,\,\text{ for }  1\leq p<\infty\,\, \ \ \mbox{and}\,\, \ \ \opnorm{x}_\infty=\max_{i\in[n]}|x_i|.
	$$
	In addition, we use $\|x\|_0$ to denote the number of nonzero coordinates in $x.$
	Let  $\bar G_n=(G_n+G_n^T)/\sqrt{2}$ be the symmetrized Gaussian matrix, distributed as the Gaussian Orthogonal Ensemble (GOE). The entries of $\bar G_n$ are denoted by $\bar g_{ij}.$ For $1\leq p<\infty,$ we denote by $\xi_p  = ( \e |z|^p )^{1/p}$, the $p$-th norm of a standard Gaussian random variable $z$.
	
	Let $M_n(\mathbb{R})$ be the space of real-valued matrices of size $n\times n.$ For any $A=(a_{ij})_{i,j\in [n]}\in M_n(\mathbb{R}),$ denote its $p$-to-$q$ operator norm of $A$ by $\|A\|_{p\to q}=\sup_{\|x\|_p=1}\|Ax\|_q$ for $1\leq p,q\leq \infty$. In the case $p=q=2,$ we simply denote $\|A\|_{2\to 2}$ by $\|A\|_2.$ The Frobenius norm of $A$ is defined as $\|A\|_F=\sum_{i,j\in [n]}|a_{ij}|^2.$  We use $e_1,\ldots,e_n$ to denote the standard basis of $\mathbb{R}^n.$ 
	


	\subsection{Main result: $1\leq p < 2$}
	
	We split our results into two cases, $p=1$ and $1<p<2$, according to their different types of scalings, $\sqrt{\log n}$ and $n^{1/p^*}.$ We start with the easy case of $p=1$.
	
	\begin{theorem} \label{prop:p=1} Let $p=1$. 
		The following statements hold:
		\begin{itemize}
			\item[$(i)$] Almost surely, 
			\begin{equation} \label{eq:Gnp_limit_p=1}
			\lim_{n \to \infty}  \frac{1}{ \sqrt{\log n} } \max_{\|x\|_1=1}\la G_nx,x\ra =  \sqrt{2}.
			\end{equation}
			\item[$(ii)$] Let $I=\mbox{argmax}_{i\in [n]}g_{ii}$. For any $\eta>0$ and $n \ge 1$, 
			\begin{equation} \label{eq:lb_p=1}
			\p \big ( \la G_ne_I,e_I\ra \ge \sqrt{ 2(1-\eta) \log n} \big )  \ge 1 - e^{- n^\eta}.
			\end{equation}
			\item[$(iii)$] There exist constants $c_0,c, C>0$ such that for any $0<\delta\leq c_0$ and for all $n \ge 1$, 
			\begin{equation} \label{eq:stability_p=1}
			\p \Big (  \max_{\|x\|_1=1, \|x\|_\infty \le 1 - \delta }\la G_nx,x\ra \le \Bigl(1- c \delta +\sqrt{\frac{\log 2}{\log n}}\Bigr) \sqrt{2 \log n} \Big) \ge 1 - Cn^{-c\delta}.
			\end{equation}
		\end{itemize}
	\end{theorem} 
	Items $(i)$ and $(ii)$ together imply that $e_I$ is a near optimizer of the $\ell_1$-Grothendieck problem. Furthermore, from $(iii)$, if we take $\eps$ satisfying $0<\eps\leq 1/4$ and $\lim_{n\to\infty}\eps \sqrt{\log n}=\infty$, then for all large $n$, with high probability, the set  $\mathcal{M}_\eps$ is a subset of the disjoint union of  $\ell_1$-balls of radius $\delta = O(\eps)$ centered around the $2n$ coordinate vectors $\pm e_1, \pm e_2, \ldots, \pm e_n$. Consequently, each near optimizer is localized with a single entry carrying at least $1- \delta$ of the weight.  
	

	Next we consider the case $1< p < 2$, which is more interesting.  Define $v_i \in \r^n$ as 
	\begin{align}
	\label{extra:eq8}
	v_i(j) = \mathrm{sgn} (  \bar g_{ij}) \Bigl( \frac{  | \bar g_{ij}|^{p^*} }{ \| \bar G_n e_i \|_{p^*}^{p^*}} \Bigr)^{1/p},\,\,\forall j\in [n],
	\end{align} 
	where $\mathrm{sgn}(a)$ if $1$ if $a\geq 0$ and $-1$ if $a<0.$
	Note that $\| v_i\|_p = 1=\|e_i\|_p.$ 
	Let 
	\begin{align}\label{add:eq---1}
	\mathcal{O}  = \Big\{ \pm \frac{e_i+  v_i}{ \| e_i+  v_i \|_p} : 1 \le i \le n  \Big \}  .
	\end{align} 
	Note that ignoring the signs of the  the vectors in $\mathcal{O}$, it can be seen that they are asymptotically orthogonal to each other. The following theorems  establish the limit of the $\ell_p$-Grothendieck problem and show that any near optimizer lies close to the set $\mathcal{O}$.
	
	\begin{theorem}\label{eq:thm_p<2}
		Let $1 < p< 2$. 
		\begin{itemize}
			\item[$(i)$] Almost surely, 
			\begin{equation} \label{eq:Gnp_limit_p<2}
			\lim_{n \to \infty}  \frac{1}{ n^{  1/p^*}} \max_{\|x\|_p=1}\la G_nx,x\ra =  2^{\tfrac{1}{2} - \tfrac{2}{p}} \xi_{p^*}.  
			\end{equation}
			\item[$(ii)$] Each $x \in \mathcal{O}$ is a near optimizer in the sense that for any $d>0,$ there exists a constant $C>0$ such that with probability at least $1-Cn^{-d},$
			\begin{equation} \label{eq:Gnp_limit_approx_opt}
			\min_{x \in \mathcal{O} }\la G_nx,x\ra \geq  2^{\tfrac{1}{2} - \tfrac{2}{p}} \xi_{p^*} n^{ \frac{1}{p^*}} -C\bigl(n^{-\frac{1}{p} + \frac{1}{2} + \frac{1}{p^*}} ( \log n )^{\frac{p^*}{4}}+\sqrt{\log n}\bigr).
			\end{equation}
		\end{itemize}
	\end{theorem}
	
	\begin{remark}\rm
		The order of the error term on the right-hand side of \eqref{eq:Gnp_limit_approx_opt} depends on the value of $p.$ When $1<p<4/3$, it is of order $\sqrt{\log n}$ and when $4/3\leq p<2$, it is of order $n^{-1/p+1/2+1/p^*}(\log n)^{p^*/4}$.
	\end{remark}

	\begin{theorem}	[Stability]\label{extra:thm1}
		$\mathcal{O}$ is essentially the set of all optimizers. More precisely, let $0<\kappa < {1}/{3p}$ and $d>0.$ There exist positive constants $C_1,C_2,C_3$ such that for any $n\geq 1,$  if $n^{-\kappa}\leq \delta\leq 1$,  the event 
		\begin{equation}\label{stability}
		\max_{\mathrm{dist}(x, \mathcal{O}) \ge \delta, \|x\|_p = 1}\la G_nx,x\ra \le 2^{\tfrac{1}{2} - \tfrac{2}{p}} \bigl(1-C_1\delta^6\bigr)\xi_{p^*}n^{  1/p^*} + C_2n^{p/(2p^*)}  \sqrt{\log n}
		\end{equation}
		occurs with probability at least $1-C_3n^{-d}$, where $\mathrm{dist}(x, \mathcal{O})  = \min_{ y \in \mathcal{O} } \| x -y \|_p$.
	\end{theorem}

	As a consequence of the above two theorems, it can be deduced from \eqref{stability} that if $$
	0<\eps  \le \tfrac{1}{4}\,\,\mbox{and}\,\,\lim_{n\to\infty}\frac{\eps}{ n^{ -\frac{1}{p^*}(1-\frac{p}{2})}\sqrt{\log n}}=\infty,$$ then with high probability,   $\mathcal{M}_\eps$ is a subset of the disjoint union of  $\ell_p$-balls of radius $\delta = O(\eps^{1/6})$ centered around the $2n$ vectors in $\mathcal{O}$. The dependence of $\delta$ on $\eps$ here is  sub-optimal. Note that from the law of large number,
	$$\frac{e_i+  v_i}{ \|e_i+  v_i \|_p}  \approx 2^{-1/p} (e_i +w_i), $$ where $$w_i(j)  = n^{-1/p} \xi_{p^*}^{-p^*/p}\mathrm{sgn}(\bar g_{ij}) |   \bar g_{ij} |^{p^*/p}. $$ Therefore, each near optimizer has a single spike that has magnitude of order $1$ while the rest of entries are $O(\delta)$. Hence,  in the regime $1<p<2$,  the near optimizers possess both localized and delocalized components.

	\subsection{Main result: $2\leq p\leq  \infty$}
	Consider the normalized Grothendieck problem
	\begin{align}\label{gp}
	\gp_{n,p}:=\max_{x\in \mathbb{R}^n:\|x\|_p=1}\frac{\la G_nx,x\ra}{n^{3/2-2/p}}.
	\end{align}
	As we have mentioned before, $p=2$ is the easy case, since $\gp_{n,2}=2^{-1/2}n^{-1/2}\max_{\|x\|_2=1}\la \bar G_nx,x\ra$ and the optimal value is the largest eigenvalue of the GOE, which converges to $2$ as $n\to\infty.$ The case $p=\infty$ corresponds to the ground state energy of the SK on the product space $[-1,1]^n$ and it is known that $\lim_{n\to\infty}\gp_{n,p}$ exists and can be expressed as the famous Parisi formula, see \cite{CP18}.

	Our main result focuses on the interpolating case $2<p<\infty.$ First, we study the convergence of $\gp_{n,p}$. Note that the optimization problem can be rewritten using the normalized norm as
	$$\gp_{n,p}=\max_{x\in \mathbb{R}^n:\opnorm{x}_p=1}\frac{\la G_nx,x\ra}{n^{3/2}}.$$
	In order to study the above optimization problem, it is natural to remove the $\ell_p$-constraint over the maximization set and compensate it by adding a $\ell_p$-norm potential,  namely, for $t>0$, we define
	\begin{align}\label{potential}
	L_{n,p}(t)&=\max_{x\in \mathbb{R}^n}\Bigl(\frac{\la G_nx,x\ra}{n^{3/2}}-t\opnorm{x}_p^p\Bigr).
	\end{align}
	We also introduce the following restricted version of the above Hamiltonian in \eqref{potential}  where the self-overlap $\|x\|^2$  is constrained to be a fixed value, 
	\begin{align}\label{Lnput}
	L_{n,p,u}(t)&=\max_{x\in \mathbb{R}^n:\opnorm{x}_2^2=u}\Bigl(\frac{\la G_nx,x\ra}{n^{3/2}}-t\opnorm{x}_p^p\Bigr),\ \,\,u>0.
	\end{align}
	The following theorem establishes the existence of the deterministic limits of $L_{n,p,u}(t)$ and $L_{n,p}(t)$.
	\begin{theorem}\label{thm-1}
		Let $2<p<\infty$. Almost surely,  the following limits exist,
		\begin{align}
		\begin{split}\label{thm-1:eq1}
		L_{p,u}(t)&:=\lim_{n\to\infty}L_{n,p,u}(t)=\lim_{n\to\infty}\e L_{n,p,u}(t),\,\,\text{ for all } u,t>0, 
		\end{split}\\
		\begin{split}\label{thm-1:eq2}
		L_p(t)&:=\lim_{n\to\infty}L_{n,p}(t)=\sup_{u>0}L_{p,u}(t),\,\,\text{ for all } t>0.
		\end{split}
		\end{align}
	\end{theorem}
	
	Our next result shows the existence of the limit of  $\gp_{n,p}$ and its connection to $L_p(t)$.
	
	\begin{theorem}\label{thm2}
		Let $2<p<\infty.$ Almost surely,
		\begin{align}\label{thm2:eq2}
		\gp_p:=\lim_{n\to\infty}\gp_{n,p}=\frac{p}{2}\Bigl(\frac{p}{2}-1\Bigr)^{2/p-1}t^{2/p}L_p(t)^{1-2/p},\,\,\text{ for all } t>0.
		\end{align}
		In addition,  for $\delta=p/2-1,$ there exist $C,C'>0$ such that with probability at least $1-C'e^{-n/C'},$ if $x_*\in \mathbb{R}^n$ with $\|x_*\|_p=1$ is an optimizer of \eqref{gp}, 
		\begin{align}\label{thm2:eq3}
		\|x_*\|_\infty\leq Cn^{\frac{1}{p+\delta}-\frac{1}{p}}.
		\end{align}
	\end{theorem}
	
	Let us point out that the right-hand side of \eqref{thm2:eq2} is independent of $t>0.$ The inequality \eqref{thm2:eq3}  implies that any optimizer of the $\ell_p$-Grothendieck problem is weakly delocalized  for $2<p<\infty$  and its $\ell_p$-mass is not concentrated on any set of coordinates with vanishing proportion.  This resembles, at least qualitatively,  the delocalization behavior in the $\ell_2$-Grothendieck problem, where the supremum norm of the optimizer (the largest eigenvector of the GOE) is  $O((\log n)^{1/2} n^{-1/2})$ with high probability. 
	
	Our next result establishes the continuity of the limiting value $\gp_p$ in $p\in [2,\infty].$

	\begin{theorem}\label{thm3}
		$\gp_p$ is continuous in $p\in [2,\infty].$
	\end{theorem}
	
	\begin{remark}\rm
		From Theorems \ref{eq:thm_p<2} and \ref{thm2}, the $\ell_p$-Grothendieck problem scales as $n^{1/p^*}$ for $1<p<2$ and as $n^{3/2-2/p}$ for $2\leq p\leq \infty$.
		These scaling exponents match at $p=2$. However,  while the limiting value of the $\ell_p$-Grothendieck problem is continuous in $p$ in the intervals $(1,2)$ (as seen from \eqref{eq:Gnp_limit_p<2}) and $[2,\infty],$ it  has a jump discontinuity at $p=2$. Indeed, 
		\begin{align*}
		\lim_{p\to 2^-}\lim_{n \to \infty}  \frac{1}{ n^{  1/p^*}} \max_{\|x\|_p=1}\la G_nx,x\ra&=\lim_{p\to 2^-}2^{1/2-2/p}\xi_{p^*}=\frac{1}{\sqrt{2}}\neq \sqrt{2}=\gp_2.
		\end{align*}
	\end{remark}

	Our final result provides a description of the limiting value $\gp_p$ by showing that $L_{p,u}(t)$ and $L_p(t)$ admit Parisi-type variational formulas. In physicists' literature, the original Parisi formula was proposed by Parisi \cite{Parisi79,Parisi801,Parisi802}, which gives a variational representation for the limiting free energy of the classical SK model \cite{SK72} at positive temperature. It minimizes a functional,  which involves the solution to a second-order parabolic PDE (known as the Parisi PDE) and a linear term,  with respect to the so-called functional order parameters. This formula was first rigorously verified by Talagrand \cite{Tal06} and later generalized to some variants of the SK model, see \cite{ArguinChatterjee13,Chen13Spherical,JKS18inhomogeneousPotts,JustinKo19,Ko20multispherical,Pan05,Pan14,Pan15Multiplespecies,Pan18vector,Pan18Potts}. 
	Furthermore, the Parisi-type formulas were also established for the ground state energies, corresponding to zero temperature,  in a number of related models, see \cite{AZ20,CP18,ChenSen17,JKS18inhomogeneousPotts,JS17}. 
	
	To state our result, let $u>0$ and $\mathcal{N}_u^d$ be the collection of all  positive measures $\gamma$ on $[0,u]$ such that $\gamma$ is atomic with finitely many atoms.
	For $(\lambda,\gamma)\in \mathbb{R}\times\mathcal{N}_u^d$ and $t > 0$, set 
	$$
	f_{\lambda}^t(x)=\sup_{r\in \mathbb{R}}\bigl(rx+\lambda r^2-t|r|^p\bigr),\,\,x\in \mathbb{R}
	$$ 
	and denote, with a slight abuse of notation,
	$\gamma(s)=\gamma([0,s]).$
	Let $\Psi_{\lambda,\gamma}^{t,u}$ be the solution to the following Parisi PDE
	\begin{align}\label{pde}
	\partial_s\Psi_{\lambda,\gamma}^{t,u}(s,x)&=-\bigl(\partial_{xx}\Psi_{\lambda,\gamma}^{t,u}(s,x)+\gamma(s)\bigl(\partial_x\Psi_{\lambda,\gamma}^{t,u}(s,x)\bigr)^2\bigr),\,\,(s,x)\in [0,u)\times \mathbb{R}
	\end{align}
	with boundary condition
	$
	\Psi_{\lambda,\gamma}^{t,u}(u,x)=f_\lambda^t(x).
	$
	
	The fact that $\gamma$ is atomic with finitely many atoms enables us to solve this Parisi PDE explicitly  by using the Hopf-Cole transformation in an iterative fashion. 
	 Towards this end, let us express any  $\gamma\in \mathcal{N}_u^d$ as 
	\begin{align}\label{gamma}
	\gamma(s)=\sum_{l=0}^{k-1} m_l1_{[q_l,q_{l+1})}(s)+m_k1_{[q_{k},u]}(s)
	\end{align}
	for some
	\begin{align*}
	q_0&=0 < q_1 <  \cdots <q_k < q_{k+1}=u,\quad 0 \le m_0 < m_1 < \cdots <  m_{k-1} < m_{k}<\infty.
	\end{align*}
 Then  the solution to \eqref{pde} can be written iteratively as 
	\begin{align}\label{hopfcole}
	\Psi_{\lambda,\gamma}^{t,u}(s,x)&=\frac{1}{m_l}\log \e \exp m_l\Psi_{\lambda,\gamma}^{t,u}(q_{l+1},x+\sqrt{2(q_{l+1}-s)}z),\,\,\forall (s,x)\in [q_l,q_{l+1})\times\mathbb{R},
	\end{align}
	for $l = k, k-1, \ldots, 0$, where $z$ is standard Gaussian. Here, when $m_0=0$, \eqref{hopfcole} should be understood as
	\begin{align*}
	\Psi_{\lambda,\gamma}^{t,u}(s,x)=\e \Psi_{\lambda,\gamma}^{t,u}(q_1,x+\sqrt{2(q_{1}-s)}z),\,\,(s,x)\in [0,q_1)\times \mathbb{R}.
	\end{align*}
	Note that \eqref{hopfcole} is well-defined since there exist some constants $C>0$ and $a\in (0,1)$ such that $f_\lambda^t(x)\leq C(1+|x|^{1+a})  $ for all $x\in \mathbb{R}$ and this property is preserved for $\Psi_{\lambda,\gamma}^{t,u}$ by induction, see Lemma \ref{add:lem2}, which also describes the 
 regularity properties of $\Psi_{\lambda,\gamma}^{t,u}$.
	For any $(\lambda,\gamma)\in\mathbb{R}\times \mathcal{N}_u^d,$ set
	\begin{align}\label{parisifunctional}
	\mathcal{P}_{t,u}(\lambda,\gamma)&=\Psi_{\lambda,\gamma}^{t,u}(0,0)-\lambda u-\int_0^u s\gamma(s)ds.
	\end{align}
	For notational clarity, we ignore the dependence of  $\mathcal{P}_{t,u}$  on $p$.  Our result states that

	\begin{theorem}\label{thm0}
		Let $2<p<\infty.$ For any $u>0$ and $t>0,$ we have that
		\begin{align}
		\begin{split}\label{thm0:eq1}
		L_{p,u}(t)&=\inf_{\lambda\in \mathbb{R},\gamma\in \mathcal{N}_u^d}\mathcal{P}_{t,u}(\lambda,\gamma),
		\end{split}\\
		\begin{split}\label{thm0:eq2}
		L_{p}(t)&=\sup_{u>0}\inf_{\lambda\in \mathbb{R},\gamma\in \mathcal{N}_u^d}\mathcal{P}_{t,u}(\lambda,\gamma).
		\end{split}
		\end{align}
	\end{theorem}
	

	Currently, we know very little about  the set $\mathcal{M}_\eps$ for $2< p < \infty$. In contrast to  the $1\le p<2$  case, we expect that the  landscape of the $\ell_p$-Grothendieck problem has a more complex geometric structure, which does not admit a simple description. In fact, as mentioned before, the $\ell_\infty$-Grothendieck problem is essentially the ground state energy of the classical SK model and  for a fixed but small $\eps>0$,
	\cite{CHL18} established, with overwhelming probability, the existence of exponentially many vectors in  $\mathcal{M}_\eps$ that are nearly orthogonal to each other, see also \cite{chatterjee2009,DEZ15,Eldan20}. This fact is also true for $p=2$, which is an easy consequence of the fact that the eigenmatrix for the GOE is Haar-distributed. Therefore, it is natural to believe that in the interpolating case $2 < p < \infty$,  the same phenomenon continues to hold. See Section \ref{sec13} for further discussion.

	\section{Overview of the paper}
	
	Our approach is based on two major ingredients: (i) the Chevet inequality for bounding the supremum of a Gaussian bilinear form and (ii) some tools arising from the study of mean-field spin glasses. To facilitate our arguments, we provide some sketches of the proofs to our main results in two cases $1< p<2$ and $2<p<\infty$ followed by a description on the organization of the rest of the paper.

	\subsection{Proof sketch for $1<p<2$}
	
We begin by sketching the proof of the upper bound in Theorem~\ref{eq:thm_p<2}.  The key idea  is to decompose every $\ell_p$-unit vector $x\in \mathbb{R}^n$ as $x=x^o+x^\dagger,$ where for a given truncation level $\varepsilon>0$,
	$x^o$ is the delocalized part, $x_i^o=x_i1\bigl(|x_i|\leq \varepsilon\bigr)$ for $i\in [n]$, and $x^\dagger$ is the localized part, $x_i^\dagger=x_i1\bigl(|x_i|>\varepsilon\bigr)$ for $i\in[n].$ In the proof, we take $\eps = \eps_n \to 0$ at an appropriate rate. From this decomposition, an application of the bilinearity of the inner product and  the triangle inequality readily yields that 
	\begin{equation}\label{eq:inner_prod_decomp}
	\Big |  \max_{\|x\|_p=1 }\la G_nx,x\ra  -  \sqrt{2}   \max_{\|x\|_p=1}\la \bar G_n x^o, x^\dagger \ra \Big |
	\le   \max_{\|x\|_p=1} | \la G_n x^o, x^o \ra | + \max_{\|x\|_p=1}  | \la G_n x^\dagger, x^\dagger \ra|.
	\end{equation}
	To control the diagonal terms on the right-hand side, we apply the Chevet inequality, which states that for arbitrary bounded sets $S,T\subset \mathbb{R}^n,$ $\e\sup_{x\in T,y\in S}\la y,G_nx\ra$ can be controlled by the Gaussian width (a measure of ``complexity'' of a set) and the $\ell_2$-radii of the index sets $S$ and $T$. See Theorem \ref{chevet} below for the precise statement. It turns out that under the assumption $p<2,$  the set of delocalized vectors in the unit $\ell_p$-ball has a vanishing $\ell_2$-radius, which makes the first term on the right side of \eqref{eq:inner_prod_decomp} smaller order than  $n^{1/p^*}$. On the other hand, each localized vector  in the unit $\ell_p$-ball  has only a few (at most $\eps^{-p}$) non-zero entries, which guarantees that the set $\{ x^\dagger : \|x\|_p=1\}$ has a low Gaussian width. As a result, the second term on the right side of \eqref{eq:inner_prod_decomp} is of small order as well. With these observations,  we can approximate
	\[   n^{-1/p^*} \max_{\|x\|_p=1 }\la G_nx,x\ra  \approx  n^{-1/p^*}   \sqrt{2}   \max_{\|x\|_p=1}\la \bar G_n x^o, x^\dagger \ra. \]
	We now normalize $x^o$ and $x^\dagger$ by their respective $\ell_p$-norms and use the duality relation $\sup_{ \| y \|_p =1 }  \la y, u\ra = \| u\|_{p^*}$ to  obtain the following upper bound
	\begin{align}\label{sketch:eq1}
	\e \max_{\|x\|_p=1}\la x^o, \bar G_n x^\dagger\ra &\le  \e \max_{\|x\|_p=1} \Bigl(\|x^\dagger\|_p \|x^o\|_p \Bigl \| \bar G_n \frac{x^\dagger}{\|x^\dagger\|_p}\Bigr\|_{p^*} \Bigr) \le 2^{-2/p}   \e \max_{\|x\|_p=1}\Bigl \| \bar G_n \frac{x^\dagger}{\|x^\dagger\|_p}\Bigr\|_{p^*}
	\end{align}
	where in the last inequality, we used the fact  that the product $\|x^\dagger\|_p \|x^o\|_p$ is maximized  if $\|x^\dagger\|_p =  \|x^o\|_p = 2^{-1/p}$ as $\|x^o\|_p^p+\|x^\dagger\|_p^p=\|x\|_p^p= 1$. Since each $x^\dagger$ is supported on at most $\varepsilon^{-p}$ many coordinates, the last expectation in \eqref{sketch:eq1} can be  bounded above by 
	$\e \max_{y \in F} \| \bar G_n y  \|_{p^*},$ where $F$ is defined as   $F = \{ y \in \r^n: \| y\|_p=1, \| y\|_0 \le  \eps^{-p} \}$. We can exploit the low cardinality of the set $F$ to argue by using concentration and union bound that 
	\[ \e \max_{y \in F} \| \bar G_n y  \|_{p^*} \le (1+o(1)) \max_{y \in F} \e  \| \bar G_n y  \|_{p^*}\]
	For any vector $y$, we have $ n^{-1/p^*} \e  \| \bar G_n y  \|_{p^*}  \approx   n^{-1/p^*} \|y\|_2 \e  \| g  \|_{p^*} \approx \|y\|_2   \xi_{p^*}  $ for $g$ a standard Gaussian vector, which implies that
	\begin{equation}\label{sketch:eq2}
	n^{-1/p^*} \e \max_{y \in F} \| \bar G_n y  \|_{p^*} \le (1+ o(1))  \xi_{p^*}  \max_{y \in F} \| y\|_2 = (1+ o(1))  \xi_{p^*},
	\end{equation}
	yielding the upper bound (for the expectation) in  Theorem~\ref{eq:thm_p<2}. The above argument also indicates how we can achieve a matching lower bound. A near maximizer $x$ should 
	satisfy 
	$\|x^\dagger\|_p =  \|x^o\|_p = 2^{-1/p}$. Moreover, in \eqref{sketch:eq2}, to attain the maximum $\ell_2$-norm,   the optimal $y = x^\dagger / \|x^\dagger\|_p \in F$ should be one of the coordinate vectors  $e_i$ (up to a sign). Working backwards,  the optimal choice of $x^o$ is now dictated by the relation $ \la x^o/ \| x^o\|_p,  \bar G_n e_i \ra =  \| \bar G_n e_i \|_{p^*}$, which is achieved by  taking $x^o/ \| x^o\|_p  = v_i$, as given in \eqref{extra:eq8}. Finally, we note that these  choices  for $x^\dagger$ and $x^o$ are valid since $e_i$ is localized and $v_i$ is delocalized for any $i$. It leads to the set of  near maximizers $\mathcal{O}$, as defined in \eqref{add:eq---1}.

	%
	
	The stability bound in Theorem~\ref{extra:thm1} follows a similar path. However, it is more technically involved as we need to keep track of error incurred in every step of the above argument. Also,  we use an appropriate  stability bound (in $\ell_p$ distance) for H\"older's inequality (stated in Lemma~\ref{lem:holder_stability}). 
	
	\subsection{Proof sketch for $p>2$} \label{sec2.2}
	
	We continue to sketch the proofs of Theorems \ref{thm-1},  \ref{thm2}, and \ref{thm0}.
	First of all,  consider the problem of maximizing a collection of random variables, $(H_n(x))_{x\in S}$, indexed by a measurable subset $S\subset \mathbb{R}^n$. In physics, $H_n(x)$ is called the Hamiltonian or the energy corresponding to the spin configuration $x\in S$ and the quantity $\max_{x\in S}H_n(x)$ is called the ground state energy.  To compute this  maximum, it is often more convenient to consider the free energy associated to $(H_n(x))_{x\in S}$,
	$$
	F_n(\beta)=\frac{1}{\beta n}\log \int_S e^{\beta H_n(x)}\mu_n(dx)
	$$
	for an appropriately chosen probability measure $\mu_n$ supported on $S$, where $\beta>0$ is usually called the (inverse) temperature. The key fact here is that in many important models, e.g., the classical SK model and its variants, the free energies
	and the corresponding Gibbs measures exhibit several nice physical and mathematical properties, which allow one to establish the existences of the limiting free energies as $n\to\infty $ and furthermore derive certain types of variational formulas for these limits, see \cite{Pan13,Tal111,Tal112}. Once the limiting free energy is obtained, the limit of the maximum of $H_n$ can then be recovered by switching the $\beta$ and $n$ limits as
	$$\lim_{\beta\to\infty}\lim_{n\to\infty}F_n(\beta)=\lim_{n\to\infty}\lim_{\beta\to\infty} F_n(\beta) =\lim_{n\to\infty}\max_{x\in S}\frac{H_n(x)}{n}.$$	
	Following this idea, we establish Theorem \ref{thm-1} by introducing a free energy  (see \eqref{interpolation}) associated to $L_{n,p,u}(t)$ and adapt the Guerra-Toninelli interpolation \cite{GT02} to show that this free energy is superadditive in $n$, which implies the existence of the limiting free energy as $n\to\infty.$ After this, we send the temperature parameter to infinity to obtain the existence of the limit of $L_{n,p,u}(t)$ in \eqref{thm-1:eq1}. The convergence of $L_{n, p}(t) = \sup_{u> 0} L_{n,p,u}(t)$, claimed in \eqref{thm-1:eq2}, now follows by switching the supremum and the $n$-limit, which can be justified by  a standard compactness argument. 
	
	The proof of Theorem \ref{thm2} is  based on the observation that the $\ell_p$-Grothendieck problem and $L_{n,p}(t)$ are the Legendre duals of each other (see the proof of Lemma \ref{add:lem5}), i.e., at the point of differentiability of $L_{n,p}(t),$
	\begin{align*}
	\max_{\opnorm{x}_p^p=-L_{n,p}'(t)}\frac{\la G_nx,x\ra}{n^{3/2}}=L_{n,p}(t)-tL_{n,p}'(t).
	\end{align*}
	From this, if additionally we have $L_{n,p}'(t)<0$,  the $\ell_p$-Grothendieck problem can be written, by a change of variable, as
    \begin{align}\label{add:eq00}
	\gp_{n,p}=	\frac{1}{\bigl(-L_{n,p}'(t)\bigr)^{2/p}}  \left ( \max_{\opnorm{x}_p^p=-L_{n,p}'(t)}\frac{\la G_nx,x\ra}{n^{3/2}} \right ) =\frac{L_{n,p}(t)-tL_{n,p}'(t)}{\bigl(-L_{n,p}'(t)\bigr)^{2/p}}.
	\end{align}
	To further simplify this equation, another key observation is that at any point $t$ of differentiability of  $L_{n,p}$, 
	\begin{align}\label{add:eq000}
	L_{n,p}(t)=-\bigl(p/2-1\bigr)tL_{n,p}'(t),
	\end{align}  
	The above identity can be  verified via a straight-forward computation. Consequently, if $L_{n,p}$ is differentiable at $t$ and $L_{n,p}'(t)<0$, plugging \eqref{add:eq000} into \eqref{add:eq00} leads to
	\begin{align}\label{add:eq01}
		\gp_{n,p}=\frac{p}{2}\Bigl(\frac{p}{2}-1\Bigr)^{2/p-1}t^{2/p}L_{n,p}(t)^{1-2/p}.
	\end{align}
	To send $n\to\infty,$ it is crucial to realize that $L_{n,p}$ and $L_p$ are convex on $(0,\infty)$, which readily imply that,  with probability one, for almost everywhere $t\in(0,\infty)$, $L_p'(t)=\lim_{n\to\infty}L_{n,p}'(t)$. In addition, since $L_p(t)$ is strictly positive in $t$, the equation \eqref{add:eq000} ensures that whenever $L_{n,p}$ is differentiable at some $t$, we have $L_{n,p}'(t)<0$ as long as $n$ is large enough. These allow us to send $n\to\infty$ in \eqref{add:eq01} to obtain \eqref{thm2:eq2}  almost everywhere for $t\in (0,\infty).$ Using the fact that $L_p(t)$ is continuous in $t$ validates \eqref{thm2:eq2} on the whole space $(0,\infty).$ The proof for the delocalization of the optimizer in the $\ell_p$-Grothendieck problem in \eqref{thm2:eq3} is established by using the idea of Lagrange multiplier corresponding to the maximizer $x_*$, which upon combining with the H\"older inequality,   leads to the bound
	\begin{align*}
	\|{x_*}\|_{p+\delta}^{p+\delta}&\leq \frac{n^{1-(p+\delta)/p}}{\sqrt{2}\gp_{n,p}}\frac{\|\bar G_n\|_2}{n^{1/2}}\,\,\ \text{for}\,\,\  \delta = p/2-1.
	\end{align*}
	The assertion \eqref{thm2:eq3} then follows from the trivial bound $\|x_*\|_\infty\leq 	\|{x_*}\|_{p+\delta}$. 
	
	
	For Theorem \ref{thm0}, to approximate $L_{n,p,u,\varepsilon}(t)$, we consider the Hamiltonian $H_{n,p,t}(x)=n^{-1/2} \la G_nx,x\ra -t\|x\|_p^p$ and  introduce two related  free energies, $F_{n,t,u,\varepsilon}(\beta)$ and $F_{n,t,u,M,\varepsilon}(\beta)$. The former integrates $H_{n,p,t}$ over a spherical constraint $\|x\|^2\in [u-\varepsilon,u+\varepsilon]$ and the latter puts an additional constraint $x\in [-M,M]^n$ - see the beginning of Section \ref{sec10}. The free energy $F_{n,t,u,\varepsilon}(\beta)$ can be bounded from above by the so-called Guerra Replica Symmetry Breaking bound \cite{guerra2003broken}, which in turn implies that $L_{p,u}(t)$ is bounded above by the desired Parisi-type formula (the right-hand side of  \eqref{thm0:eq1}) after sending the temperature parameter to infinity, see Section \ref{12.7}.
	
	Deriving a matching lower bound is more difficult. First of all, it is easy to see that $L_{n,p,u,\varepsilon}(t)\geq F_{n,t,u,M,\varepsilon}(\beta)$, where the right-hand side, after a  change of variable  $x \to x/M$, is  the free energy of the SK model on the product space $[-1,1]^n$ and its limit can be expressed as a Parisi-type formula \cite{Pan05}, similar to \eqref{thm0:eq1}, as stated in  \eqref{Pf:freeenergy}. In order to obtain the matching lower bound, we need to study the $\beta$-limit and then the $M$-limit in this formula. In view of the Parisi PDE term in \eqref{Pf:freeenergy}, it is a Hamilton-Jacobi-Bellman equation induced by a linear diffusion control problem and is known to admit a stochastic optimal control representation \cite{AC15,BovierKlimovsky09,JT16}, which is described in  Proposition \ref{property1} below. Using this expression, we can successfully handle the $\beta$-limit  by adapting and refining the approaches in \cite{AC171,JS17}. In contrast, the argument for the $M$-limit is more intricate mainly due to the fact that the boundary condition of the Parisi PDE is not Lipschitz and as a consequence, the control of the $M$-limit in the Parisi PDE solution can not be done directly from the stochastic optimal control representation as the control processes therein lack uniform pointwise controls, see Remarks~\ref{rmk2} and \ref{rmk4}. Fortunately,  this technical issue can resolved by showing that the control processes are indeed uniformly square integrable relying on the assumption $p>2.$
	
	\subsection{Organization} 
	
	The rest of paper consists of two major parts. The first is devoted to proving the results for $1\leq p<2.$ In Section \ref{sec3}, we introduce the Chevet inequality and explain how one can obtain the correct scalings for our $\ell_p$-Grothendieck problem from this inequality. Additionally, we list some standard results on the Gaussian concentration and  maximal inequalities. These contents will be repeatedly used in the proofs of Theorems \ref{prop:p=1} and \ref{eq:thm_p<2} in Sections \ref{sec4} and \ref{sec5}, respectively. 
	
	The second part establishes our main results for $2<p<\infty$ based the methodologies generally used in statistical physics and in mean-field spin glasses. The proofs of Theorems \ref{thm-1}, \ref{thm2}, and \ref{thm3} are presented in Sections \ref{sec7}, \ref{sec8}, and \ref{sec9}, respectively. The more intricate part is the proof of Theorem \ref{thm0} that will be handled in Section \ref{sec12}. To this end, Sections \ref{sec10} and \ref{sec11} prepare some fundamental results. More precisely, Section \ref{sec10} derives upper and lower inequalities for $L_{p,u}(t)$ in terms of the free energies $F_{n,t,u,\varepsilon}(\beta)$ and $F_{n,t,u,M,\varepsilon}(\beta)$. In Section \ref{sec11}, we show that the $\beta$-limit of the Parisi formula associated to $\lim_{\varepsilon\downarrow 0}\lim_{n\to\infty}F_{n,t,u,M,\varepsilon}(\beta)$ can also be expressed as a Parisi-type formula followed by some a priori bounds for the corresponding minimizers in terms of the parameter $M.$
	
	Some open questions of interest are mentioned in Section \ref{sec13}. Finally, Appendix~\ref{appA} provides a new result on the stability of the H\"older inequality that is used in the proof of Theorem~\ref{extra:thm1}. In addition, analytical properties of the Parisi PDE \eqref{pde} and its stochastic optimal control representation are studied in Appendix~\ref{AppB}.
	
\section{Gaussian Bounds}\label{sec3}

\subsection{Chevet's inequality and operator norms} 

For a nonempty subset $S$ of $\r^n$, define  its Gaussian width by
$ w(S)  =\e \sup_{x \in S}  \langle g, x \rangle$ 
and  radius  by
$r(S) =  \sup_{ x \in  S} \| x \|_2.$ The Chevet inequality states that

\begin{theorem}[Chevet's inequality] \label{chevet}
	Let $S$ and $T$ be nonempty bounded subsets of $\r^n$. We have
	\begin{align}\label{eq:chevet}
	\max\bigl( w(S) r(T), w(T) r(S)\bigr) &\le  \e  \sup_{x \in T, y \in S}   \langle y, G_n x \rangle  \le w(S) r(T) + w(T) r(S). 
	\end{align}
\end{theorem} 

\begin{remark}\rm
	The upper bound of \eqref{chevet} goes back to the work of Chevet \cite{chevet78} (see also \cite{gordon85} for  sharper constant). 
	It can be derived using Sudakov-Fernique’s inequality (see
	Theorem 8.7.1 and Exercise 8.7.4 in \cite{HDP}). The lower bound is an easy exercise. 
\end{remark}

Chevet’s inequality helps us find the correct scaling of the $p$-to-$q$ operator norms of Gaussian matrices and perhaps surprisingly, it can be used to even pin down the limiting constants for some choices of $p$ and $q$.
In Section~\ref{sec:SGP}, we use this knowledge to identify the correct scaling for our $\ell_p$-Grothendieck problem, which  
may not obvious be at the first glance.  To elaborate how Chevet's inequality provides bounds for the operator norms,  for $p,q\in [1,\infty]$, write
\[ \| G_n \|_{p \to q} = \max_{ x\in B_p^n} \| G_n x  \|_q=\sup_{ x \in B_p^n, y \in B_{q^*}^n}  \langle G_nx, y \rangle,\]
where $B_r^n := \{ x \in \r^n: \| x \|_r   \le 1 \} $ denotes the unit $\ell_r$-ball in $\r^n$. Note that 
\begin{equation}\label{eq:id_norm}
r(B_r^n)  = \left \{ \begin{array}{cl}  1, & \mbox{if $1 \le r <2$},\\
n^{\tfrac{1}{r^*} - \tfrac{1}{2}}, & \mbox{if $2\leq r \leq \infty$},
\end{array} \right.
=n^{(\frac{1}{r^*}-\frac{1}{2})_+}=n^{(\frac{1}{2}-\frac{1}{r})_+}, 
\end{equation}
and from \cite[Example 3.5.4]{embrechts2013} and \cite[Proposition 3]{biau15},
\begin{equation}\label{eq:Gaussian_p_norm}
w(B_r^n)=\e \| g\|_{r^*}=
\left\{
\begin{array}{ll}
\sqrt{2\log n}(1+o(1)),&\mbox{if $r=1$},\\
\xi_{r^*} n^{1/r^*} + O(n^{1/r^* -1}),&\mbox{if $1< r<\infty,$}\\
\xi_1n,&\mbox{if $r=\infty$},
\end{array}\right.
\end{equation}
where $g$ is an $n$-dimensional standard Gaussian vector. For $1\leq a,b\leq \infty,$ set $\mathfrak{p}(a,b)=(1/a-1/2)_++1/b.$
By plugging $S = B_{q^*}^n$ and $T =  B_p^n$ into \eqref{eq:chevet}, a direct computation using \eqref{eq:id_norm} and \eqref{eq:Gaussian_p_norm} readily yields 

\begin{proposition}\label{prop1}
	The following two-sided bounds are valid.
	\begin{itemize}
		\item[$(i)$] For $p=1$ and $q=\infty,$
		$$
		(1+o(1))\sqrt{2\log n}\leq \e \| G_n \|_{p \to q}\leq 2(1+o(1))\sqrt{2\log n}.
		$$
		\item[$(ii)$] For $1<p\leq \infty$ and $1\leq q<\infty$, if $\mathfrak{p}(q,p^*) =\mathfrak{p}(p^*,q), $
		\begin{align*}
		(1+ o(1)) \max \big(  \xi_{p^*} ,   \xi_{q} \big )n^{ \mathfrak{p}(p^*,q)  }&\leq \e \| G_n \|_{p \to q}  \le(1+ o(1)) \big(  \xi_{p^*} +   \xi_{q}   \big)n^{\mathfrak{p}(p^*,q)   }.
		\end{align*} 
	\end{itemize}
\end{proposition}

In the complement of the above cases of $(p,q)$, one of the terms $w(B_p^n)r(B_{q^*}^n)$ and $w(B_{q^*}^n)r(B_p^n)$ strictly dominates the other as $n$ tends to infinity and as a result, we can locate the limiting constants of the scaled $\e \| G_n\|_{p\to q}$.

\begin{proposition}\label{prop:operator_norm}
	The following limits are valid.
	\begin{itemize}	
		\item[$(i)$] For $p=1$ and $1\leq q<\infty,$ 
		$
		\lim_{n\to\infty}{n^{-1/q}}\e \| G_n \|_{p \to q} =\xi_q.$
		\item[$(ii)$] For $1<p\leq \infty$ and $q=\infty,$
		$
		\lim_{n\to\infty}n^{-1/p^*}\e \| G_n \|_{p \to q}=\xi_{p^*}.
		$
		\item[$(iii)$] For $1<p\leq \infty$ and $1\leq q<\infty,$ if $\mathfrak{p}(q,p^*)\ne \mathfrak{p}(p^*,q)$, then
		\begin{equation*}
		\lim_{n\to\infty}	n^{ - \max({\mathfrak{p}(q,p^*)   }  ,{\mathfrak{p}(p^*,q)   }  )}   \e \| G_n \|_{p \to q}=  \left\{
		\begin{array}{ll}
		\xi_{p^*},  &  \mbox{if $\mathfrak{p}(q,p^*)     >  \mathfrak{p}(p^*,q)$}  ,\\
		\xi_q, &   \mbox{if $\mathfrak{p}(q,p^*)   <  \mathfrak{p}(p^*,q)$} .
		\end{array}\right.
		\end{equation*}
	\end{itemize}
	
\end{proposition}

\begin{remark}\rm \label{rmk1}
	Note that items $(i)$ and $(ii)$ in Proposition \ref{prop:operator_norm} are equivalent due to the duality relation $\|G_n\|_{p, q}  = \|G_n^T\|_{q^*, p^*}$. 
	Moreover, if $1<p<\infty$ and $q=p^*$, then $\mathfrak{p}(p^*,q)=\mathfrak{p}(q,p^*)$, which is equal to $1/p^*$ if $1<p<2$ and $3/2-2/p$ if $2\leq p<\infty.$ Therefore, by Proposition \ref{prop1}, we have that
	\begin{align}\label{extra:eq9}
	\begin{array}{cl}
	c\sqrt{\log n}\leq \e\|G_n\|_{p\to p^*}\leq C\sqrt{\log n},&\mbox{if $p=1$},\\
	cn^{\frac{1}{p^*}}\leq  \e \| G_n \|_{p \to p^*} \leq Cn^{\frac{1}{p^*}},&\mbox{if $1< p<2$},\\
	cn^{\frac{3}{2}-\frac{2}{p}}\leq  \e \| G_n \|_{p \to p^*}  \leq Cn^{\frac{3}{2}-\frac{2}{p}},&\mbox{if $2\leq p< \infty,$}
	\end{array}
	\end{align}
	where $c,C>0$ are universal constants depending only on $p.$
\end{remark}

\subsection{Scalings of the Grothendieck problem}\label{sec:SGP}

The scalings of the $\ell_p$-Grothendieck problem for $1\leq p<\infty$ are obtained in the following proposition.

\begin{proposition}
	We have that
	\begin{align*}
	\begin{array}{cl}
	c\sqrt{\log n}\leq   \e \max_{\|x\|_p=1}\la G_nx,x\ra\leq C\sqrt{\log n},&\mbox{if  \ $p=1$},\\
	cn^{\frac{1}{p^*}}\leq   \e \max_{\|x\|_p=1}\la G_nx,x\ra\leq Cn^{\frac{1}{p^*}},&\mbox{if  \ $1< p<2$},\\
	cn^{\frac{3}{2}-\frac{2}{p}}\leq  \e \max_{\|x\|_p=1}\la G_nx,x\ra  \leq Cn^{\frac{3}{2}-\frac{2}{p}},&\mbox{if  \ $2\leq p<\infty,$}
	\end{array}
	\end{align*}
	where $c, C>0$ are universal constants depending only on $p.$
\end{proposition}

\begin{proof}
	First, by H\"older's inequality,
	\[ \sup_{\| x\|_p =1 }\langle G_nx, x\rangle  \le \sup_{\| x\|_p =1 }  \| G_n x\|_{p^*} \| x\|_{p}  =  \sup_{\| x\|_p =1 }  \| G_n x\|_{p^*}  = \|G_n\|_{p \to p^*}. \]
	From Remark \ref{rmk1}, the Grothendieck problem shares the same upper bounds as those in \eqref{extra:eq9}. The lower bound requires a bit of work. We argue in three cases:
	
	\smallskip
	
	\noindent{\bf Case $2\leq p<\infty$.} Let $w$ be the $\ell_2$-normalized eigenvector of  the GOE, $\bar G_n $, associated with the top eigenvalue $\lambda_1(\bar G_n)$. It is well-known that $w$ equals $g /\|g\|_2$ in distribution for an $n$-dimensional standard Gaussian vector $g$, leading to  $n^{1/2-1/p}\| w\|_p   \to \xi_p/\xi_2$ almost surely and in $L^2$. Also, $\lambda_1(\bar G_n) /\sqrt{n}\to 2$ almost surely and  in $L^2$.
	Therefore, we can take $x=w/\|w\|_p$ to obtain that 
	\[ \e  \sup_{\| x\|_p =1 }\langle G_n x, x\rangle  = \frac{1}{\sqrt{2}}  \e \sup_{\| x\|_p =1 }\langle \bar G_n x, x\rangle \ge   \frac{1}{\sqrt{2}}  \e \| w\|_p^{-2} \lambda_1( \bar G_n ) \ge c n^{ \tfrac{3}{2} - \tfrac{2}{p} },\]
	for some constant $c>0$. 
	
	\smallskip
	
	{\noindent \bf Case $1<p<2$.} Without loss of generality, assume that $n$ is even.  We readily bound that
	\[ \e \sup_{\| x\|_p =1 }\langle G_n x, x\rangle   
	\ge \sup_{\| v\|_p = 2^{-1/p} }  \e  \sup_{\| u\|_p =  2^{-1/p} }\Big\langle  G_n  \begin{pmatrix} u  \\ v  \end{pmatrix}, \begin{pmatrix} u  \\ v  \end{pmatrix}\Big\rangle ,   \]
	where $u,v\in \mathbb{R}^{n/2}.$
	Divide $G_n$ into four smaller blocks of sizes $(n/2)\times(n/2),$
	\[ G_n  = \begin{bmatrix}  B_{11} &   B_{12} \\  B_{21} & B_{22} \end{bmatrix}.\]
	Take $v=(2^{-1/p},0,\ldots,0)^T$. Clearly $\e \langle B_{22} v, v\rangle = 0$ and thus,
	\begin{align*}
	\e \sup_{\| u\|_p =  2^{-1/p} }\Big\langle G_n  \begin{pmatrix} u  \\ v  \end{pmatrix}, \begin{pmatrix} u  \\ v  \end{pmatrix}\Big\rangle&=  \e\Bigl[ \sup_{\| u\|_p =  2^{-1/p} } \bigl(\langle B_{11} u, u\rangle +  \langle  (B_{12}+B_{21}^T) v, u\rangle\bigr)\Bigr] \\
	&=  \e\Bigl[ \sup_{\| u\|_p =  2^{-1/p} } \bigl(\langle B_{11} u, u\rangle +  2^{1/2-1/p}\langle  \tilde g, u\rangle\bigr)\Bigr], 
	\end{align*} 
	where $\tilde g$ is a standard Gaussian vector in $\r^{n/2}$.
	Given $\tilde g$, we can choose $u^*  = u^* (\tilde g) $  with $\| u^*\|_p =  2^{-1/p}$  such that   $ \langle  \tilde g, u^*\rangle =  2^{-1/p} \| \tilde g \|_{p^*}$, which implies that there exists some $c'>0$ such that
	\begin{align*}
	\e\Bigl[ \sup_{\| u\|_p =  2^{-1/p} } \bigl(\langle B_{11} u, u\rangle +  2^{1/2-1/p}\langle  \tilde g, u\rangle\bigr)\Bigr] &\ge  \e\bigl[ \langle B_{11} u^*, u^*\rangle + 2^{1/2-1/p} \| \tilde g \|_{p^*}\bigr]\\
	& =  2^{1/2-1/p}\e  \| \tilde g \|_{p^*} \ge c' n^{1/p*},
	\end{align*} 
	where the equality above follows from the fact that $\e[ \langle B_{11} u^*, u^*\rangle | u^*] = 0$ since $u^*$ is independent of $B_{11}$ whereas  the last inequality uses \eqref{eq:Gaussian_p_norm}. Hence, we obtain that $\e\sup_{\| x\|_p =1 }\langle G_nx, x\rangle  \ge c' n^{1/p*}.$ 
	
	\smallskip
	
	{\noindent \bf Case $p=1$.} A matching lower bound can  be easily obtained by optimizing only over the coordinate vectors $e_i$ yielding  $\e \max_{\|x\|_1=1}\la G_nx,x\ra \ge \e \max_{i \in [n]} g_{ii} = \sqrt{2\log n}(1+o(1)).$
\end{proof}

\subsection{Concentration and maximal inequalities}

We collect some Gaussian bounds that will be of great use throughout the rest of this paper. We start with the well-known concentration inequality for the Lipschitz function of a standard Gaussian vector states (see, e.g., \cite[Theorems~5.5 and 5.6]{BLM13}).

\begin{lemma}[Gaussian concentration] \label{lem:Gconc}
	Let $g$ be a standard Gaussian vector in $\r^n$. Let $F: \r^n \to \r$ be a $L$-Lipschitz function, i.e., $|F(x) - F(y)| \le L \| x - y\|_2$ for all $x, y \in \r^n$. Then 
	\begin{equation}\label{eq:subG}
	\e e^{ s( F(g) - \e F(g)  )} \le e^{\frac{s^2 L^2}{2}},\,\,\forall s\in \mathbb{R}
	\end{equation}
	and
	\begin{equation}\label{eq:Gconc}
	\p(| F(g) - \e F(g) | \ge s )  \le 2 e^{ -\frac{s^2}{2L^2}},\,\,\forall s>0.
	\end{equation}
\end{lemma}

\begin{remark}\rm Let $1 \leq p \leq2$. From Lemma \ref{lem:Gconc}, for any nonempty $S, T \subseteq B_p^n$, and any $s>0$, 
	\begin{align}
	\begin{split}\label{concen1}
	\p\Bigl(\Bigl| \sup_{ x \in S} \la x, G_n x \ra - \e \sup_{ x \in S} \la x, G_n x \ra \Bigr | \ge s \Bigr)  &\le 2 e^{ -s^2/2},
	\end{split}\\
	\begin{split}\label{concen2}
	\p\Bigl(\Bigl| \sup_{ x \in S, y \in T} \la x, G_n y \ra - \e \sup_{ x \in S, y \in T} \la x, G_n y \ra \Bigr | \ge s \Bigr)  &\le 2 e^{ -s^2/2}.  
	\end{split}
	\end{align}  
	To see this, note that the functions $G_n\mapsto\la x, G_n x \ra$ and $G_n \mapsto  \sup_{ x \in S, y \in T} \la x, G_n y \ra $ are Lipschitz continuous with respect to $\|\cdot\|_F$, the Frobenius norm. The former case has the Lipschitz constant, $\sup_{x\in S}\|x\|_2$, whereas  the latter case has the Lipschitz constant $  \sup_{ x \in S, y \in T } \| x\|_2 \| y\|_2 $. Both of these Lipschitz constants are bounded above by $1$ by the assumption that $p\leq 2$. 
\end{remark}

\begin{remark}\rm
	Let $2<p<\infty.$ It can be shown that $\gp_{n,p}$ and $L_{n,p,u}(t)$ are $n^{-1/2}$ and $un^{-1/2}$-Lipschitz with respect to $\|\cdot\|_F$, respectively. From Lemma \ref{lem:Gconc},
	\begin{align}
	\begin{split}\label{extra:eq18}
	\p\bigl(\bigl|\gp_{n,p}-\e \gp_{n,p}\bigr|\geq s\bigr)&\leq 2e^{-\frac{ns^2}{2}},\,\,\forall s>0,
	\end{split}\\
	\begin{split}\label{extra:eq12}
	\p\bigl(\bigl|L_{n,p,u}(t)-\e L_{n,p,u}(t)\bigr|\geq s\bigr)&\leq 2e^{-\frac{ns^2}{2u^2}},\,\,\forall s>0.
	\end{split}
	\end{align}
\end{remark}

\begin{remark}\rm We also mention the following concentration  bound for the $\ell_\infty$-norm of a standard Gaussian vector $g$ in $\r^n$ (see \cite[Corollary~4.10]{paouris2017} and \cite{boucheron2012,schechtman2004}), which is an improvement upon the one obtained from the preceding Gaussian concentration inequality.  There exist absolute constants $c, C>0$ such that for all $n \ge 1$ and for all $\eps \in (0, 1)$, 
	\begin{equation}\label{eq:Gaussian_maxima_upper_tail}
	\p(\big | \| g \|_\infty  - \e \| g \|_\infty \big | \ge \eps  \e \| g \|_\infty  ) \le C n^{-c\eps}. 
	\end{equation}
\end{remark}

Recall that a mean zero random variable $X$ is called $\sigma^2$-subgaussian  if $ \e e^{ t X} \le e^{t^2 \sigma^2/2}$ holds for all $t \in \r$.
The following result gives a crude bound of the maxima of a collection of subgaussian random variables (see \cite[Lemma~5.1]{van14}). 
\begin{lemma}[Maximal inequality] \label{lem:maxi_subG}
	Let $X_1, X_2, \ldots, X_m$ be a collection (not necessarily independent) of centered $\sigma^2$-subgaussian random variables. Then 
	\[ \e \max_{1\leq i\leq m} X_i \le \sqrt{ 2 \sigma^2 \log m}.\]
\end{lemma}

\section{Proof of Theorem~\ref{prop:p=1}}\label{sec4}
In this section, we provide a proof of Theorem~\ref{prop:p=1}. We begin with part $(i)$.

\begin{proof}[\bf Proof of $(i)$.] In view of the identity $\la G_nx,x\ra =2^{-1/2} \la \bar G_nx,x\ra$, it suffices to prove that 
\[  \frac{1}{\sqrt{\log n}} \max_{\|x\|_1=1}\la \bar G_nx,x\ra \to 2, \quad \text{almost surely}.\]
Note that the diagonal entries of the symmetric matrix $\bar G_n$ are $\bar g_{ii}=\sqrt{2} g_{ii}$ whereas  the off-diagonal entries are equal to $\bar g_{ij} = (g_{ij} + g_{ji})/\sqrt{2}.$ 
We bound
\begin{align*}
\la \bar G_nx,x\ra  = \sqrt 2 \sum_{i} g_{ii} x_i^2 + 2 \sum_{i < j} \bar g_{ij} x_i x_j &\le \max \big( \sqrt 2 \max_{i} g_{ii},\ \max_{i < j} | \bar g_{ij}|  \big)  \Big( \sum_i x_i^2 + 2 \sum_{i < j}  |x_i| |x_j| \Big)\\
&= \max \big ( \sqrt 2 \max_{i} g_{ii},\ \max_{i < j} |\bar g_{ij}|  \big)  \| x\|_1^2, 
\end{align*}
which leads to that
\begin{equation}\label{eq:p=1:upper_bound}
\frac{1}{ \sqrt{\log n} } \max_{\|x\|_1=1}\la \bar G_nx,x\ra \le \frac{\max \big ( \sqrt 2 \max_{i} g_{ii},\ \max_{i < j} |\bar g_{ij}|  \big)}{\sqrt{\log n} }.  \end{equation}
It is a well-known fact (see, e.g., \cite[Example 3.5.4]{embrechts2013}) that if $z_1, z_2, \ldots $ is a sequence of  i.i.d.\ $ N(0,1)$, then $ (2 \log n)^{-1/2} \max_{i \in [n]} z_i  \to 1$ almost surely. Consequently, the following limits exist almost surely. 
\[  \lim_n \frac{\sqrt 2 \max_{i} g_{ii}}{ 2 \sqrt{ \log n} }   = 1 \,\,
\text{ and }  \,\,
\lim_{n\to\infty}   \frac{ \max_{i < j} |\bar g_{ij}| }{ 2 \sqrt{ \log n} }  =  \lim_{n\to\infty} \max \Bigl( \frac{ \max_{i < j} \bar g_{ij} }{ \sqrt{2\log { n \choose 2}}},  \frac{ \max_{i < j} (- \bar g_{ij}) }{ \sqrt{2\log { n \choose 2}}} \Bigr)  =1. \]
Plugging the above limits in \eqref{eq:p=1:upper_bound}, we have the upper bound
\[ \limsup_{n\to\infty} \frac{1}{ \sqrt{\log n} } \max_{\|x\|_1=1}\la \bar G_nx,x\ra \le 2. \]
For the lower bound, we optimize over just  the coordinate vectors $e_i, i \in [n]$ to obtain
\[ \liminf_{n\to\infty}\frac{1}{ \sqrt{\log n} } \max_{\|x\|_1=1}\la \bar G_nx,x\ra \ge  \liminf_{n\to\infty}\frac{1}{ \sqrt{\log n} } \max_i \sqrt{2} g_{ii}  = 2.\]
This completes the proof of \eqref{eq:Gnp_limit_p=1}.
\end{proof}

\begin{proof}[\bf Proof of $(ii)$.] We realize that $\max_{\| x\|_1= 1} \la \bar G_nx,x\ra  \ge \max_{i \in [n]}  \sqrt{2} g_{ii}$ and  then proceed to prove \eqref{eq:lb_p=1} as 
\begin{align*}
\p \big (  \max_{i \in [n]}  g_{ii} \le   \sqrt{ 2(1-\eta) \log n} \big ) &= \p \big ( g_{11} \le   \sqrt{ 2(1-\eta) \log n} \big )^n\\
&\le  \big ( 1 - e^{-(1-\eta) \log n}  \big )^n=  \big ( 1 - n^{-(1-\eta)}  \big )^n \le e^{- n^\eta},
\end{align*}
where the first inequality above follows from the Gaussian tail bound $\p(z > t) \le e^{-t^2/2}$ for all $t>0$ and the last inequality uses the bound $1 - t \le e^{-t}$ for all $t \in \r$.
\end{proof}

\begin{proof}[\bf Proof of $(iii)$] 
	For $x$ with $\|x\|_1=1$ and $\|x\|_\infty \le 1 - \delta$, we have 
	\[ | (\bar G_n x)_i | \leq \sum_{j\in [n]} |\bar g_{ij}||x_j|\leq (1-\delta)X_{i, (1)} + \delta X_{i, (2)}, \quad i \in [n], \]
	where $X_{i, (1)} $ and $X_{i, (2)} $ denote the largest and the second largest values among the absolute values of the entries of the $i$-th row of $\bar G_n$. Then 
	\begin{align} \label{eq:p=1stability}
	\max_{\|x\|_1=1, \|x\|_\infty \le 1 - \delta }\la \bar G_nx,x\ra  
	\le \max_{i \in [n]}  \big ( (1-\delta)X_{i, (1)} + \delta X_{i, (2)} \big)\leq (1-\delta)\max_{i \in [n]} X_{i, (1)}+\delta\max_{i\in [n]}X_{i, (2)} .
	\end{align}
	Therefore, we need to bound the maxima, $\max_{i \in [n]} X_{i, (1)}$ and $\max_{i \in [n]} X_{i, (2)}.$ Note that
	\[ \max_{i \in [n]} X_{i, (1)} = \max \big ( \sqrt 2 \max_{i} |g_{ii}|,\ \max_{i < j} |\bar g_{ij}|  \big). \]
	Here, from Lemma \ref{lem:maxi_subG},
	\[ \e \max_{i} |g_{ii}| =\e \max\bigl(g_{11},-g_{11},\ldots, g_{nn},-g_{nn}\bigr)\leq  \sqrt{2\log 2n}.   \]
	Therefore, \eqref{eq:Gaussian_maxima_upper_tail} implies that for $\delta \in (0, 1),$ 
	\[  \p( \max_{i} |g_{ii}| \ge (1+ \delta/22) \sqrt{ 2 \log (2n)} ) \le C' n^{-c'\delta},\]
	where $c', C'$ are positive universal constants. We  reapply the above tail bound to the maximum of ${n \choose 2}$ i.i.d.\ standard Gaussian random variables $( \bar g_{ij})_{1 \le i < j \le n}$ and then use the union 
	bound to deduce that 
	\begin{align*}
	\p \big ( \max_{i \in [n]} X_{i, (1)}  \ge 2(1+ \delta/22) \sqrt{ \log (2n)} \big) \le C''n^{-c''\delta}
	\end{align*}
	for some positive constants $c''$ and  $C''$. 
	Let $K = {10 \sqrt{2} }/{11}\in (\sqrt{3/2}, \sqrt{2})$. From the usual Gaussian  tail bound, denote  
	\[ \theta_n: = \p( \sqrt{2} |z| \ge K \sqrt{ 2 \log n} ) =O(n^{-K^2/2}), \quad  \phi_n :=\p( |z| \ge K \sqrt{ 2 \log n} ) =O( n^{-K^2}). \]
	Note that $X_{i,(2)}\geq K\sqrt{2\log n}$ implies that at least two of $(\bar g_{ij})_{1\leq j\leq n}$ are not less than $K\sqrt{2\log n}.$ By using the union bound,
	\begin{align*} 
	\p( X_{i, (2) } \ge K \sqrt{ 2 \log n} ) &\le \p(\sqrt{2}|g_{ii}|\geq K\sqrt{2\log n}\,\,\mbox{and}\,\,|\bar g_{ij}|\geq K\sqrt{2\log n}\,\,\mbox{for some $j\neq i$})\\
	&+ \p(|\bar g_{ij}|,|\bar g_{ij'}|\geq K\sqrt{2\log n}\,\,\mbox{for some distinct $j,j'\neq i$})\\
	&\leq (n -1) \theta_n \phi_n  + {n-1 \choose 2} \phi_n^2 = O( n^{- ( 3K^2/2  -1)}) + O( n^{- ( 2K^2  -2)}),
	\end{align*}
	which, by noting that our choice of $K$ implies $3K^2/2-2>2K^2-3$ and using the union bound again, implies that 
	\begin{equation*}
	\p( \max_{i \in [n]} X_{i, (2) } \ge K \sqrt{ 2 \log n} )  = O( n^{- ( 3K^2/2  -2)}) + O( n^{- ( 2K^2  -3)}) = O( n^{- ( 2K^2  -3)}).
	\end{equation*}
	Note that the exponent $ 2K^2  -3$ is positive. 
	Therefore, for all $0<\delta\leq \min( (2K^2-3)/c'',1),$ on the event where both the bounds  $\max_{i \in [n]} X_{i, (1)}  \le 2(1+ \delta/22) \sqrt{ \log (2n)}$ and $ \max_{i \in [n]} X_{i, (2) } \le K \sqrt{ 2 \log n}$ are satisfied, which occurs with probability at least $1 - O( n^{- ( 2K^2  -3)}) -  O( n^{-c''\delta}) = 1 - O( n^{-c''\delta}),$ 
	we have, by a direction of computation, 
	\begin{align*}
	(1- \delta) \max_{i \in [n]}  X_{i, (1)} + \delta  \max_{i \in [n]} X_{i, (2)}
	&\leq 2(1+ \delta/22) \sqrt{ \log (2n)}(1-\delta)+K \delta \sqrt{ 2 \log n}\\
	&= \sqrt{2}\big(  1   - \delta/22   \big) \sqrt{2 \log (2n)} \leq \sqrt{2}\Big(  1   - \delta/22   +\sqrt{\log 2/\log n}\Big)\sqrt{2\log n},
	\end{align*}
	where the last inequality used $\sqrt{a+b}\leq \sqrt{a}+\sqrt{b}$ for all $a,b\geq 0.$
	Finally, this inequality and \eqref{eq:p=1stability}, coupled with the fact that $\la G_nx,x\ra = 2^{-1/2} \la \bar G_nx,x\ra$, yield \eqref{eq:stability_p=1}. 
\end{proof}

\section{Proof of Theorem \ref{eq:thm_p<2}}\label{sec5}
In order to prove Theorem~\ref{eq:thm_p<2}, we first begin with some preparation in the following subsection. 
\subsection{Preliminary bounds}

For a given truncation level $\eps>0$, set
\[D = B_p^n \cap \{ x \in \r^n:  \| x\|_\infty \le \eps \} \ \ \ \ \ \text{ and } \ \ \ \ \ L =B_p^n \cap \{ x \in \r^n:  |x_i|\neq 0 \Rightarrow  |x_i|  > \eps \text{ for each } i \}. \]
For $x \in \r^n$, we can uniquely decompose $x = x^o + x^\dagger$ with $x^o \in D$ and $x^\dagger \in L$ by setting
\begin{equation}\label{eq:decompose1}
x^o_i  = x_i 1( |x_i| \le \eps) \ \  \text{ and } \ \  x^\dagger_i  = x_i 1( |x_i| > \eps) \ \ \text{ for each } i. 
\end{equation}
Note that $x^\dagger$ carries the localized part of $x$, whereas $x^o$ is the remaining delocalized part of $x$.
With these notations in hand, we are ready to state two crucial  propositions.

\begin{proposition}\label{lem:decomp}
	Let $1<p<2$. There exists a constant $C>0$ such that the following statements are valid. 
	\begin{itemize}
		\item[$(i)$] For any $\varepsilon>0$ and $n\geq 1,$
		\begin{equation} \label{eq:cross_term_dominate}
		\e\sup_{ S \subseteq B_p^n} \Big |  \max_{x \in S }\la G_nx,x\ra  -  \sqrt{2}   \max_{x \in S}\la \bar G_n x^o, x^\dagger \ra \Big |  \le C \bigl( \eps^{1 - p/2} n^{1/p*} + \eps^{-p/p^*} \sqrt{\log n} \bigr).
		\end{equation}
		\item[$(ii)$] For any $\eps>0,$ $c>0,$ and $n\geq 1,$ with probability at least $1-2n^{-c}$,
		\begin{equation} \label{eq:cross_term_dominate2}
		\sup_{ S \subseteq B_p^n} \Big |  \max_{x \in S }\la G_nx,x\ra  -  \sqrt{2}   \max_{x \in S}\la \bar G_n x^o, x^\dagger \ra \Big |  \le C\bigl( \eps^{1 - p/2} n^{1/p*} + \eps^{-p/p^*} \sqrt{\log n}\bigr) +2\sqrt{2c\log n}.
		\end{equation}
	\end{itemize}
\end{proposition}
%
Since the $\ell_p$-Grothedieck problem scales as $n^{1/p^*}$  for $1<p<2$,  the above proposition says that the main contribution of $\ell_p$-Grothendieck problem on an arbitrary subset $S$ of $B_p^n$  comes from the cross-term $\sqrt{2}   \max_{x \in S}\la \bar G_n x^o, x^\dagger \ra$ as long as $\varepsilon=\varepsilon_n$ satisfies that $\limsup_{n\to\infty}\varepsilon=0$ and  $\liminf_{n\to\infty}\varepsilon n^{1/p}/(\log n)^{p^*/(2p)}=\infty$.  The next result then provides an upper bound on this cross-term in terms of the maximum number of localized coordinates of the vectors in $S$. For $x \in \r^n$, recall that $\|x\|_0$ denotes the number of nonzero entries of $x$.

\begin{proposition} \label{lem:max_over_T} 
	There exists a constant $C$ such that
	for any $1 \le s \le n$ and any subset $L^\# \subseteq \{ x: \|x\|_p \le 1,  \|x\|_0 \le s\} $, 
	\[  \e \sup_{ x \in L^\#, y \in B_p^n} \la  y,  \bar G_n x \ra  \le r(L^\#)  \e\| g\|_{p^*}  + C \bigl(n^{p/(2p^*)}  \sqrt{\log n}   + s^{1/p^*} \sqrt{\log n}  \bigr),  \]
	where $g$ is an $n$-dimensional standard Gaussian vector.
\end{proposition}

\begin{remark}\rm 
	The bound in  Proposition~\ref{lem:max_over_T} will mainly be used in the proof of \eqref{eq:Gnp_limit_p<2} in Theorem \ref{eq:thm_p<2}. The key feature here is that the term $r(L^\#)  \e\| g\|_{p^*}$ is asymptotically sharp up to an additive error of smaller order for some properly chosen $s = o(n)$ and $L^\#$, yielding the correct limit of the $\ell_p$-Grothendieck problem. If we replace $\bar G_n$ with $G_n$, then we could also directly apply Chevet's inequality to obtain 
	\[  \e \sup_{ x \in L^\#, y \in B_p^n} \la  y,  G_n x \ra =\e \sup_{ x \in L^\#, y \in B_p^n} \la  y,  G_n^T x \ra\le r(L^\#) \e\| g\|_{p^*}   + C  s^{1/p^*} \sqrt{\log n},  \]
	which implies that, by recalling $\bar G_n=(G_n+G_n^T)/\sqrt{2}$,
		\[  \e \sup_{ x \in L^\#, y \in B_p^n} \la  y,  \bar G_n x \ra \le \sqrt{2} r(L^\#) \e\| g\|_{p^*}   + C  s^{1/p^*} \sqrt{\log n}.  \]
		However, this extra prefactor $\sqrt{2}$ in front of  $r(L^\#) \e\| g\|_{p^*}$ makes the bound sub-optimal.  
		\end{remark}

We establish the above two propositions in  the rest of this subsection.  The next  lemma bounds the Gaussian width of sparse sets.

\begin{lemma}\label{lem:Gaussian_width_sparse}
	Fix $p>1$ and $s \le n$. Let $W = \{ x\in \r^n:  \| x\|_p \le 1, \| x \|_0 \le s\}$. Then 
	\[w(W) \le \sqrt{2} s^{1/p^*} \sqrt{\log n}. \]
\end{lemma}
\begin{proof}
	Observe that  for any subset $I$ of $\{1, 2, \ldots, n\}$, 
	\[ \max_{ \| x\|_p \le 1,\mathrm{supp}(x) = I} \la g, x \ra  = \Bigl(\sum_{i \in I} |g_i|^{p^*} \Bigr)^{1/p^*}.\]
	Therefore,  by Lemma \ref{lem:maxi_subG}, we have
	\begin{align*}
	w(W) &= \e  \sup_{ x \in W} \la g, x \ra    \le  \e \max_{  |I |  \le  s  } \Big (  \sum_{i \in I} |g_i|^{p^*} \Big )  ^{1/p^*}\\
	&\le s^{1/p^*} \e \max_i |g_i|    \le  s^{1/p^*} \sqrt{2 \log n}.
	\end{align*}
\end{proof}
\begin{proof}[\bf Proof of Proposition~\ref{lem:decomp}]
	Let $\eps>0$ be fixed. Write, according to \eqref{eq:decompose1},
	\[ \la G_nx,x\ra   =   \sqrt{2} \la \bar G_n x^o, x^\dagger \ra + \la G_n x^o, x^o \ra + \la G_n x^\dagger, x^\dagger \ra,\]
	and then, use the triangle inequality to bound
	\begin{align}\label{eq:decomp1}
	\sup_{ S \subseteq B_p^n} \Big |  \max_{x \in S }\la G_nx,x\ra  -  \sqrt{2}   \max_{x \in S}\la \bar G_n x^o, x^\dagger \ra \Big |\nonumber
	&\le   \max_{\|x\|_p=1} | \la G_n x^o, x^o \ra | + \max_{\|x\|_p=1}  | \la G_n x^\dagger, x^\dagger \ra|\nonumber\\
	&\leq \max_{x\in D} | \la G_n x, x\ra | + \max_{x\in L}  | \la G_n x, x \ra|.
	\end{align}
	Note that as $D$ and $L$ are symmetric, i.e., $D=-D$ and $L=-L$, we have 
	\begin{align*}
	\max_{x\in D} | \la G_n x, x\ra | \leq \max_{x,y\in D}\la G_nx,y\ra,\quad
	\max_{x\in L} | \la G_n x, x\ra | \leq \max_{x,y\in L}\la G_nx,y\ra.
	\end{align*}
	Applying Chevet's inequality with $S=T=D$, we have 
	$ \e \max_{x,y \in D} \la G_n x, y \ra \le 2 w(D) r(D).$
	We can bound the Guassian width of $D$ by that of $B_p^n$ to obtain 
	$w(D) \le w(B_p^n) = \e \| g \|_{p^*} \le C n^{1/p^*},$
	using \eqref{eq:Gaussian_p_norm}. On the other hand, for $x \in D$, we have
	$\sum_{i} |x_i|^2 \le \eps^{2 - p} \sum_{i} |x_i|^p \le \eps^{2 - p}$. 
	Hence, 
	$ r(D) \le \eps^{1 - p/2}$
	and  therefore, 
	\begin{equation}\label{eq:Sbound}
	\e \max_{x \in D} |\la G_n x, x \ra|\le 2C \eps^{1 - p/2} n^{1/p*}.
	\end{equation}
	Similarly, we have  $\e \max_{x,y \in L} \la G_n x, y \ra \le 2 w(L) r(L).$
	Obviously $r(L)=1$.
	Note that if $x\in L$, then $\| x \|_0 \le \eps^{-p}$.
	Hence, by Lemma~\ref{lem:Gaussian_width_sparse}, we have 
	$w(L) \le \sqrt{2}  \eps^{-p/p^*} \sqrt{\log n}$ and this yields
	\begin{equation}\label{eq:Tbound}
	\e \max_{x\in L}  |\la G_n x, x \ra| \le 2 \sqrt{2} \eps^{-p/p*} \sqrt{\log n}.
	\end{equation}
	Now \eqref{eq:cross_term_dominate} follows from \eqref{eq:decomp1}, \eqref{eq:Sbound}, and \eqref{eq:Tbound} after taking $C$ sufficiently large.
	To prove \eqref{eq:cross_term_dominate2}, note that $G_n\mapsto\max_{x\in D}  |\la G_n x, x \ra|$ and $G_n\mapsto\max_{x\in L}  |\la G_n x, x \ra|$ are $1$-Lipschitz with respect to the Frobenius norm. For any $c>0$, using \eqref{eq:Gconc} with $s=\sqrt{2c\log n}$, \eqref{eq:Sbound}, and \eqref{eq:Tbound} yields that  with probability at least $1-2n^{-c}$,
	\begin{align*}
	\max_{x\in D}  |\la G_n x, x \ra|+  \max_{x\in L}  |\la G_n x, x \ra|
	&\le 2C \eps^{1 - p/2} n^{1/p*} \sqrt{\log n}+2 \sqrt{2} \eps^{-p/p*}\sqrt{\log n}+2\sqrt{2c\log n}.
	\end{align*}
	Consequently,  it follows from \eqref{eq:decomp1} that \eqref{eq:cross_term_dominate2} is valid with probability at least $1-2n^{-c}.$
	
\end{proof}


\begin{proof}[\bf Proof of Proposition \ref{lem:max_over_T}]
	 For any $x \in L^\#$, write $x=x^o+x^\dagger$ as in \eqref{eq:decompose1} by setting $\eps = n^{-1/p^*}$. Then $x^o$ satisfies that $\|x^o\|_p\leq 1,\|x^o\|_\infty\leq \eps,$ and $\|x^o\|_0\leq s.$ On the other hand, since any non-zero entry of $x^\dagger$ satisfies $|x_i^\dagger|\geq \eps$, it follows that $\|x^\dagger\|_0\eps^p\leq \|x\|_p^p\leq 1,$ which implies that $\|x^\dagger\|_0\leq \eps^{-p}.$ So, we have 
	\[ L^\# \subseteq U+V := \{ x + x': x \in U, x' \in V\},\]
	where 
	\[ U=  \{ x: \| x \|_p \le 1,    \|x \|_\infty \le \eps,  \| x \| _0 \le s \},  \ \ \ V =  L^\# \cap \{ x : \| x \|_0 \le \eps^{-p} \}. \]
	Then by triangle inequality, 
	\[  \e \sup_{ x \in L^\#, y \in B_p^n} \la  y,  \bar G_n x \ra \le \e \sup_{ x \in U, y \in B_p^n} \la  y,  \bar G_n x \ra + \e \sup_{ x \in V, y \in B_p^n} \la  y,  \bar G_n x \ra =  (I) + (II).\]
	
	\n {\bf Bounding (I):}  Since $G_n\stackrel{d}{=}G_n^T$, we have $ (I) \le \sqrt{2}  \e \sup_{ x \in U, y \in B_p^n}  \la  y,   G_n x \ra. $
	By Chevet's inequality,
	\[ \e \sup_{ x \in U, y \in B_p^n}  \la  y,   G_n x \ra \le r ( U) w(B_p^n) +  r ( B_p^n) w(U).\]
	Arguing similarly as we did in the proof of Proposition \ref{lem:decomp}, $r (U)  \le \eps^{1- p/2}$ and $ r ( B_p^n) = 1$. Moreover, we have $w(B_p^n) = \e \| g \|_{p^*} \le C_1 n^{1/p^*}$ and $w(U) \le \sqrt{2} s^{1/p^*} \sqrt{\log n}$
	(by Lemma~\ref{lem:Gaussian_width_sparse}). Therefore, we obtain
	\begin{equation*}
	(I) \le  \sqrt{2}C_1 \eps^{1-p/2} n^{1/p^*} + 2 s^{1/p^*} \sqrt{\log n} = \sqrt{2}C_1  n^{p/(2p^*)} + 2 s^{1/p^*} \sqrt{\log n}.
	\end{equation*}
	
	\n {\bf Bounding (II):}  Note that $(II) =\e  \sup_{x \in V}  \|  \bar G_n x \|_{p^*}.$
	Let $\delta=n^{-1}$. Since each $x \in V $ is  supported by  at most $\eps^{-p}$ coordinates, we can choose a $\delta$-net $N_\delta$ of $V$ with respect to $\| \cdot \|_p$-norm such that the cardinality of  $N_\delta$ is at most $O(  ( n \delta^{-1} )^{\eps^{-p}})$. We now  bound 
	\begin{align*}
	\e  \sup_{x \in V}  \|  \bar G_n x \|_{p^*} &\le \e  \max_{x \in N_\delta}  \|  \bar G_n x \|_{p^*} + \e \sup_{ x , x' \in B_p^n : \| x - x' \|_p \le \delta}  \|  \bar G_n (x-x') \|_{p^*} \\
	&\le \e  \max_{x \in N_\delta}  \|  \bar G_n x \|_{p^*} +  \delta \e   \| \bar G_n  \|_{p \to p^*} \le \e  \max_{x \in N_\delta}  \|  \bar G_n x \|_{p^*} + O(1),
	\end{align*}
	where in the last step, we used the bound $ \e   \| \bar G_n  \|_{p \to p^*} \le \sqrt{2} \e   \| G_n  \|_{p \to p^*} \le C_2 n^{1/p^*}$ from Remark \ref{rmk1}.
	
	Next, for each $x \in B_p^n$, the function $G_n \mapsto  \| \bar G_n x\|_{p^*}  =   \sup_{ y \in B_p^n } \langle y, \bar G_n  x \rangle $ is Lipschitz continuous with respect to the Frobenius norm  with Lipschitz constant $ \sqrt{ 2} \sup_{ y \in B_p^n } \| y\|_2 \| x\|_2 \le \sqrt{2} $. Hence, \eqref{eq:subG} yields 
	\[ \e  e^{ t ( \| \bar G_n x\|_{p^*} -\e \| \bar G_n x\|_{p^*} )    }  \le   e^{t^2},\,\,\forall x\in B_p^n,\,\,t\in \r. \]
	In other words,  the random variable $\|\bar G_n x\|_{p^*} -\e \| \bar G_n x\|_{p^*}$ is $2$-subgaussian. Therefore,  by Lemma~\ref{lem:maxi_subG}, 
	\begin{equation*}
	\e \max_{ x \in N_\delta} ( \| \bar G_n x\|_{p^*} -\e \| \bar G_n x\|_{p^*} )  \le \sqrt{ 4 \log |N_\delta|} = O( \eps^{-p/2} \sqrt { \log n} ) = O( n^{p/(2p^*)} \sqrt { \log n} ).
	\end{equation*}
	From the above estimates, we deduce that
	\begin{equation} \label{eq:fluc_V_bd}
	\e  \sup_{x \in V}  \|  \bar G_n x \|_{p^*} \le  \sup_{x \in V}  \e \| \bar G_n x\|_{p^*} +   O( n^{p/(2p^*)} \sqrt { \log n} ).
	\end{equation} 
	It remains to estimate $ \e \| \bar G_n x\|_{p^*}$. For any $x \in \r^n$,  
	we have  $ \bar G_n x \stackrel{d}{=} \|x\|_2 g +   z x$,
	where $g$, as always,  is the standard normal vector on $\r^n$ and $z$ is an independent $N(0,1)$ random variable.
	Therefore,  for any $x \in B_p^n$, 
	\begin{align*}
	\Big \vert \e  \|  \bar G_n x \|_{p^*}  -  \e \| g \|_{p^*} \| x\|_2 \Big \vert  \le \| x \|_{p^*} \e |z| \le 1.
	\end{align*} 
	Consequently,
	\begin{align} \label{eq:norm_bound_over_V}
	\sup_{x \in V}  \e \| \bar G_n x\|_{p^*}  &\le r(V) \e \| g\|_{p^*}  +     1 \le r(L^\#) \e \| g\|_{p^*}  +     1.
	\end{align}
	Combining \eqref{eq:fluc_V_bd} and \eqref{eq:norm_bound_over_V}, we obtain  that
	\[ (II) \le  r(L^\#) \e \| g\|_{p^*} +O( n^{p/(2p^*)} \sqrt { \log n} ), \]
	completing the proof of the proposition.
\end{proof}

\subsection{Proof of \eqref{eq:Gnp_limit_approx_opt}:  the lower bound} 

For convenience, throughout this proof, for three sequences of random variables $(X_n)_{n\geq 1}$, $(Y_n)_{n\geq 1}$, and $(Z_n)_{n\geq 1}$, we denote 
\begin{align}\label{extra:notation}
X_n=Y_n+O_\p(Z_n)
\end{align}  if for any $d>0$, there exists some $C>0$ such that $\p(X_n\leq Y_n+CZ_n)=1-O(n^{-d})$ for all $n\geq 1.$


The proof of  \eqref{eq:Gnp_limit_approx_opt} is argued as follows. Expanding
\[  \la v_i + e_i, G_n ( v_i+ e_i) \ra  = \la v_i, G_n v_i \ra+ \la e_i, G_n e_i \ra + \sqrt{2}  \la v_i, \bar G_n e_i \ra,  \]
we have
\begin{equation} \label{eq:O1}
\min_i \la v_i + e_i, G_n ( v_i+ e_i) \ra \ge \min_i \sqrt{2}  \la v_i, \bar G_n e_i \ra  -  \max_i |\la v_i, G_n v_i \ra| -  \max_i |\la e_i, G_n e_i \ra|.
\end{equation}
We will control each term on the right-hand side. First, we handle the middle one as follows. 
Since $\| \bar G_n e_i \|_{p^*} - \e \| \bar G_n e_i \|_{p^*} $ is $2$-subgaussian,
\begin{align}\label{eq:add1}
\p\bigl(  \max_i  \big | \| \bar G_n e_i \|_{p^*}  - \e  \| \bar G_n e_i \|_{p^*} \big|  > s \bigr) \le 2n e^{-s^2/4}.
\end{align} 
Write, by Jensen's inequality  and \eqref{eq:Gaussian_p_norm},
\begin{equation} \label{eq:O4}
\e  \| \bar G_n e_i \|_{p^*}  =   \e  \| \bar G_n e_1 \|_{p^*} \ge \e \| g\|_{p^*}=\xi_{p^*}n^{1/p^*}\bigl(1-a_n),
\end{equation}
 for some $a_n$  satisfying $0< a_n=O(n^{-1})$ and $g$ a standard Gaussian vector. Plugging \eqref{eq:O4} into \eqref{eq:add1} yields
\begin{align*}
\p\bigl(\min_i\| \bar G_n e_i \|_{p^*}\geq \xi_{p^*}n^{1/p^*}(1-a_n)+s\bigr)\leq 2ne^{-s^2/4}.
\end{align*}
Using the well-known  lower bound of Gaussian tail $\p(|z|\geq t)\geq 2t (t^2+1)^{-1}e^{-t^2/2}$ for $t>0,$ we have 
\begin{align*}
\p\bigl(\max_{i\leq j}|\bar G_n(i,j)|\leq s\bigr)&=\p(|z|\leq s)^{n(n-1)/2}\p(\sqrt{2}|z|\leq s)^n \geq \p(\sqrt{2}|z|\leq s)^{n^2}\geq \Bigl(1-\frac{2\sqrt{2}s}{s^2+2}e^{-s^2/4}\Bigr)^{n^2}.
\end{align*}
Combining the above bounds  and recalling \eqref{extra:eq8}, we have that for any $s>0$ and sufficiently large $n$, 
\begin{align*}
\p\Bigl(\max_i \| v_i \|_\infty \le \frac{2s^{p^*/p}}{\xi_{p^*}^{p^*/p}n^{1/p}}\Bigr)
&\geq \p\Bigl(\max_i \| v_i \|_\infty \le \frac{s^{p^*/p}}{\bigl(\xi_{p^*}n^{1/p_*}(1-a_n)+s\bigr)^{p^*/p}}\Bigr)\\
&\geq \Bigl(1-\frac{2\sqrt{2}s}{s^2+2}e^{-s^2/4}\Bigr)^{n^2}-2ne^{-s^2/4} \geq \Bigl(1-\frac{2\sqrt{2}}{s}e^{-s^2/4}\Bigr)^{n^2}-2ne^{-s^2/4}.
\end{align*}
As a result, for any $d>0$, if we let $s=\xi_{p^*}\sqrt{8(1+d)\log n}$, then for $\eps: =  2(8(1+d)\log n)^{p^*/(2p)}  n^{ - 1/p},$ 
\begin{align}
\begin{split}\notag
\p\bigl(\max_i \| v_i \|_\infty \le\varepsilon\bigr)&\geq \Bigl(1-\frac{2\sqrt{2}}{n^{2(1+d)\xi_{p^*}^2}\xi_{p^*}\sqrt{8(1+d)\log n}}\Bigr)^{n^2}-\frac{2}{n^{(1+d)\xi_{p^*}^2-1}}  \end{split}
\\
\begin{split}\label{extra:eq3}
&\geq 1-\frac{2\sqrt{2}}{n^{2(1+d)\xi_{p^*}^2-2}\xi_{p^*}\sqrt{8(1+d)\log n}}-\frac{2}{n^{2(1+d)\xi_{p^*}^2-1}}
\geq 1-\frac{C}{n^d}
\end{split}
\end{align}
for some universal constant $C>0$, where the second inequality used the fact that for $m\geq 1$, $(1-x)^m\geq 1-mx$ for any $0\leq x\leq 1$ and the third inequality is valid because $\xi_{p^*}\geq 1.$ 
Now, on the event $\max_i \| v_i \|_\infty \le\varepsilon$,
we can bound
$ \max_i | \la v_i, G_n v_i \ra|  \le  \sup_{ \| x\|_p \le 1, \| x\|_\infty \le \eps  } |\la x, G_n x \ra|. $
As argued in the proof of Proposition~\ref{lem:decomp}, 
\begin{align*} \e \sup_{ \| x\|_p \le 1, \| x\|_\infty \le \eps  } |\la x, G_n x  \ra|  &= O(  \eps^{1- \frac{p}{2}} n^{\frac{1}{p^*}})  \\
&=  O((\log n)^{(1-\frac{p}{2})\frac{p^*}{2p}}n^{-(1-\frac{p}{2})\frac{1}{p}+\frac{1}{p^*}})=O( n^{-\frac{1}{p} + \frac{1}{2} + \frac{1}{p^*}} ( \log n )^{p^*/4}),
\end{align*}
where, in handling the logarithmic term to validate the third equality, we used $(1-p/2)/p\leq 1/2.$
Consequently, from \eqref{extra:eq3} and the Gaussian concentration \eqref{concen1}, 
\begin{equation}\label{eq:O2}
\max_i | \la v_i, G_n v_i \ra|=O_\p(n^{-\frac{1}{p} + \frac{1}{2} + \frac{1}{p^*}} ( \log n )^{p^*/4}+\sqrt{\log n}).
\end{equation}

To control the third term of \eqref{eq:O1}, note that $ \max_i |\la e_i, G_n e_i \ra| \le \max_i |g_{ii}| $ and $\e \max_i|g_{ii}|\leq \sqrt{2\log n}$. Since $\max_i|g_{ii}|$ is $1$-Lipschitz with respect to the $\|\cdot\|_2$-norm, it follows that
\begin{equation} \label{eq:O3}
\max_i |\la e_i, G_n e_i \ra| =O_\p(\sqrt{\log n}).
\end{equation}
Finally, the first term of \eqref{eq:O1} can be handled as follows. Note that $\la v_i, \bar G_n e_i \ra = \| \bar G_n e_i \|_{p^*}$. Write
\[ \min_i \| \bar G_n e_i \|_{p^*}  \ge \min_i  \e \| \bar G_n e_i \|_{p^*}   -  \max_i  \big | \| \bar G_n e_i \|_{p^*}  - \e  \| \bar G_n e_i \|_{p^*} \big|.\]
Here, from \eqref{eq:add1}, 
$$\max_i  \bigl| \| \bar G_n e_i \|_{p^*}  - \e  \| \bar G_n e_i \|_{p^*} \big|  = O_\p(\sqrt{\log n}),$$ which together with \eqref{eq:O4} implies that
\begin{align} \label{eq:O31}
\min_i \| \bar G_n e_i \|_{p^*} =\xi_{p^*}n^{1/p_*}-O_\p(\sqrt{\log n}).
\end{align}
Finally, combining \eqref{eq:O31}, \eqref{eq:O1}, \eqref{eq:O2}, and \eqref{eq:O3}, we have
\begin{align} \label{eq:O5}
\min_i \la v_i + e_i, G_n ( v_i+ e_i) \ra &= \sqrt{2} \xi_{p^*} n^{1/p^*} - O_\p(n^{-\frac{1}{p} + \frac{1}{2} + \frac{1}{p^*}} ( \log n )^{p^*/4}+\sqrt{\log n}).
\end{align}
From the definition of $v_i$ and the fact that $||1+x|^p-1-|x|^p|\leq 2^{p}p|x|$, 
\begin{equation} \label{eq:o_i_pnorm:1}
\max_i  \bigl| \| v_i+ e_i \|_p^p  -2 \bigr| = \max_i \bigl|  |1+ v_i(i)|^p  - 1  -  |v_i(i)|^p \bigr|  \leq 2^pp\max_{i}|v_i(i)|. 
\end{equation}
Hence, from  \eqref{extra:eq3},
\begin{align*}
\max_i  \bigl| \| v_i+ e_i \|_p^p  -2 \bigr|=O_\p((\log n)^{p^*/(2p)}  n^{ - 1/p}).
\end{align*}
The desired lower bound \eqref{eq:Gnp_limit_approx_opt} follows from \eqref{eq:O5} after dividing both sides by  $\| v_i+ e_i \|_p^2$. This completes our proof.

\subsection{Proof of \eqref{eq:Gnp_limit_p<2}: the upper bound}

By the virtue of  \eqref{eq:Gnp_limit_approx_opt}, it remains to establish the upper bound. Let $S$ be a nonempty subset of $B_p^n$ and $\eps  =n^{ - 2/(pp^*)}$. For this $\varepsilon$ and any $x\in S,$ write according to \eqref{eq:decompose1}, $x=x^o+x^\dagger.$ 
 Applying  \eqref{eq:cross_term_dominate} yields
\begin{align*}
\e \max_{x \in S }\la G_nx,x\ra  &\le   \sqrt{2}  \e \max_{x \in S}\la  x^o, \bar G_n x^\dagger \ra  +O( n^{2/p^{*2}} \sqrt{\log n}).
\end{align*}
Since $x^o$ and $x^\dagger$ have disjoint supports, $\|x^o\|_p^p + \| x^\dagger\|_p^p =\| x\|_p^p \le 1$. After normalizing $x^o$ and $x^\dagger$ by their $\ell_p$-norms, we write
\begin{align*}
\e \max_{x \in S}\la x^o,  \bar G_n x^\dagger \ra  &=    \e \max_{x \in S } \|x^o \|_p \| x^\dagger\|_p  \left \la  \frac{ x^o}{\|x^o \|_p},  \bar G_n\frac{x^\dagger}{\| x^\dagger\|_p} \right  \ra \\
&\le 2^{ -2/p} \e  \max_{x \in S } \left \la  \frac{ x^o}{\|x^o \|_p},  \bar G_n \frac{x^\dagger}{\| x^\dagger\|_p} \right  \ra \le  2^{ -2/p} \e  \max_{x \in S, y \in B_p^n } \left \la y,   \bar G_n \tfrac{x^\dagger}{\| x^\dagger\|_p} \right  \ra,
\end{align*}
where we used the fact that  the product $ \|x^o \|_p \| x^\dagger\|_p $  is maximized if $\|x^o \|_p^p = \|x^\dagger\|_p^p = 1/2$.
Denote  $L^\#=\{x^\dagger/\|x^\dagger\|_p:x\in S\}$. Since $\|x^\dagger\|_0\leq \eps^{-p}$ for all $x\in S$, Proposition~\ref{lem:max_over_T} yields
\begin{align*}
\e  \max_{x \in S, y \in B_p^n } \left \la y,   \bar G_n \tfrac{x^\dagger}{\| x^\dagger\|_p} \right  \ra&\leq r(L^\#)  \e \| g\|_{p^*} + O( (n^{p/(2p^*)}  + n^{2/p^{*2}}  ) \sqrt{\log n})\\
&= r (L^\#) \e \| g\|_{p^*} + O (n^{p/(2p^*)}  \sqrt{\log n}),
\end{align*}
the last equality being a consequence of the fact $pp^* \ge 4$.
Combining the above estimates, we arrive at
\begin{equation}\label{eq:upperbound_Gnp_expectation}
\e \max_{x \in S }\la G_nx,x\ra \leq 2^{1/2} 2^{ -2/p}  r(L^\#) \e \| g\|_{p^*}  +  O (n^{p/(2p^*)}  \sqrt{\log n}).
\end{equation}
Since $\|x^\dagger\|_2 \le \|x^\dagger\|_p$, we have
$r(L^\#) \le 1.$
Consequently, using \eqref{eq:Gaussian_p_norm} in \eqref{eq:upperbound_Gnp_expectation}, we obtain that 
\begin{align*}
\e \max_{\| x\|_p=1 }\la G_nx,x\ra 
&\le 2^{1/2-2/p} \xi_{p^*} n^{1/p^*}  +  O (n^{p/(2p^*)}  \sqrt{\log n}).  
\end{align*}
Now the Gaussian concentration \eqref{concen1} ensures that  
\begin{equation}\label{eq:upperbound_Gnp}
\max_{\| x\|_p=1 }\la G_nx,x\ra = 2^{1/2-2/p} \xi_{p^*} n^{1/p^*}  +  O_\p (n^{p/(2p^*)}  \sqrt{\log n}),
\end{equation}
where the notion $O_\p$ is defined in \eqref{extra:notation}.
Finally, the upper bound \eqref{eq:upperbound_Gnp}, combined with the lower bound \eqref{eq:Gnp_limit_approx_opt} and the fact that $p<2$, implies \eqref{eq:Gnp_limit_p<2}.

\section{Proof of Theorem \ref{extra:thm1}: the stability bound}

For each $i\in [n]$, define three vectors, 
\begin{align*}
o_i  = \frac{v_i + e_i }{\|v_i + e_i \|_p}, \ \ \ o_i^o = \frac{v_i  - v_i(i) e_i}{\|v_i + e_i \|_p},\ \ \ o_i^\dagger =  \frac{(1+v_i(i))e_i}{\|v_i + e_i \|_p}.
\end{align*}
Obviously $o_i^\dagger/\| o_i^\dagger\|_p  = e_i$. Note that a direct computation gives
\begin{align*}
\| o_i^\dagger\|_p &= \frac{| 1+ v_i(i) | }{\|v_i + e_i \|_p} ,\quad
\| o_i^o\|_p = \frac{ (1  -  |v_i(i) |^p)^{1/p} }{\|v_i + e_i \|_p}, \ \ \  \text{ and } 
\end{align*}
\begin{align*}
\big \| \frac{o_i^o}{\| o^o_i \|_p} - v_i \big  \|_p &=\big  \| \frac{v_i  - v_i(i) e_i}{\| v_i  - v_i(i) e_i \|_p}  - v_i \big  \|_p =\bigl((1-(1-|v_i(i)|^p)^{1/p})^p+|v_i(i)|^p\bigr)^{1/p}\\
&\leq 1-(1-|v_i(i)|^p)^{1/p}+|v_i(i)|\leq 2|v_i(i)|.
\end{align*}
The above results and \eqref{eq:o_i_pnorm:1} imply that if $2^pp\max_i\|v_i\|_\infty\leq 1,$ then there exists a constant $K\geq 1$ such that 
\begin{align}\label{extra:eq4}
\max\Bigl(\max_i | \| o_i^o\|_p - 2^{-1/p} | ,\max_i | \| o_i^\dagger \|_p - 2^{-1/p} |,\max_i\big \| \frac{o_i^o}{\| o^o_i \|_p} - v_i \big  \|_p\Bigr)&\leq K\max_i\|v_i\|_\infty.
\end{align}

Next, in view of the proof of \eqref{eq:Gnp_limit_p<2}, we continue to work with the same $\eps  =n^{ - 2/(pp^*)}$.  Fix $0<\kappa < 1/(3p)$ and let  $n^{-\kappa}\leq \delta \leq 1$. Let $M\geq 1,$ set 
\begin{equation}\label{def:delta1}
\bd = \frac{1}{100M}\min\bigl(\delta,2^{-\frac{1}{2p}}\bigr).
\end{equation} 
Define $A_n=\bigl\{\max_i\|v_i\|_\infty\leq \min(\eps, \bd^3/K)\bigr\}.$
Since $\kappa<1/(3p)$ and $p^*>2,$ using \eqref{extra:eq3} yields that for any $d>0,$ there exists a universal constant $C>0$ such that 
\begin{align} \label{extra:eq7}
\p(A_n)\geq 1-Cn^{-d}.
\end{align}
On the event $A_n,$ we readily see that $2^pp\max_i\|v_i\|_\infty\leq 8\max_i\|v_i\|_\infty \leq 1$ and from \eqref{extra:eq4}, 
\begin{align}
\begin{split}\label{extra:lem1:eq1}
\max_i | \| o_i^o\|_p - 2^{-1/p} |&\leq \bd^3,
\end{split}\\
\begin{split}\label{extra:lem1:eq2}
\max_i | \| o_i^\dagger \|_p - 2^{-1/p} |&\leq\bd^3,
\end{split}\\
\begin{split}\label{extra:lem1:eq3}
\max_i \big \| \frac{o_i^o}{\| o^o_i \|_p} - v_i \big  \|_p &\leq \bd^3.
\end{split}
\end{align}
Let $S =\{ x  :\|x\|_p=1, \mathrm{dist}(x, \mathcal{O}) \ge \delta \}$. Define $S_1 = \bigl\{ x :\|x\|_p=1,\, | \|x^\dagger \|_p - 2^{-1/p}| \ge \bd^3 \bigr\},$ and 
\begin{align*}
S_2& = \bigl\{ x: \|x\|_p=1,\,\,| \|x^\dagger \|_p - 2^{-1/p}| \le \bd^3,  \  \min\bigl(\big \| \frac{x^\dagger}{\|x^\dagger \|_p}  - e_i \big \|_p,\big \| \frac{x^\dagger}{\|x^\dagger \|_p}  + e_i \big \|_p\bigr) \ge  \bd^3 \text{ for all } i\bigr\}.
\end{align*}
Also, let $S_3$ be the collection of all $x \in S$  satisfying that $| \|x^\dagger \|_p - 2^{-1/p}| \le \bd^3$ and there exists some $i$ such that either 
\begin{align}\label{extra:eq5}
\left ( \Big \| \frac{x^\dagger}{\|x^\dagger\|_p} - e_i  \Big \|_p \le   \bd^3, \ \Big  \| \frac{x^o}{\|x^o\|_p}  -v_i \Big \|_p  \ge  \bd \right)  \  \text{ or }  \   \left ( \Big \| \frac{x^\dagger}{\|x^\dagger\|_p} + e_i  \Big \|_p \le   \bd^3,\ \Big  \| \frac{x^o}{\|x^o\|_p}  + v_i \Big \|_p  \ge \bd \right).
\end{align}
\begin{lemma}\label{extra:lem1}
	For any $M\geq 1,$ on the event $A_n$, we have that $S \subseteq S_1 \cup S_2 \cup S_3.$
\end{lemma}

\begin{proof}
	Recall that, on $A_n$,  the inequalities \eqref{extra:lem1:eq1}, \eqref{extra:lem1:eq2}, and \eqref{extra:lem1:eq3} are valid. Take $x \in S \setminus (S_1 \cup S_2)$.  Then 
	\begin{align}\label{eq:o_i0}
	| \|x^\dagger \|_p - 2^{-1/p}| \le \bd^3, \ \ \text{ and }
	\end{align}
	\begin{align}\label{eq:o_i3}
	\mbox{either}\ \ \ \Bigl\| \frac{x^\dagger}{\|x^\dagger \|_p}  - e_i \Bigr\|_p \le \bd^3\ \ \ \mbox{or} \ \ \  \Bigl\| \frac{x^\dagger}{\|x^\dagger \|_p}  + e_i \Bigr \|_p \le \bd^3 \ \ \mbox{for some $i$}.
	\end{align}
	Assume, for definiteness, that  the former condition of \eqref{eq:o_i3} is valid for some $i$. To show that  $x\in S_3,$ we verify that $\| {x^o}/{\|x^o\|_p}  -v_i \big \|_p  \ge \bd$ as follows.
	First, note that from \eqref{extra:lem1:eq2} and \eqref{eq:o_i0},  
	\begin{align*}
	\big | \| x^\dagger \|_p   - \| o_i^\dagger \|_p \big |  \le 2\bd^3.
	\end{align*}
	We claim that $\| x^\dagger - o_i^\dagger\|_p \leq  3\bd^3$.  If $\| x^\dagger - o_i^\dagger\|_p > 3\bd^3$, we use the fact that  $\|x^\dagger\|_p\leq 1$ and the triangle inequality to deduce 
	\begin{align*}
	\Bigl\| \frac{x^\dagger}{\|x^\dagger \|_p}  - e_i \Bigr \|_p  &=  \Bigl\| \frac{x^\dagger}{\|x^\dagger\|_p}  -  \frac{o_i^\dagger}{\|o_i\|_p} \Bigr\|_p \ge  \frac{ \| x^\dagger - o_i^\dagger \|_p}{ \|x^\dagger\|_p}  - \| o_i^\dagger\|_p \Big | \frac{1}{\|x^\dagger\|_p}  - \frac{1}{\|o_i^\dagger\|_p} \Big| \\
	&=   \frac{1}{\|x^\dagger\|_p}\bigl(\| x^\dagger - o_i^\dagger \|_p  -| \|x^\dagger\|_p  - \|o_i^\dagger\|_p| \bigr)>3 \bd^3 - 2\bd^3 = \bd^3,
	\end{align*}  
	which contradicts our assumption that  $  \| {x^\dagger}/{\|x^\dagger \|_p}  - e_i \big \|_p \le \bd^3 $.  This validates our claim.
	Now since $x\in S$, we have that $\| x - o_i \|_p \ge \delta$, which, coupled  with  the fact $\| x^\dagger - o_i^\dagger\|_p \leq  3\bd^3$,  leads to
	\begin{align}\label{add:Eq-2}
	\| x^o - o_i^o\|_p&\geq \|x-o_i\|_p-\|x^\dagger-o_i^\dagger\|_p\geq \delta-3\bd^3.
	\end{align}
	On the other hand, it follows from \eqref{eq:o_i0}  that
	\begin{align}\label{add:Eq-1}
	| \|x^o\|_p-2^{-1/p} | \le 16 \bd^3.
	\end{align} 
	Indeed, a straightforward computation using  \eqref{eq:o_i0}, \eqref{def:delta1}, and  the fact $1=\|x\|_p^p=\|x^o\|_p^p+\|x^\dagger\|_p^p$ shows  $  \|x^o\|_p \ge 4^{-1} $. Now by applying the mean value theorem on the function $t \mapsto t^p$ twice and using \eqref{eq:o_i0}, the assertion \eqref{add:Eq-1} follows since
	\[    4^{-1}  \bigl| \|x^o\|_p-2^{-1/p} \bigr|\leq \bigl | \|x^o\|^p _p-  2^{-1} \bigr | =\bigl | \|x^\dagger\|^p _p-  2^{-1} \bigr |  \leq 2  \bigl| \|x^\dagger\|_p-2^{-1/p} \bigr|\leq 2\bd^3.\,\, \]
%
	Now from  \eqref{extra:lem1:eq1},  \eqref{add:Eq-2}, \eqref{add:Eq-1}, and the fact that $\|x^o\|_p\leq 1$, we obtain that 
		\begin{align} \label{eq:add_a1}
	\Big\| \frac{x^o}{\|x^o \|_p}  -  \frac{o_i^o}{\|o_i^o \|_p} \Big \|_p  &\ge   \frac{1}{\|x^o\|_p}\bigl(\| x^o - o_i^o \|_p -| \|x^o\|_p  - \|o_i^o\|_p| \bigr) \nonumber \\
	&\geq \frac{1}{\|x^o\|_p}\bigl(\| x^o - o_i^o \|_p -| \|x^o\|_p-2^{-1/p}|  -|2^{-1/p}- \|o_i^o\|_p| \bigr) \ge \delta- 20\bd^3.
	\end{align} 
	Moreover, another application of  triangle inequality on \eqref{extra:lem1:eq3} and \eqref{eq:add_a1} and the fact that $\delta\geq 100 \bd$ yield 
	\[ \Bigl\| \frac{x^o}{\|x^o \|_p}  - v_i \Big \|_p \ge \delta- 21 \bd^3 \geq  \bd^3.\]
	In the case that the second condition of \eqref{eq:o_i3} holds, we can similarly show, by using \eqref{extra:lem1:eq2} and \eqref{extra:lem1:eq3}, that $\| {x^o}/{\|x^o \|_p}  + v_i \big \|_p  \ge \bd.$ Therefore,  $x \in S_3$ and the assertion is established.  
\end{proof}

\begin{lemma}\label{extra:lem2}
    For any $d>0,$ there exist constants $M\geq 1$ and $C,C'>0$ such for all sufficiently large $n,$ with probability at least $1-8n^{-d}$, we have that whenever $n^{-\kappa}<\delta\leq 1$,
	\begin{align}\label{extra:lem2:eq1}
	\max_{x \in S_1\cup S_2\cup S_3}\la G_nx,x\ra&\leq 2^{1/2-2/p} \big( 1 - C\bd^6\big) \xi_{p^*} n^{1/p^*}  +   C'n^{p/{(2p^*)}}  \sqrt{\log n}.
	\end{align}
\end{lemma}

\begin{proof}
	Throughout this proof, the constants $C_1,C_2,$ etc. are independent of $\delta$ and $n.$ Let $d>0$. Below we handle each of the cases  $x \in S_i$ for $i=1,2, 3$ separately. Together they will complete the proof of the lemma.

	\smallskip
	
	{\noindent \bf Maximization over $S_1$.} Note that $|1-a^p|\geq |1-a|$ for all $a>0$. For $x \in S_1$, since $ | \|x^\dagger \|_p - 2^{-1/p}| \ge \bd^3$, we see that $$
	 | \|x^\dagger \|^p_p - 2^{-1} | \ge 2^{-1+1/p}|\|x\|_p-2^{-1/p}|\geq 2^{-1/2}|\|x\|_p-2^{-1/p}|\geq \frac{\bd^3}{2}.$$ Hence, 
	\[ \| x^o\|_p \|x^\dagger\|_p = \big( \|x^\dagger\|_p^p (1- \|x^\dagger\|_p^p)  \big)^{1/p}  \le \Big( \frac{1}{4}  -  \frac{\bd^6}{4}   \Big)^{1/p} \le 2^{-2/p} \big( 1 - \bd^6 \big). \]
	By following the argument used in the derivation  in \eqref{eq:upperbound_Gnp_expectation} and noting that $r \big( \big \{{x^\dagger/\| x^\dagger\|_p}: x \in S_1 \big \}  \big) \le 1$, we obtain
	\begin{align*}
	\e \max_{x \in S_1 }\la G_nx,x\ra \le  2^{1/2-2/p} \big( 1 - \bd^6\big) \xi_{p^*} n^{1/p^*}  +  O (n^{p/(2p^*)}  \sqrt{\log n}).  
	\end{align*}
	Note that $S_1$ is nonrandom. We can apply \eqref{concen1} with $t=\sqrt{2d\log n}$ to get that there exists some $C_1>0$ such that the following event is valid with probability at least $1-2n^{-d},$
	\begin{align*}
	\max_{x \in S_1 }\la G_nx,x\ra \le  2^{1/2-2/p} \big( 1 - \bd^6\big) \xi_{p^*} n^{1/p^*}  +  C_1n^{p/(2p^*)}  \sqrt{\log n}.
	\end{align*}
	
	{\noindent \bf Maximization over $S_2$.}
	For $x \in S_2$, $\big \| {x^\dagger}/{\|x^\dagger \|_p}  \mp e_i \big \|_p \ge  \bd^3$ for all $i$. By Lemma~\ref{lem:deficient_l_2_norm}, there exists a constant $C_2 >0$ such that 
	\[ r \big( \big \{x^\dagger/\| x^\dagger\|_p: x \in S_2 \big \}  \big) \le 1 - C_2\bd^{3p} \le 1 - C_2\bd^{6}  . \]
	On the other hand, the product $\| x^o\|_p \|x^\dagger\|_p$ can be trivially bounded above by $1/4$.
	Hence, \eqref{eq:upperbound_Gnp_expectation} yields
	\begin{align*}
	\e \max_{x \in S_2 }\la G_nx,x\ra \le  2^{1/2-2/p} \big( 1 - C_2\bd^{6}\big) \xi_{p^*} n^{1/p^*}  +  O (n^{p/(2p^*)}  \sqrt{\log n}).  
	\end{align*}
	Similar to the first case, as $S_2$ is also non-random, we can apply \eqref{concen1} with $t=\sqrt{2d\log n}$ to get  that there exists some $C_3>0$ such that the following event is valid with probability at least $1-2n^{-d},$
	\begin{align*}
	\max_{x \in S_2 }\la G_nx,x\ra \le  2^{1/2-2/p} \big( 1 - C_2\bd^{6}\big) \xi_{p^*} n^{1/p^*}  +  C_3n^{p/(2p^*)}  \sqrt{\log n}.  
	\end{align*}
	
	{\noindent \bf Maximization over $S_3$.} This case requires some extra treatments since $S_3$ is a random set.
	If \eqref{extra:eq5} is satisfied, then 
	\begin{align*}
	\max_{x \in S_3 } \left \la \frac{x^o}{\|x^o\|_p},   \bar G_n \frac{x^\dagger}{\| x^\dagger\|_p} \right  \ra &\le   \max_i \max_{ \substack{\| u \|_p  = 1, \| u - e_i\|_p  \le \bd^3, \\ \|v\|_p=1, \| v  - v_i \|_p \ge \bd} } \left \la v ,   \bar G_n u \right  \ra \\
	&\le \max_i \max_{  \|v\|_p=1, \| v  - v_i \|_p \ge \bd } \left \la v ,   \bar G_n e_i \right  \ra  + \bd^3 \| \bar G_n \|_{p \to p^*}  \\
	&\le \Bigl( 1- \frac{p^2  \bd^2 }{16 p^*}    \Bigr) \max_i   \| \bar G_n e_i \|_{p^*}  + \bd^3\| \bar G_n \|_{p \to p^*},
	\end{align*} 
	where the third inequality used Lemma~\ref{lem:holder_stability}.  Now the  fact that $\sup_{x\in B_n^p}\|x^o\|_p\|x^\dagger\|_p \leq 2^{-2/p}$ implies that
	\begin{align*}
	\max_{x \in S_3 } \bigl \la x^o,   \bar G_n x^\dagger  \bigr \ra&=\max_{x \in S_3 }\|x^o\|_p\|x^\dagger\|_p \Bigl \la \frac{x^o}{\|x^o\|_p},   \bar G_n \frac{x^\dagger}{\| x^\dagger\|_p} \Bigr  \ra  \\
	&\leq 2^{-2/p}\Bigl(\Bigl ( 1- \frac{p^2 \bd^2 }{16 p^*}    \Bigr) \max_i   \| \bar G_n e_i \|_{p^*}  + \bd^3\| \bar G_n \|_{p \to p^*}\Bigr).
	\end{align*}
	To control the right-hand side, note that $\| \bar G_n e_i \|_{p^*} - \e \| \bar G_n e_i \|_{p^*} $ is $2$-subgaussian. Lemma~\ref{lem:maxi_subG} yields that
	\[ \e  \max_i   \| \bar G_n e_i \|_{p^*}  \le \max_i \e    \| \bar G_n e_i \|_{p^*} + 2 \sqrt{ \log n} \le \xi_{p^*} n^{1/p^*} + O(\sqrt{\log n}).\]
	Moreover, $ \e \| \bar G_n \|_{p \to p^*} = O( n^{1/p^*})$ by \eqref{extra:eq9}. Using \eqref{concen1} with $t=\sqrt{2(d+1)\log n}$ and \eqref{concen2} with $t=\sqrt{2d\log n}$, we see that there exist universal positive constants $C_4$ and $C_5$ such that with probability at least $1-4n^{-d}$,
	\begin{align*}
	\max_i   \| \bar G_n e_i \|_{p^*}& \leq \xi_{p^*} n^{1/p^*}+C_4\sqrt{\log n} \ \ \ \mbox{and} \ \ \ \| \bar G_n \|_{p \to p^*}\leq C_5n^{1/p^*}.
	\end{align*}
	From these estimates, we can choose $M\geq 1$ large enough such that the following statement holds: there exist  constants $C_6,C_7>0$ such that with probability at least $1-4n^{-d},$ we have that whenever $n^{-\kappa}<\delta\leq 1$, 
	\[  \max_{x \in S_3 } \Bigl \la \frac{x^o}{\|x^o\|_p},   \bar G_n \frac{x^\dagger}{\| x^\dagger\|_p} \Bigr\ra \le 2^{-2/p} (1- C_6\bd^2) \xi_{p^*} n^{1/p^*}+C_7n^{p/{(2p^*)}}  \sqrt{\log n}.\]
	Plugging this into  \eqref{eq:cross_term_dominate2}, we arrive at that with probability at least $1-4n^{-d},$
	\begin{align*}
	\begin{split}\notag
	\max_{x \in S_3 }\la G_nx,x\ra &\le  2^{1/2-2/p} \big( 1 - C_6\bd^{2}\big) \xi_{p^*} n^{1/p^*}  +  C_7n^{2/{p^*}^2}  \sqrt{\log n}
	\end{split}\\
	\begin{split}
	&\leq 2^{1/2-2/p} \big( 1 - C_6\bd^{6}\big) \xi_{p^*} n^{1/p^*}  +  C_7n^{p/{(2p^*)}}  \sqrt{\log n},
	\end{split}
	\end{align*}
	where the second inequality used $\bd\leq 1$ and $pp^*\geq 4.$ 
	
\end{proof}

We now proceed to complete the proof of Theorem \ref{extra:thm1}. Let $d>0$. Denote by $A_n'$ the event on which \eqref{extra:lem2:eq1} is valid. By Lemmas \ref{extra:lem1} and \ref{extra:lem2}, we have that on the event $A_n\cap A_n'$, for all $n^{-\kappa}<\delta\leq 1,$
\begin{align*}
\max_{x\in S}\la G_nx,x\ra&\leq \max_{x \in S_1\cup S_2\cup S_3}\la G_nx,x\ra\leq 2^{1/2-2/p} \big( 1 - C\delta^6\big) \xi_{p^*} n^{1/p^*}  +   C'n^{p/{(2p^*)}}  \sqrt{\log n}.
\end{align*} 
Here, from \eqref{extra:eq7} and Lemma \ref{extra:lem2}, $\p(A_n\cap A_n')\geq 1-\p(A_n^c)-\p({A_n'}^c)\geq 1-(8+C)n^{-d}$. This completes our proof.

\section{Proof of Theorem \ref{thm-1}}\label{sec7}

\subsection{Proof of \eqref{thm-1:eq1}}
For $t>0$ and $x\in \mathbb{R}^n$, set
\begin{align} \label{add:eq-3}
H_{n,p,t}(x)&=\frac{\la G_nx,x\ra}{n^{1/2}}-t\|x\|_p^p.
\end{align}
Let $u,t>0$ be fixed. Note that $L_{n,p,u}(t)=n^{-1}\max_{\opnorm{x}_2^2=u}H_{n,p,t}(x).$
We claim that $n\e L_{n,p,u}(t)$ is superadditive, i.e.,
\begin{align}\label{eq:supper_additive_Lu}
n\e L_{n,p,u}(t)&\geq n_1\e L_{n_1,p,u}(t)+n_2\e L_{n_2,p,u}(t),\,\,\forall n_1,n_2\in \mathbb{N}\,\,\mbox{and}\,\,n=n_1+n_2.
\end{align}
To prove this, we use Guerra-Toninelli's interpolation \cite{GT02}. Let $n_1,n_2,n\in \mathbb{N}$ satisfy $n=n_1+n_2.$ 
For $0\leq s\leq 1,$ consider the interpolating Hamiltonian,
\begin{align*}
H_{n,s}(z)&=\sqrt{\frac{s}{n}}\la G_nz,z\ra+\sqrt{\frac{1-s}{n_1}}\la G_{n_1}^1x,x\ra+\sqrt{\frac{1-s}{n_2}}\la G_{n_2}^2 y,y\ra-t\|x\|_p^p-t\|y\|_p^p
\end{align*}  
for any $z=(x,y)\in \mathbb{R}^{n_1}\times \mathbb{R}^{n_2}$, where $G_{n_1}^1=(g_{ij}^1)_{i,j\in[n_1]}$ and $G_{n_2}^2=(g_{ij}^2)_{i,j\in [n_2]}$ are i.i.d.\ standard Gaussian. Let $\mu_{n_1}$  and $\mu_{n_2}$ be probability measures  fully supported on the sphere $S_{n_1}=\{x\in \mathbb{R}^{n_1}:\opnorm{x}_2^2=u\}$ and $S_{n_1}=\{y\in \mathbb{R}^{n_2}:\opnorm{y}_2^2=u\}$, respectively. Here note that the $\opnorm{x}_2$ and $\opnorm{y}_2$ are normalized by $n_1$ and $n_2,$ respectively.  Define
\begin{align}\label{interpolation}
F_{n_1,n_2,\beta}(s)=\frac{1}{\beta n}\e\log \int_{S_{n_1}\times S_{n_2}} e^{\beta H_{n,s}(z)}\mu_{n_1}(dx)\otimes\mu_{n_2}(dy).
\end{align}
Note that 
\begin{align*}
F_{n_1,n_2,\beta}(1)&=\frac{1}{\beta n}\e\log \int_{S_{n_1}\times S_{n_2}} e^{\frac{\beta}{\sqrt{n}}\la G_nz,z\ra}\mu_{n_1}(dx)\otimes\mu_{n_2}(dy),\\
F_{n_1,n_2,\beta}(0)&=\frac{n_1}{\beta n}\e\log \int_{S_{n_1}}e^{\frac{\beta}{\sqrt{n_1}}\la G_{n_1}^1x,x\ra}\mu_{n_1}(dx)+\frac{n_2}{\beta n}\e\log \int_{S_{n_2}}e^{\frac{\beta}{\sqrt{n_2}}\la G_{n_2}^2y,y\ra}\mu_{n_2}(dy).
\end{align*}
Denote by $R_k(a,b)=k^{-1}\sum_{i\in [k]}a_ib_i$ for any $k\geq 1$ and $a,b\in \mathbb{R}^k.$ Note that $R_n(z,z)=R_{n_1}(x,x)=R_{n_2}(y,y)=u$ for $z=(x,y)\in S_{n_1}\times S_{n_2}.$ Computing the derivative of \eqref{interpolation} in $s$ directly by using the Gaussian integration by parts and noting that $\|x\|_2^2=un_1$ and $\|y\|_2^2=un_2$, we obtain that 
\begin{align*}
& nF_{n_1,  n_2,\beta}'(s)=\frac{1}{2}\e \Bigl\la \frac{\la G_nz,z\ra}{\sqrt{s}}-\frac{\la G_n^1 x,x\ra}{\sqrt{1-s}}-\frac{\la G_n^2 y,y\ra}{\sqrt{1-s}}\Bigr\ra\\
&=\frac{\beta^2n}{2}\e \bigl\la R_n(z,z)^2-R_n(z^1,z^2)^2\bigr\ra
-\frac{\beta^2n_1}{2}\e \bigl\la R_{n_1}(x,x)^2-R_{n_1}(x^1,x^2)^2\bigr\ra
-\frac{\beta^2n_2}{2}\e \bigl\la R_{n_2}(y,y)^2-R_{n_2}(y^1,y^2)^2\bigr\ra\\
&=-\frac{\beta^2}{2}\e \bigl\la nR_n(z^1,z^2)^2-n_1R_{n_1}(x^1,x^2)^2-n_2R_{n_2}(y^1,y^2)^2\bigr\ra,
\end{align*}
where the angle bracket outside is the expectation with respect to the Gibbs measure 
\begin{align*}
G_{n_1,n_2,\beta}(dx,dy)=\frac{e^{\beta H_{n,s}(z)}\mu_{n_1}(dx)\otimes\mu_{n_2}(dy)}{\int e^{\beta H_{n,s}(z)}\mu_{n_1}(dx)\otimes\mu_{n_2}(dy)}.
\end{align*}
and $z^1=(x^1,y^1)$ and $z^2=(x^2,y^2)$ are two independent samples from this measure. Now note that from Jensen's inequality,
\begin{align*}
R_n(z^1,z^2)^2&=\Bigl(\frac{n_1}{n}\frac{1}{n_1}R_{n_1}(x^1,x^2)+\frac{n_2}{n}R_{n_2}(y^1,y^2)\Bigr)^2\\
&\leq \frac{n_1}{n}R_{n_1}(x^1,x^2)^2+\frac{n_2}{n}R_{n_2}(y^1,y^2)^2,
\end{align*}
leading to $F_{n_1,n_2,\beta}'(s)\geq 0$ for all $s\in (0,1).$ Hence, $F_{n_1,n_2,\beta}(1)\geq F_{n_1,n_2,\beta}(0)$ and sending $\beta $ to infinity imply that
\begin{align*}
n\e L_{n,p,u}(t)&\geq \e\max_{S_{n_1}\times S_{n_2}}H_{n,p,t}(z)\geq n_1\e L_{n_1,p,u}(t)+n_2\e L_{n_2,p,u}(t),
\end{align*}
establishing our claim \eqref{eq:supper_additive_Lu}. So,  $L_{p,u}(t) : = \lim_{n\to\infty} \e[L_{n, p,u}(t)] $ exists by superadditivity. Hence, from the concentration bound \eqref{extra:eq12}, it follows that $\lim_{n\to\infty}L_{n,p,u}(t) = \lim_{n\to\infty} \e [L_{n,p,u}(t)]  =   L_{p,u}(t)$ exists almost surely for each $u > 0$ and $t>0$. 
It remains to show that  the statement that almost surely this limit exists simultaneously for all $u >  0$ and $t>0$. We need the following lemma.

\begin{lemma}\label{extra:lem3}
	For any $0<K_1<K_2,$ there exist positive constants $C$ and $D$ such that with probability at least $1-Ce^{-n/C},$
	\begin{align}\label{extra:lem3:eq1}
	\bigl|L_{n,p,u}(t)-L_{n,p,u'}(t')\bigr|\leq D(|u-u'|+|t-t'|),\,\,\forall u,u',t,t'\in [K_1,K_2].
	\end{align}
\end{lemma} 

From the above inequality and  the Borel-Cantelli lemma,  it follows that  that for any $0<K_1<K_2$, there exists some $D = D(K_1, K_2) >0$ such that almost surely 
\begin{align}
\label{extra:eq19}
\limsup_n |L_{n, p,u}(t)-L_{n, p,u'}(t')|\leq D(|u-u'|+|t-t'|),\,\,\forall u,u',t,t'\in [K_1,K_2].
\end{align}
The above inequality also implies that 
\begin{equation} \label{extra:eq192}
 |L_{p,u}(t)-L_{p,u'}(t')|\leq D(|u-u'|+|t-t'|),\,\,\forall u,u',t,t'\in [K_1,K_2].
\end{equation}

For any $u,t\in \mathbb{Q}\cap(0,\infty)$,  let $\Omega(u,t)$ be the event on which $\lim_{n\to\infty}L_{n,p,u}(t)$ exists.  Also, for any $K_1,K_2\in \mathbb{Q}\cap (0,\infty)$ with $K_1<K_2,$ let $\Omega(K_1,K_2)$ be the event on which the inequality \eqref{extra:eq19} holds.  Let $\Omega$ be the intersection of all $\Omega(u,t)$ and $\Omega(K_1,K_2).$ Obviously,  $\Omega$ is of probability one. Moreover, on $\Omega$,  for any $t>0$ and $u>0$ with  $ u, t \in [K^{-1}, K] $ for some rational $K >0$, we have, using \eqref{extra:eq19} and \eqref{extra:eq192},  that 
\begin{align*}
&\limsup_{n\to\infty}\bigl|L_{n,p,u}(t)-L_{p,u}(t)\bigr|\\
&\leq \limsup_{n\to\infty}\bigl|L_{n,p,u}(t)-L_{n,p,u'}(t')\bigr|+\limsup_{n\to\infty}\bigl|L_{n,p,u'}(t')-L_{p,u'}(t')\bigr| + \bigl|L_{p,u'}(t')-L_{p,u}(t)\bigr|  \\
&\leq 2D'(|u-u'|+|t-t'|)
\end{align*}
for any $u',t'\in \mathbb{Q}\cap [K^{-1}, K]$, where $D'$ is a constant depending only on $K$. Sending $u'\to u$ and $t'\to t$ implies that $\lim_{n\to\infty}L_{n,p,u}(t)=L_{p,u}(t)$ on $\Omega$. This completes the proof of \eqref{thm-1:eq1}.

\begin{proof}[\bf Proof of Lemma \ref{extra:lem3}]
	Let $0< K_1< K_2$ be fixed. Note that a change of variable, $x\to x/\sqrt{u}$, yields that 
	\begin{align}\label{extra:eq15}
	L_{n,p,u}(t)&=\max_{\opnorm{x}_2=1}\Bigl(\frac{u}{n^{3/2}}\la G_nx,x\ra-tu^{p/2}\opnorm{x}_p^p\Bigr).
	\end{align}
	Let $\mathbbm{1}=(1,\ldots,1) \in \r^n$. Since $$
	L_{n,p,u}(t)\geq \frac{u}{n^{3/2}}\la G_n\mathbbm{1},\mathbbm{1}\ra-tu^{p/2}\opnorm{\mathbbm{1}}_p^p\geq -\frac{u\|G_n\|_2}{n^{1/2}}-tu^{p/2},
	$$
	it follows that if $x$ is an optimizer of \eqref{extra:eq15}, then
	\begin{align}\label{add:E2}
	\opnorm{x}_p^p&\leq \frac{u^{1-p/2}}{tn^{3/2}}\la G_nx,x\ra+\frac{u^{1-p/2}\|G_n\|_2}{tn^{1/2}}+1\leq  \frac{2u^{1-p/2}\| G_n\|_2}{tn^{1/2}}+1\leq M_n
	\end{align}
	for 
	$
	M_n:=2K_1^{-p/2}n^{-1/2}\|G_n\|_2+1.
	$
	Consequently, $L_{n,p,u}(t)$ can be written as the same supremum as \eqref{extra:eq15}, but with the extra constraint \eqref{add:E2}. This implies that for any $u,u',t,t'\in [K_1,K_2],$
	\begin{align} \label{eq:lip_bound_L_npu}
	|L_{n,p,u}(t)-L_{n,p,u'}(t')|&\leq \frac{\| G_n\|_2}{\sqrt{n}}|u-u'|+M_n|tu^{p/2}-t'{u'}^{p/2}| \nonumber \\
	&\leq \frac{\|G_n\|_2}{\sqrt{n}}|u-u'|+M_n\bigl|(t-t')u^{p/2}+t'(u^{p/2}-{u'}^{p/2})\bigr|  \nonumber \\
	&\leq \frac{\| G_n\|_2}{\sqrt{n}}|u-u'|+M_n\Bigl(K_2^{p/2}|t-t'|+\frac{pK_2^{p/2}}{2}|u-u'|\Bigr).
	\end{align}
	Since $G_n\mapsto \|G_n\|_2$ is $1$-Lipschitz with respect to $\|\cdot\|_F$,  It follows from  the concentration inequality \eqref{eq:Gconc} for $\| G_n\|_2$ and \eqref{extra:eq9} that there exist some $C,C_*>0$ such that
	\begin{align}\label{add:eq2}
	\p\bigl(n^{-1/2}\|G_n\|_2\geq C_*\bigr)\leq  Ce^{-n/C}.
	\end{align} 
	The bound \eqref{extra:lem3:eq1} now is a consequence of     \eqref{eq:lip_bound_L_npu} and \eqref{add:eq2}.
\end{proof}

\subsection{Proof of \eqref{thm-1:eq2}}

\begin{lemma}\label{extra:lem4}
	There exist absolute constants $c_0,C>0$ such that with probability at least $1-Ce^{-n/C}$,  for any $0< K<t<\infty$,
	\begin{align*}
	L_{n,p}(t)=\sup_{u\in [0,c_0/K]}L_{n,p,u}(t).
	\end{align*}
\end{lemma}

\begin{proof}
	Note that $L_{n,p}(t)\geq 0.$ Recall the definition of $H_{n,p,t}(x)$ from \eqref{add:eq-3}. 
	Whenever $H_{n,p,t}(x)\geq 0$ for some $x\in \mathbb{R}^n$, we have, by Jensen's inequality, that
	\begin{align*}
	t\opnorm{x}_p^p&\leq \frac{\la G_nx,x\ra}{n^{3/2}}\leq \frac{\|G_n\|_2}{\sqrt{n}}\opnorm{x}_2^2\leq \frac{\|G_n\|_2}{\sqrt{n}}\opnorm{x}_p^2.
	\end{align*}
	The above inequality, together with another application of Jensen's inequality, yields that 
	\begin{align}\label{extra:eq20}
	\opnorm{x}_2\leq \opnorm{x}_p&\leq \Bigl(\frac{\| G_n\|_2}{t\sqrt{n}}\Bigr)^{1/(p-2)}\leq \Bigl(\frac{\|G_n\|_2}{K\sqrt{n}}\Bigr)^{1/(p-2)}.
	\end{align}
	Consequently, we can write $L_{n,p}(t)$ as the supremum of $L_{n,p,u}(t)$ for all $u$ between zero and the above upper bound. 
	Using \eqref{add:eq2} completes our proof.
\end{proof}

The proof of \eqref{thm-1:eq2} is now argued as follows. From Lemma \ref{extra:lem4}, there exist $c_0,C>0$ such that with probability at least $1-Ce^{-n/C}$, $L_{n,p}(t)=\sup_{0\leq u\leq c_0/K}L_{n,p,u}(t)$ for any $0<K\leq t<\infty.$
Now consider an arbitrary $0<K< \min(\sqrt{c_0}, 1).$ From Lemma \ref{extra:lem3}, there exist $C',D>0$ such that with probability at least $1-C'e^{-n/C'},$
\begin{align}\label{extra:eq14}
|L_{n,p,u}(t)-L_{n,p,u'}(t)|&\leq D|u-u'|,\,\,\ \forall u,u',t\in [K,c_0/K].
\end{align}
Consequently, for any $k\geq 1$, if $u_0 <\cdots < u_{k}$ form  an equidistant partition of $[K,c_0/K]$, then   with probability at least
$1 - C'e^{-n/C'},$
\begin{align*}
\Bigl|\sup_{u\in [K,c_0/K]}L_{n,p,u}(t)-\sup_{0 \leq j\leq k}L_{n,p,u_j}(t)\Bigr|\leq \frac{D}{k}(c_0/K-K),\,\,\forall t\in [K,c_0/K],
\end{align*}
This implies that with probability at least
$1-Ce^{-n/C}-C'e^{-n/C'},$
\begin{align}
\begin{split}\notag
\Bigl|L_{n,p}(t)-\sup_{0\leq j\leq k}L_{n,p,u_j}(t)\Bigr|&\leq \sup_{u\in[0,K]}L_{n,p,u}(t)+\frac{D}{k}(c_0/K-K)
\end{split}\\
\begin{split}\label{extra:eq16}
&\leq \frac{\|G_n\|_2K}{\sqrt{n}}+\frac{D}{k}(c_0/K-K),\,\,\forall t\in [K,c_0/K],
\end{split}
\end{align}
where the last inequality used the fact that $L_{n,p,u}(t)\leq \|G_n\|_2u/\sqrt{n}.$
Combining the fact that  $|L_{p,u}(t)-L_{p,u'}(t)|\leq D|u-u'|$ for all $u,u'\in [K,c_0/K]$, which follows  from \eqref{extra:eq14}, and the bound \eqref{add:eq2} on $\|G\|_2$, we can  now pass to the limit in \eqref{extra:eq16} by first letting  $n\to\infty$ and then letting $k\to\infty$  to obtain  that almost surely,  for all $t\in [K,c_0/K]$, 
$$
-C_*K+\sup_{u\in [K,c_0/K]}L_{p,u}(t)\leq \liminf_{n\to\infty}L_{n,p}(t)\leq \limsup_{n\to\infty}L_{n,p}(t)\leq \sup_{u\in [K,c_0/K]}L_{p,u}(t)+C_*K
$$
 Since this is valid for all $0<K<\sqrt{c_0}$, we can further send $K\downarrow 0$ to get that almost surely,
\begin{align}\label{extra:eq17}
\lim_{K\downarrow 0}\sup_{u\in [K,c_0/K]}L_{p,u}(t)\leq \liminf_{n\to\infty}L_{n,p}(t)\leq \limsup_{n\to\infty}L_{n,p}(t)\leq \sup_{u\in [0,\infty)}L_{p,u}(t),\,\,\forall t>0.
\end{align}
Finally, note that for any $u\geq 0$ and $t>0$, $L_{n,p}(t)\geq L_{n,p,u}(t)$. Also, almost surely,  $\lim_{n \to \infty } L_{n,p,u}(t)  = L_{p, u}(t)$ for all $t > 0$ and $u > 0$.
From \eqref{extra:eq16},  
\begin{align*}
L_{p,u}(t)
&\leq \sup_{0 \leq j\leq k}L_{p,u_j}(t)+C_*K+\frac{D}{k}(c_0/K-K) \leq \sup_{[K,c_0/K]}L_{p,u}(t)+C_*K+\frac{D}{k}(c_0/K-K)
\end{align*} 
for all $t\in [K,c_0/K],$ $u\geq 0$, and $k\geq 1.$
Hence,
\begin{align*}
\sup_{u\in [0,\infty)}L_{p,u}(t)\leq \sup_{[K,c_0/K]}L_{p,u}(t)+C_*K, 
\end{align*}
which, together with \eqref{extra:eq17},   completes our proof after sending $K\downarrow 0$.

\section{Proof of Theorem \ref{thm2}}\label{sec8}

We establish the proof of Theorem \ref{thm2} in this section. First, we need the following important lemma that connects $L_{n,p}$ and $\gp_{n,p}.$ 

\begin{lemma}\label{add:lem5}
	Let $2<p<\infty$ and $n \ge 1$. Then the following statements hold for any realization of $G_n$: 
	If $L_{n,p}$ is differentiable at some $t\in (0,\infty)$ with $L_{n,p}'(t)<0$, then
		\begin{align}\label{add:lem5:eq-1}
	L_{n,p}(t)&=-\Bigl(\frac{p}{2}-1\Bigr)tL_{n,p}'(t)
	\end{align}
	and
	\begin{align}\label{add:lem5:eq0}
	\gp_{n,p}=\frac{p}{2}\Bigl(\frac{p}{2}-1\Bigr)^{2/p-1}t^{2/p}L_{n,p}(t)^{1-2/p}.
	\end{align}
\begin{remark}
Though the differential equation \eqref{add:lem5:eq-1} has a simple explicit solution, it does not yield an expression for $L_{n,p}(t)$ in absence of a boundary condition. 
\end{remark}

\end{lemma}

\begin{proof}
  Let $t$ be a point of differentiability of $L_{n,p}$ with $L_{n,p}'(t)<0.$ 
	Let $\varepsilon>0$ be fixed. Then 	\begin{align*}
	\max_{\opnorm{x}_p^p\geq -L_{n,p}'(t)+\varepsilon}\frac{H_{n,p,t}(x)}{n}&\leq \max_{\opnorm{x}_p^p\geq -L_{n,p}'(t)+\varepsilon}\Bigl(\frac{H_{n,p,t}(x)}{n}+\lambda\bigl(\opnorm{x}_p^p+L_{n,p}'(t)-\varepsilon\bigr)\Bigr)\\
	&\leq L_{n,p}(t-\lambda)+\lambda L_{n,p}'(t)-\lambda \varepsilon\\
	&=\lambda\Bigl(L_{n,p}'(t)-\frac{L_{n,p}(t)-L_{n,p}(t-\lambda)}{\lambda}\Bigr)-\lambda \varepsilon+L_{n,p}(t)
	\end{align*}
	and
	\begin{align*}
	\max_{\opnorm{x}_p^p\leq  -L_{n,p}'(t)-\varepsilon}\frac{H_{n,p,t}(x)}{n}&\leq \max_{\opnorm{x}_p^p\leq -L_{n,p}'(t)-\varepsilon}\Bigl(\frac{H_{n,p,t}(x)}{n}+\lambda\bigl(-\opnorm{x}_p^p-L_{n,p}'(t)-\varepsilon\bigr)\Bigr)\\
	&\leq L_{n,p}(t+\lambda)-\lambda L_{n,p}'(t)-\lambda \varepsilon\\
	&=\lambda\Bigl(\frac{L_{n,p}(t+\lambda)-L_{n,p}(t)}{\lambda}-L_{n,p}'(t)\Bigr)-\lambda \varepsilon+L_{n,p}(t).
	\end{align*}
	From the above two inequalities, since $L_{n,p}$ is differentiable at $t,$ we can choose $\lambda$ small enough such that for some $c>0$ (depending on $\eps$)
	\begin{align}\label{add:eq--5}
	\max_{|\opnorm{x}_p^p+ L_{n,p}'(t)|\geq \varepsilon}\frac{H_{n,p,t}(x)}{n}&\leq L_{n,p}(t)-c\varepsilon.
	\end{align}
	In other words, if $x$ satisfies $|\opnorm{x}_p^p+ L_{n,p}'(t)|\geq \varepsilon$, then it can not be an optimizer of $L_{n,p}(t).$
	
	Now to show \eqref{add:lem5:eq-1}, note that $\opnorm{x}_p^p$ is differentiable on $\mathbb{R}^n$ since $p>2$. If $x$ is a critical point of $H_{n,p,t}$, then
	\begin{align*}
	\nabla H_{n,p,t}(x)&=\nabla \Bigl(\frac{1}{\sqrt{2}n^{3/2}}\la \bar G_nx,x\ra-t\opnorm{x}_p^p\Bigr)=\frac{2}{\sqrt{2}n^{3/2}}\bar G_nx- \Bigl( \frac{pt}{n}|x_i|^{p-1}\mbox{sgn}(x_i)\Bigr)_i=0.
	\end{align*}
	So, we have $\la x, \nabla H_{n,p,t}(x) \ra = 0$, which yields 
	\begin{align*}
	\frac{2}{n^{3/2}}\la G_nx,x\ra=\frac{\sqrt{2}}{n^{3/2}}\la \bar G_nx,x\ra =pt\opnorm{x}_p^p.
	\end{align*}
	It follows that if $x(t)$ is an optimizer of $H_{n,p,t}$, then
	\begin{align*}
	L_{n,p}(t)&=\Bigl(\frac{p}{2}-1\Bigr)t\opnorm{x(t)}_p^p.
	\end{align*}
	On the other hand, from \eqref{add:eq--5}, we also see that $|\opnorm{x(t)}_p^p+L_{n,p}'(t)|<\varepsilon$ for all $\varepsilon>0,$ which implies that $\opnorm{x(t)}_p^p=-L_{n,p}'(t)$. Plugging this into the above display results in  \eqref{add:lem5:eq-1}.
	
	Next, we establish \eqref{add:lem5:eq0}. First of all, from \eqref{add:eq--5},
	\begin{align*}
   \max_{|\opnorm{x}_p^p+  L_{n,p}'(t)|\leq  \varepsilon}\frac{H_{n,p,t}(x)}{n}&= L_{n,p}(t). 
	\end{align*}
	Since this is valid for all $\varepsilon>0,$ sending it to zero yields that 
		\begin{align*}
	L_{n,p}(t)&=\max_{\opnorm{x}_p^p=-  L_{n,p}'(t)}\frac{H_{n,p,t}(x)}{n} =\max_{\opnorm{x}_p^p=-L_{n,p}'(t)}
	\frac{\la G_nx,x\ra}{n^{3/2}}+tL_{n,p}'(t).
	\end{align*}
	Now, performing a change of variable $x\to x/(-L_{n,p}'(t))^{1/p}$ yields 
		\begin{align*}
	\gp_{n,p}=\frac{L_{n,p}(t)-tL_{n,p}'(t)}{\bigl(-L_{n,p}'(t)\bigr)^{2/p}}.
	\end{align*}
	Finally, by substituting $L_{n,p}'(t)$ using \eqref{add:lem5:eq-1} implies \eqref{add:lem5:eq0}.
	
\end{proof}

\begin{lemma}\label{add:lem6}
	Let $2<p<\infty.$ We have that $L_p(t)$ is differentiable in $t\in (0,\infty)$. In addition, for all $t\in (0,\infty),$
	\begin{align}
	\begin{split}\label{add:lem6:eq1}
	L_p(t)=-\Bigl(\frac{p}{2}-1\Bigr)tL_p'(t)
\end{split}
\end{align}
and
\begin{align}
\begin{split}\label{add:lem6:eq2}
	L_p'(t)<0.
	\end{split}
	\end{align}
\end{lemma}

\begin{proof}
	 To show differentiability of $L_p(t)$ and \eqref{add:lem6:eq1}, we first fix a realization of $(G_n)_{n\geq 1}$ for which $L_{n,p}(t)$  converges to  $L(t)$  for all $t\in (0,\infty)$.   Note that $L_{n,p} $ and $L_p$ are  convex.  Hence, both of them are differentiable almost everywhere on $(0,\infty).$  For any $0<s'<s<s'',$ define 
	\begin{align*}
	D_{n,p}^-(s;s')=\frac{L_{n,p}(s)-L_{n,p}(s')}{s-s'},\,\,D_{n,p}^+(s;s'')=\frac{L_{n,p}(s'')-L_{n,p}(s)}{s''-s}.
	\end{align*}
  Note that for any $0<t_1<s<t_2,$ if $L_{n,p}$ is differentiable at $s,$ then
	\begin{align*}
	D_{n,p}^-(s;t_1)\leq L_{n,p}'(s)\leq D_{n,p}^+(s;t_2),
	\end{align*}
	which combining with \eqref{add:lem5:eq-1} implies that
	\begin{align*}
    -\Bigl(\frac{p}{2}-1\Bigr)sD_{n,p}^+(s;t_2)\leq L_{n,p}(s)\leq -\Bigl(\frac{p}{2}-1\Bigr)s D_{n,p}^-(s;t_1).
	\end{align*}
	Since the set of points of differentiability of $L_{n,p}$ is dense in $(0,\infty)$, this inequality is indeed valid for all $0<t_1<t<t_2<\infty,$
	\begin{align*}
	-\Bigl(\frac{p}{2}-1\Bigr)tD_{n,p}^+(t;t_2)\leq L_{n,p}(t)\leq -\Bigl(\frac{p}{2}-1\Bigr)tD_{n,p}^-(t;t_1).
	\end{align*}
	 By sending $n\to\infty$ and then letting $t_1\uparrow t$ and $t_2\downarrow t$,  it follows from the above inequality that  if $L_p$ is differentiable at $t$, then \eqref{add:lem6:eq1} is valid at this point as well.  To show that \eqref{add:lem6:eq1} holds for all $t>0,$ fix an arbitrary $t>0$ and  pick two monotonic sequences $t_k$ and $s_k$ belonging to the set of differentiable points of $L_p$ such that $t_k\uparrow t$ and $s_k\downarrow t.$ From the continuity and convexity of $L_p$ and the validity of \eqref{add:lem6:eq1} on $(t_k)$ and $(s_k)$, 
	\begin{align*}
	-(p/2-1)t\lim_{k\to\infty}L_p'(t_k)=L_p(t)=-(p/2-1)t\lim_{k\to\infty}L_p'(s_k),
	\end{align*}
	which implies that $\lim_{k\to\infty}L_p'(t_k)=\lim_{k\to\infty}L_p'(s_k)$. From the convexity of $L_p$ again, for any $a,b$ satisfying $t_k<a<t<b<s_k$,
	\begin{align*}
	L_p'(t_k)\leq  \frac{L_p(t)-L_p(a)}{t-a}\leq \frac{L_p(b)-L_p(t)}{b-t}\leq L_p'(s_k).
	\end{align*}
	Consequently, we obtain the differentiability of $L_p$ at $t$ and this completes the proof of \eqref{add:lem6:eq1}. 
	
	To show \eqref{add:lem6:eq2}, note that $L_p'(t)\leq 0$  since $L_p$ is nonincreasing. If $L_p'(t)=0,$ \eqref{add:lem6:eq1} implies that $L_p(t)=0.$ Let $\lambda_1$ be the top eigenvalue of $\bar G_n$ associated with the eigenvector $x$ with $\|x\|_2=\sqrt{n\delta}$ for some $0<\delta<1$. We can express $x$ as $x/\|x\|_2=g/\|g\|_2$ for a $n$-dimensional standard Gaussian vector $g$. Thus using the  law of large number and the fact that $\lambda_1/\sqrt{n}\to 2 $ almost surely that we have almost sure lower bound
	\begin{align*}
	L_{n,p}(t)&\geq \frac{1}{n^{3/2}}\la G_nx,x\ra-t\opnorm{x}_p^p=\frac{1}{\sqrt{2}n^{3/2}}\la \bar G_nx,x\ra-t\opnorm{x}_p^p\to \sqrt{2}\delta-t\delta^{p/2}\e|z|^p,
	\end{align*}
	where $z\thicksim N(0,1)$. The lower bound is strictly positive if $\delta$ is taken to be small enough, contradicting  $L_p(t)=0$. Hence, $L_p'(t)<0.$
\end{proof}

Now we turn to the proof of Theorem \ref{thm2}. For a given realization of $(G_n)_{n\geq 1}$ for which $\lim_{n\to\infty}L_{n,p}(t)$ converges to $L_p(t)$ for all $t\in (0,\infty)$, let $\Omega\subset (0,\infty)$ be the collection of all points $t$, on which $L_{n,p}(t)$ is differentiable for all $n\geq 1.$  Since $L_p$ is differentiable everywhere, by Griffiths’ lemma, we have that  $\lim_{n\to\infty}L_{n,p}'(t)=L_p'(t)$ at an arbitrary $t\in \Omega$, which, together with  \eqref{add:lem6:eq2}, implies that $L_{n,p}'(t) < 0$ as long as $n$ is large enough. Thus, we can pass to the limit in \eqref{add:lem5:eq0} to get that
\begin{align*}
\lim_{n\to\infty}\gp_{n,p}=\frac{p}{2}\Bigl(\frac{p}{2}-1\Bigr)^{2/p-1}t^{2/p}L_{p}(t)^{1-2/p}.
\end{align*} 
Since this is valid for all $t\in \Omega$ and this set is dense in $(0,\infty)$ since $L_{n,p}$ is convex for all $n\geq 1$, this equation must also be valid on $(0,\infty)$ by using the continuity of $L_p$ and this completes the proof of \eqref{thm2:eq2}.
 
Next, we are going to prove \eqref{thm2:eq3}. Assume that $x_*$ is an optimization of $\sup_{\|x\|_p=1}\la G_nx,x\ra$. If we let $x_0=n^{1/p}x_*$, then $x_0$ is an optimizer of  $\sup_{\opnorm{x}_p=1}\la G_nx,x\ra.$ Note that as will be argued later in the proof of Lemma~\ref{add:lem-1}, $x_0$ must satisfy the inequality \eqref{add:eq--3} with $\delta=p/2-1.$ As a result,
\begin{align*}
\frac{1}{n}\|x_0\|_\infty^{p+\delta}\leq \frac{1}{n}\sum_{i\in [n]}|x_{0,i}|^{p+\delta}=\opnorm{x_0}_{p+\delta}^{p+\delta}&\leq \frac{1}{\sqrt{2}\gp_{n,p}}\frac{\|\bar G_n\|_2}{n^{1/2}}.
\end{align*}
From \eqref{extra:eq18} and \eqref{add:eq2}, we see that there exist constants $C,C'>0$ depending only on $p$ such that with probability at least $1-C'e^{-n/C'}$,
$
\|{x_0}\|_{\infty}\leq Cn^{1/(p+\delta)}.
$
This implies that $\|x_*\|_\infty\leq Cn^{1/(p+\delta)-1/p}.$ 

\section{Proof of Theorem \ref{thm3}} \label{sec9}

In this section, we establish the continuity of $\gp_p$ in $p\in [2,\infty].$ 
First we do a change of variable to obtain
\begin{align}\label{add:eq--4}
\gp_{n,p}&=\max_{\opnorm{x}_p=1}\frac{1}{n^{3/2}}\la G_nx,x\ra=\max_{\opnorm{x}_p\leq 1}\frac{1}{n^{3/2}}\la G_nx,x\ra.
\end{align}
The following lemma establishes the monotonicity of $\gp_p.$

\begin{lemma}
	\label{add:lem-0}
	For any $2\leq p\leq  p'\leq \infty$, we have that $\gp_p\geq \gp_{p'}$.
\end{lemma}

\begin{proof}
	In view of \eqref{thm2:eq2} and \eqref{add:eq--4}, the assertion follows immediately from the monotonicity of the averaged norm $\opnorm{x}_p$ in $p$ due to Jensen's inequality.
\end{proof}

We divide our discussion into three cases: $p=2,$ $2<p<\infty$, and $p=\infty.$

\subsection{Case $p=2$}
Note that $\gp_2\geq \gp_p$ for all $p\geq 2.$ It suffices to show that $\gp_2\leq \lim_{p\downarrow 2}\gp_p.$ 
Write
\begin{align*}
\gp_{n,2}&=\max_{\opnorm{x}_2=1}\frac{1}{n^{3/2}}\la G_nx,x\ra=\frac{1}{\sqrt{2}}\max_{\opnorm{x}_2=1}\frac{1}{n^{3/2}}\la \bar G_nx,x\ra.
\end{align*}
Here the maximum on the right-hand side is the largest eigenvalue of the GOE, $\bar G_n/\sqrt{n}$, and the corresponding eigenvector $y$ with $\opnorm{y}_2=1$ has the distribution $g/\opnorm{g}_2$ for $g$ an $n$-dimensional standard Gaussian vector. For all $p\geq 2$, since $\lim_{n\to\infty}\opnorm{y}_p^p= \e |z|^p$ almost surely for $z\thicksim N(0,1)$ by the strong law of large number. Therefore,  for any $\varepsilon>0,$ 
\begin{align*}
\gp_{2}&=\lim_{n\to\infty}\frac{1}{\sqrt{2}}\frac{1}{n^{3/2}}\la \bar G_ny,y\ra\\
&\leq\frac{1}{\sqrt{2}}\limsup_{n\to\infty}\max_{|\opnorm{x}_p^p-\e |z|^p|<\varepsilon }\frac{1}{n^{3/2}}\Bigl\la\bar G_n\Bigl(\frac{x}{\opnorm{x}_p}\Bigr),\Bigl(\frac{x}{\opnorm{x}_p}\Bigr)\Bigr\ra \opnorm{x}_p^2\\
&\leq \frac{1}{\sqrt{2}}\bigl(\e |z|^p+\varepsilon\bigr)^{2/p}\limsup_{n\to\infty}\max_{|\opnorm{x}_p^p-\e |z|^p|<\varepsilon }\frac{1}{n^{3/2}}\Bigl\la\bar G_n\Bigl(\frac{x}{\opnorm{x}_p}\Bigr),\Bigl(\frac{x}{\opnorm{x}_p}\Bigr)\Bigr\ra \\
&\leq \bigl(\e |z|^p+\varepsilon\bigr)^{2/p}\gp_p.
\end{align*} 
Hence, $\gp_2\leq \lim_{p\downarrow  2}\gp_p$ and this completes our proof.

\subsection{Case $2<p<\infty$}

To discuss this case, we need a key lemma, which states that for all $\delta>0$ with $2(1+\delta)<p,$ we can essentially rewrite $\gp_{n,p}$ as a maximization problem restricted to $x\in \mathbb{R}^n$ satisfying $\opnorm{x}_p=1$ and $\opnorm{x}_{p+\delta}^{p+\delta}\leq \gp_2/\gp_\infty$. In other words, this means that the optimizer of $\gp_{n,p}$ has a bounded norm of slightly higher order than $p.$

\begin{lemma}\label{add:lem-1}
	Let $p\in (2,\infty).$ For any $\delta>0$ satisfying $2(1+\delta)\leq p$, we have that
	\begin{align*}
	\gp_p&=\lim_{n\to\infty}\max_{x\in A_{n,p,\delta}}\frac{1}{n^{3/2}}\la G_nx,x\ra,
	\end{align*}
	where $A_{n,p,\delta}$ is the collection of all $x\in \mathbb{R}^n$ satisfying that $\opnorm{x}_p=1$ and $\opnorm{x}_{p+\delta}^{p+\delta}\leq \gp_2/\gp_\infty.$ 
\end{lemma}

\begin{proof}
	Assume that $x_0=(x_{0,1},\ldots,x_{0,n})\in \mathbb{R}^n$ with $\opnorm{x_0}_p=1$ is an optimizer of the first maximization problem in \eqref{add:eq--4}. Notice that $\sum_{i\in [n]}|x_i|^p$ is twice differentiable since $p>2.$
	Using the Lagrange multiplier, there exists some $c\in \mathbb{R}$ such that for all $i \in [n]$, 
	\begin{align}\label{add:eq--2}
	\frac{cp}{n}|x_{0,i}|^{p-1}\mbox{sgn}(x_{0,i})&=\frac{1}{n^{3/2}}\bigl( G_nx_0\bigr)_i+\frac{1}{n^{3/2}}\bigl( G_n^{T}x_0\bigr)_i=\frac{\sqrt{2}}{n^{3/2}}\bigl(\bar G_nx_0\bigr)_i.
	\end{align}
	Multiplying $x_{0,i}$ to both sides and adding them together give
	\begin{align*}
	cp=cp\opnorm{x_0}_p^p=\frac{\sqrt{2}}{n^{3/2}}\la \bar G_nx_0,x_0\ra=2\max_{\opnorm{x}_p=1}\frac{1}{n^{3/2}}\la G_nx,x\ra=2\gp_{n,p}.
	\end{align*}
	Hence, $c=2p^{-1}\gp_{n,p}$. On the other hand, for any $\delta>0$, we can multiply \eqref{add:eq--2} by $|x_{0,i}|^{1+\delta}\mbox{sgn}(x_{0,i})$ and use the Cauchy-Schwarz inequality to get
	\begin{align*}
	\begin{split}
	cp\opnorm{x_0}_{p+\delta}^{p+\delta}&=\frac{\sqrt{2}}{n^{1/2}}\frac{1}{n}\sum_{i\in [n]}\bigl(\bar G_nx_0\bigr)_i|x_{0,i}|^{1+\delta}\mbox{sgn}(x_{0,i})
	\end{split}\\
	\begin{split}\notag
	&\leq \frac{\sqrt{2}}{n^{1/2}}\opnorm{\bar G_nx_0}_2\opnorm{x_0}_{2(1+\delta)}^{1+\delta}
	\leq \sqrt{2}\frac{\|\bar G_n\|_2}{n^{1/2}}\opnorm{x_0}_2\opnorm{x_0}_{2(1+\delta)}^{1+\delta}.
	\end{split}
	\end{align*}
	If now $\delta$ satisfies $2(1+\delta)\leq p,$ then  Jensen's inequality implies
	\begin{align*}
	cp\opnorm{x_0}_{p+\delta}^{p+\delta}&\leq \sqrt{2}\frac{\|\bar G_n\|_2}{n^{1/2}}\opnorm{x_0}_p\opnorm{x_0}_p^{1+\delta}=\sqrt{2}\frac{\|\bar G_n\|_2}{n^{1/2}}.
	\end{align*}
	Hence,
	\begin{align}\label{add:eq--3}
	\opnorm{x_0}_{p+\delta}^{p+\delta}&\leq \frac{1}{\sqrt{2}\gp_{n,p}}\frac{\|\bar G_n\|_2}{n^{1/2}}.
	\end{align}
	Since $\gp_{n,p}$ converges to $\gp_p$ almost surely and $\|\bar G_n\|_2/\sqrt{n}$ converges to $2$ almost surely, it follows that
	\begin{align*}
	\gp_p=\lim_{n\to\infty}\gp_{n,p}&=\lim_{n\to\infty}\max\frac{1}{n^{3/2}}\la G_nx,x\ra,
	\end{align*}
	where the maximum is taken over all $x\in \mathbb{R}^n$ satisfying that $\opnorm{x}_p=1$ and $\opnorm{x}_{p+\delta}^{p+\delta}\leq  \sqrt{2}/\gp_p$. Our proof is then completed by noting that $\gp_2=\sqrt{2}$ and $\gp_p\geq \gp_\infty.$
\end{proof}

We now proceed to verify the continuity of $\gp_p$ on $(2,\infty)$.  
For any $2<p_1<p_2<p_3<\infty$, set $\theta =(p_2-p_1)/(p_3-p_1)\in (0,1),$ which yields $p_2=(1-\theta)p_1+\theta p_3$.
From the H\"older inequality,
$
\opnorm{x}_{p_2}^{p_2}\leq \opnorm{x}_{p_1}^{(1-\theta)p_1}\opnorm{x}_{p_3}^{\theta p_3}.
$
Note that if $x\in \mathbb{R}^n$ satisfies that $\opnorm{x}_{p_1}=1$ and $\opnorm{x}_{p_3}^{p_3}\leq \gp_{2}/\gp_\infty$, then 
$\opnorm{x}_{p_2}^{p_2}\leq (\gp_{2}/\gp_{\infty})^{\theta}.$
Consequently, if $p_1,p_2,p_3$ also satisfy  $2(1+p_3-p_1)<p_1,$ then by applying Lemma~\ref{add:lem-1} with $p=p_1$ and $\delta=p_3-p_1$,
\begin{align*}
\gp_{p_1}&\leq \limsup_{n\to\infty}\max_{\opnorm{x}_{p_2}^{p_2}\leq \bigl(\frac{\gp_{2}}{\gp_{\infty}}\bigr)^{\theta}}\frac{1}{n^{3/2}}\la G_nx,x\ra\\
&=\limsup_{n\to\infty}\max_{\opnorm{x}_{p_2}^{p_2}\leq \bigl(\frac{\gp_{2}}{\gp_{\infty}}\bigr)^{\theta}}\frac{1}{n^{3/2}}\Bigl\la G_n\frac{x}{\opnorm{x}_{p_2}},\frac{x}{\opnorm{x}_{p_2}}\Bigr\ra \opnorm{x}_{p_2}^2\\
&\leq \Bigl(\frac{\gp_{2}}{\gp_{\infty}}\Bigr)^{2\theta /p_2}\gp_{p_2}=\Bigl(\frac{\gp_{2}}{\gp_{\infty}}\Bigr)^{\frac{2(p_2-p_1)}{p_2(p_3-p_1)}}\gp_{p_2}.
\end{align*} 
This and Lemma \ref{add:lem-0} imply
\begin{align*}
0\leq \gp_{p_1}-\gp_{p_2}\leq \Bigl(\Bigl(\frac{\gp_{2}}{\gp_{\infty}}\Bigr)^{\frac{2(p_2-p_1)}{p_2(p_3-p_1)}}-1\Bigr)\gp_{p_2}.
\end{align*}
This clearly validates the continuity of $\gp_p$ on $(2,\infty).$

\subsection{Case $p=\infty.$} Note that $\gp_\infty\leq \gp_p$ for all $p\geq 2.$ From Lemma \ref{add:lem-0}, it suffices to show that $\lim_{p\to\infty}\gp_p\leq \gp_\infty.$ For any $\varepsilon>0$ and $x\in \mathbb{R}^n$ satisfying $\opnorm{x}_p=1,$ denote $x_i^\varepsilon=x_i 1(|x_i|\geq 1+\varepsilon)$ and $\widetilde x_i^{\varepsilon}=x_i 1(|x_i|<1+\varepsilon).$ An application of the H\"older inequality  followed by the Markov inequality yields 
\begin{align*}
\opnorm{x^\varepsilon}_2^2&\leq \opnorm{x}_p^2\Bigl(\frac{1}{n}\sum_{i\in [n]}1(|x_i|\geq 1+\varepsilon)\Bigr)^{1-2/p}=\Bigl(\frac{1}{n}\sum_{i\in [n]}1(|x_i|\geq 1+\varepsilon)\Bigr)^{1-2/p}\\
&\leq \Bigl(\frac{1}{(1+\varepsilon)^{p}}\opnorm{x}_p^p\Bigr)^{1-2/p}=\frac{1}{(1+\varepsilon)^{p-2}}.
\end{align*}
Therefore, 
\begin{align*}
&\Bigl|\frac{1}{n^{3/2}}\la G_nx^\varepsilon,x^\varepsilon\ra\Bigr|\leq \frac{\| G_n\|_2}{\sqrt{n}}\opnorm{x^\varepsilon}_2^2\leq \frac{\|G_n\|_2}{\sqrt{n}(1+\varepsilon)^{p-2}}
\end{align*}
and
\begin{align*}
&\Bigl|\frac{1}{n^{3/2}}\la G_nx^\varepsilon,\widetilde{x}^\varepsilon\ra\Bigr|
\leq   \frac{\|G_n\|_2}{\sqrt{n}}\opnorm{x^\varepsilon}_2\opnorm{\widetilde x^\varepsilon}_2\leq \frac{\|G_n\|_2}{\sqrt{n}(1+\varepsilon)^{p/2-1}}.
\end{align*}
From \eqref{add:eq2}, it follows that 
\begin{align*}
\lim_{p\to\infty}\gp_p&=\lim_{p\to\infty}\lim_{n\to\infty}\max_{\opnorm{x}_p=1}\frac{1}{n^{3/2}}\bigl(\la G_nx^\varepsilon,x^\varepsilon\ra+2\la   G_nx^\varepsilon,\widetilde x^\varepsilon\ra+\la  G_n\widetilde x^\varepsilon,\tilde x^\varepsilon\ra\bigr)\\
&=\lim_{p\to\infty}\lim_{n\to\infty}\max_{\opnorm{x}_p=1}\frac{1}{n^{3/2}}\la G_n\widetilde x^\varepsilon,\widetilde x^\varepsilon\ra\\
&\leq \lim_{n\to\infty}\max_{\|x\|_\infty\leq 1+\varepsilon}\frac{(1+\varepsilon)^2}{n^{3/2}}\la G_nx(1+\varepsilon)^{-1},x(1+\varepsilon)^{-1}\ra\leq (1+\varepsilon)^{2}\gp_\infty.
\end{align*}
Letting  $\varepsilon \to 0+,$ we obtain that $\lim_{p\to\infty}\gp_p\leq \gp_\infty.$ This completes our proof.

\section{Bounding $L_{p,u}(t)$ via free energies}\label{sec10}

For $2<p<\infty$ and $u>0,$ recall $L_{n,p,u}(t)$ from \eqref{Lnput} and that from Theorem \ref{thm-1}, its $n$-limit exists and is denoted by $L_{p,u}(t)$. This section is a preparation for the proof of Theorem \ref{thm0}.  We establish upper and lower bounds for $L_{p,u}(t)$ in terms of two free energies that we define now.
Let $\varepsilon>0.$ For $\beta>0,$ define the free energy associated with $H_{n,p,t}$ with $\ell_2$-band constraint by
\begin{align*}
F_{n,t,u,\varepsilon}(\beta)&=\frac{1}{\beta n}\log \int_{x\in \mathbb{R}^n:\opnorm{x}_2^2\in [u-\varepsilon,u+\varepsilon]}\exp \beta H_{n,p,t}(x)dx,
\end{align*}
where $dx$ is the Lebesgue measure on $\mathbb{R}^n$. Similarly, we shall also consider an analogous version of $F_{n,t,u,\varepsilon}(\beta)$ with an additional box confinement, namely, for $u^{1/2}<M<\infty$ and $\beta>0,$ define
\begin{align}\label{freeenergy2}
F_{n,t,u,M,\varepsilon}(\beta)&=\frac{1}{\beta n}\log \int_{x\in [-M,M]^n:\opnorm{x}_2^2\in [u-\varepsilon,u+\varepsilon]}\exp \beta H_{n,p,t}(x)dx.
\end{align}
The parameter $\beta$ in the above two free energies is usually called the (inverse) temperature in statistical physics.  The following proposition gives upper and lower bounds for $L_{p,u}(t)$.

\begin{proposition}
	\label{extra:prop1}
	Let $2<p<\infty$. For any $u>0$ and $t>0,$
	\begin{align*}
	\lim_{M\to\infty}\liminf_{\beta\to\infty}\lim_{\varepsilon\downarrow 0}\liminf_{n\to\infty}\e F_{n,t,u,M,\varepsilon}(\beta)\leq L_{p,u}(t)\leq  \lim_{\beta \to\infty}\limsup_{n\to\infty}\e F_{n,t,u,\beta^{-2}}(\beta).
	\end{align*}
\end{proposition}

\begin{remark}\rm
We emphasize that the outer limit in the upper bound is along $\varepsilon=\beta^{-2}.$
 To use these inequalities, the highly nontrivial part is to find the tight limits on both sides. In Section \ref{sec11}, we will show that the lower bound can be handled by a Parisi-type variational formula similar to the one in Theorem~\ref{thm0}. The upper bound will be controlled by the so-called Guerra Replica Symmetry Breaking bound in Section \ref{12.7}. Ultimately, we show that these match each other, leading to the desired formula in \eqref{thm0:eq1}, see Section \ref{sec12}.
\end{remark}

For the rest of this section, we establish the proof of Proposition \ref{extra:prop1}.
Consider a modified version of $L_{n,p,u}(t)$ on a $\ell_2$-band, namely, for $\varepsilon>0$ and $t>0$, define
\begin{align*}
L_{n,p,u,\varepsilon}(t)&:=\max_{x\in \mathbb{R}^n:\opnorm{x}_2^2\in[u-\varepsilon,u+\varepsilon]}\frac{H_{n,p,t}(x)}{n}=\max_{x\in \mathbb{R}^n:\opnorm{x}_2^2\in[u-\varepsilon,u+\varepsilon]}\Bigl(\frac{\la G_nx,x\ra}{n^{3/2}}-t\opnorm{x}_p^p\Bigr).
\end{align*}
When $\varepsilon$ is small, this quantity is essentially the same as $L_{n,p,u}(t)$. The following lemma establishes upper and lower bounds for $L_{n,p,u,\varepsilon}(t)$ in terms of the free energies introduced above.

\begin{lemma}\label{lem2} 
	Let $2<p<\infty$. For any $0<u< \sqrt{M}$ and $t>0,$
	\begin{align}\label{lem2:eq1}
	&\liminf_{\beta\to\infty}\lim_{\varepsilon\downarrow 0}\lim_{n\to\infty}\e F_{n,t,u,M,\varepsilon}(\beta)\leq \lim_{\varepsilon\downarrow 0}\liminf_{n\to\infty}\e L_{n,p,u,\varepsilon}(t).
	\end{align}
	In addition, there exist some constants $C,\kappa>0$ (depending on $t,u$) such that for any $\beta>0$, $0<\varepsilon<1$, and $n\geq 1,$
	\begin{align}\label{lem2:eq2}
	\e F_{n,t,u,\varepsilon}(\beta)&\geq \e L_{n,p,u,\varepsilon/\kappa}(t)-\varepsilon C+\frac{1}{\beta}\log \frac{2\varepsilon}{\kappa}-Ce^{-n/C}.
	\end{align}
\end{lemma}

\begin{proof}
  It is  easy to see that
	$
	F_{n,t,u,M,\varepsilon}(\beta)\leq L_{n,p,u,\varepsilon}(t)+\beta^{-1}\log 2M,
	$
	which immediately implies   the upper bound \eqref{lem2:eq1}.
	So, let us focus on the lower bound. Let $0<\varepsilon<1$. Note that we can choose $\kappa>\max(1/u,1)$ sufficiently large such that for any $x^*$ satisfying $|\opnorm{x^*}_2^2-u|<\varepsilon/\kappa,$ we have 
		$$
	 B(x^*):=x^* + [ -\varepsilon/\kappa,  \varepsilon/\kappa]^n \subset \{x\in \mathbb{R}^n:|\opnorm{x}_2^2-u|<\varepsilon\}.
	$$
	Note that $\varepsilon/\kappa<1.$  For any $x\in B(x^*)$,  a direct computation using triangle inequalities shows that 
	\begin{align} \label{add:eq-5}
H_{n,p,t}(x)-H_{n,p,t}(x^*)
	&\geq -n^{-1/2}\| G_n\|_2\|x-x^*\|_2(\|x^*\|_2+\|x\|_2)-t\bigl((\|x-x^*\|_p+\|x^*\|_p)^p-\|x^*\|_p^p\bigr) \nonumber \\
	&\geq -n^{1/2}\| G_n\|_2\opnorm{x-x^*}_2(\opnorm{x^*}_2+\opnorm{x}_2)-tp n \opnorm{ x- x^*}_p  \bigl (\opnorm{ x- x^*}_p + \opnorm{x^*}_p \bigr)^{p-1} \nonumber \\
&\geq -\frac{n\varepsilon}{\kappa}\Bigl(2n^{-1/2}\| G_n\|_2\sqrt{u+1}+tp  (1 + \opnorm{x^*}_p )^{p-1} \Bigr),
\end{align}
	where  we have used the inequality $(a+b)^p - a^p  = \int_a^{a+b} p x^{p-1} \le p b (a+b)^{p-1} $ for any $a, b \ge 0$ in the second step above. 
	From now on, we take $x^*$ as a maximizer of $H_{n,p,t}$ on $\{x\in \mathbb{R}^n:|\opnorm{x}_2^2-u|<\varepsilon/\kappa\}$. By imitating the proof as that of \eqref{add:E2}, it can be verified that there exist constants $C,D>0$ such that 
	\begin{align}\label{con}
	\p\bigl(\opnorm{x^*}_p\leq D\bigr)\geq 1-Ce^{-n/C}.
	\end{align} 
	From \eqref{add:eq2}, \eqref{add:eq-5}, and \eqref{con}, there exists a constant $C(t, u)$ such that with probability at least $1-C'e^{-n/C'}$, we have for any $x\in B(x^*),$
	\begin{align*}
	H_{n,p,t}(x)-H_{n,p,t}(x^*)&\geq -\frac{n\varepsilon }{\kappa}C(t,u).
	\end{align*}
	Let $A$ be the event such that this inequality is valid. From the above inequality and noting that $\mathrm{Vol}(B(x^*))=(2\varepsilon/\kappa)^n,$ we have that
	\begin{align}
	\begin{split}\notag
	\e F_{n,t,u,\varepsilon}(\beta)&\geq  \frac{1}{n\beta}\e\Bigl[\log \int_{B(x^*)}e^{n\beta H_{n,p,t}(x)}dx;A\Bigr]\\
	&\geq \e \bigl[L_{n,p,u,\varepsilon/\kappa}(t);A\bigr]-\frac{\varepsilon }{\kappa}C(t,u)+\frac{1}{\beta}\log \frac{2\varepsilon}{\kappa}
	\end{split}\\
	\begin{split}\label{add:eq-6}
	&=\e L_{n,u,p,\varepsilon/\kappa}(t)-\e \bigl[L_{n,p,u,\varepsilon/\kappa}(t);A^c\bigr]-\frac{\varepsilon }{\kappa}C(t,u)+\frac{1}{\beta}\log \frac{2\varepsilon}{\kappa}.
	\end{split}
	\end{align}
	To control $\e \bigl[L_{n,p,u,\varepsilon/\kappa}(t);A^c\bigr]$, we argue as follows. Recall that if $x$ is an optimizer of $L_{n,p}(t)$, then it must satisfy the second inequality of \eqref{extra:eq20}. Using this and dropping the negative term in $L_{n,p}(t)$ yield that 
	\begin{align*}
	L_{n,p}(t)&\leq \frac{1}{\sqrt{n}}\|G_n\|_2\opnorm{x}_2^2\leq t^{-2/(p-2)}\Bigl(\frac{\|G_n\|_2}{\sqrt{n}}\Bigr)^{1+2/(p-2)}.
	\end{align*}
	Here, applying the concentration inequality \eqref{eq:Gconc} to $\|G_n\|_2$ and using \eqref{extra:eq9} leads to $\e (n^{-1/2}\|G_n\|_2)^k\leq C_k$ for all $k\geq 1$. Thus, there exists $C(t)$ independent of $n$ such that
	\begin{align}\label{add:lem2:proof:eq1}
	\e L_{n,p}(t)^2&\leq C(t),
	\end{align}
	which implies that
	\begin{align*}
	\bigl|\e \bigl[L_{n,p,u,\varepsilon/\kappa}(t);A^c\bigr]\bigr|&\leq \bigl(\e L_{n,p,u,\varepsilon/\kappa}(t)^2\bigr)^{1/2}\p(A^c)^{1/2}\\
	&\leq \bigl(\e L_{n,p}(t)^2\bigr)^{1/2}\bigl(C'e^{-n/C'}\bigr)^{1/2}
	\leq \sqrt{C(t)}\bigl(C'e^{-n/C'}\bigr)^{1/2}.
	\end{align*}
	The above inequality  and \eqref{add:eq-6} yield the second assertion.
\end{proof}

Next we show that $L_{n,p,u}(t)$ and $L_{n,p,u,\varepsilon}(t)$ are asymptotically the same.

\begin{lemma}
	\label{add:lem--1}
	Let $2<p<\infty.$ For any $t,u>0$, we have that almost surely, 
	\begin{align*}
\lim_{\varepsilon\rightarrow 0}\lim_{n\to\infty}L_{n,p,u,\varepsilon}(t)=\lim_{\varepsilon\rightarrow 0}\lim_{n\to\infty}\e L_{n,p,u,\varepsilon}(t)=L_{p,u}(t).
	\end{align*}
\end{lemma}
\begin{proof}
	The first limit follows directly from the Gaussian concentration inequality \eqref{eq:Gconc} as  $L_{n,p,u,\varepsilon}(t)$ is $(u+\varepsilon)n^{-1/2}$-Lipschitz in $G_n$ with respect to the $\|\cdot\|_F$-norm. Thus it remains to prove the second equality.
	Write
	\begin{align*}
	L_{n,p,u,\varepsilon}(t)&=\max_{v:|v-u|\leq \varepsilon}\max_{\opnorm{x}_2^2=1}\Bigl(\frac{v}{n^{3/2}}\la G_nx,x\ra-tv^{p/2}\opnorm{x}_p^p\Bigr),\\
	L_{n,p,u}(t)&=\max_{\opnorm{x}_2^2=1}\Bigl(\frac{u}{n^{3/2}}\la G_nx,x\ra-tu^{p/2}\opnorm{x}_p^p\Bigr).
	\end{align*}
	If $x_v$ is an optimizer to the inner optimization problem of $L_{n,p,u,\varepsilon}(t)$ associated to $v$, then $$
	t\opnorm{x_v}_p^p\leq v^{1 - p/2}  n^{-1/2}\|G_n\|_2\opnorm{x_v}_2^2=v^{1 - p/2} n^{-1/2}\|G_n\|_2$$ since $L_{n,p,u,\varepsilon}(t)\geq 0$. Consequently, for any $0<\varepsilon<\min(1,u),$
	\begin{align*}
	L_{n,p,u,\varepsilon}(t)&\leq L_{n,p,u}(t)+\varepsilon\bigl(n^{-1/2}\|G_n\|_2+2^{-1}tp(u+\varepsilon)^{p/2-1}\max_{v:|v-u|\leq \varepsilon}\opnorm{x_v}_p^p\bigr)\\
	&\leq L_{n,p,u}(t)+\varepsilon\bigl(n^{-1/2}\|G_n\|_2+2^{-1}pn^{-1/2}\|G_n\|_2(u+\varepsilon)^{p/2-1} (u - \varepsilon)^{1 - p/2}\bigr).
	\end{align*}
	By Passing to limit, we obtain that $\lim_{\varepsilon\to 0}\limsup_{n\to\infty}L_{n,p,u,\varepsilon}(t)\leq L_{p,u}(t).$ Since obviously we also have $L_{n,p,u}(t)\leq L_{n,p,u,\varepsilon}(t)$, it follows that $\lim_{\varepsilon\to 0}\lim_{n\to\infty}L_{n,p,u,\varepsilon}(t)=\lim_{n\to\infty}L_{n,p,u}(t)=L_{p,u}(t)$ and this completes our proof.
	
\end{proof}

To complete the proof of Proposition \ref{extra:prop1}, the lower bound follows directly from \eqref{lem2:eq1} and Lemma \ref{add:lem--1}. As for the upper bound, it can be obtained by sending $n\to\infty$ and then $\varepsilon=\beta^{-2}\to 0$ in \eqref{lem2:eq2} and using again Lemma \ref{add:lem--1}.

\section{Temperature limit of the Parisi formula}\label{sec11}

By Proposition \ref{extra:prop1}, the validity of Theorem \ref{thm0} relies on showing that  $\lim_{\beta \to\infty}\limsup_{n\to\infty}\e F_{n,t,u,\beta^{-2}}(\beta)$ and $\lim_{M\to\infty}\lim_{\beta\to\infty}\lim_{\varepsilon\downarrow 0}\liminf_{n\to\infty}\e F_{n,t,u,M,\varepsilon}(\beta)$ are equal, together with establishing the desired formula in \eqref{thm0:eq1}. In this section, we shall focus on handling the limit
\begin{align}\label{sec11:eq1}
\lim_{\beta\to\infty}\lim_{\varepsilon\downarrow 0}\liminf_{n\to\infty}F_{n,t,u,M,\varepsilon}(\beta).
\end{align}
 To this end, we shall need several Parisi-type variational formulas associated to a variety of the free energies and the ground state energies. For the convenience of the reader, we collect all of them together  in Section~\ref{sec11.1}, including those that will be used later on. Our main result, Proposition  \ref{add:prop1}, establishing a Parisi-type formula for \eqref{sec11:eq1}, is stated in Section \ref{sec11.2} followed by its proof in the rest of this section.

\subsection{Parisi PDEs and functionals}\label{sec11.1}

 Fix $u > 0$.  
 Let $\mathcal{N}_u$ be  the collection of all positive measures $\gamma$ on $[0,u]$ with $\int_0^u \gamma(s)ds<\infty$.  Recall that $\mathcal{N}^d_u \subset \mathcal{N}_u$ is the collection of  all positive measures on $[0, u]$ with finitely many atoms.   Moreover, let  $\mathcal{M}_u$ denote the collection of probability measures on $[0, u]$ and let  $\mathcal{M}^d _u \subset \mathcal{M}_u$ be the collection of probability measures on $[0, u]$ with finitely many atoms. For each $\gl \in \r$, let $h_\gl:\r \to \r_+$  be a continuous function.

For $\beta<\infty,$ we take $(\gl,\alpha) \in \r\times \mathcal{M}_u^d$ or $ \r\times \mathcal{M}_u$ and consider the solution $\Theta^\beta_{\lambda, \alpha }$ to the following PDE,
\begin{align}\label{pdeeq2}
\partial_{s} \Theta^\beta_{\lambda, \alpha }(s, x) &= - \bigl(\partial_{xx} \Theta^\beta_{\lambda, \alpha }(s, x) +\beta \alpha(s)\bigl(\partial_x \Theta^\beta_{\lambda, \alpha }(s, x)  \bigr)^2\bigr), \quad (s,x)\in [0,u)\times \mathbb{R},&   \\
\Theta^\beta_{\lambda, \alpha }(u,x) &= h_{\lambda}(x). \nonumber
\end{align}
 Here  $\Theta_{\lambda,\alpha}^\beta$ may exist in the classical or weak sense (see Remarks \ref{rmk3} and \ref{rmk2} below) depending on whether the measure $\alpha$ is atomic and on the regularity of the boundary condition $h_\gl$. 
The associated Parisi functional $\mathcal{P}_\beta$ on $\r \times \mathcal{M}^d_u$ or $\r \times \mathcal{M}_u$ is defined as
\begin{align}\label{parisieq2}
\mathcal{P}_\beta(\lambda,\alpha)=\Theta^\beta_{\lambda, \alpha }(0,0)-\lambda u-\int_0^u s\beta  \alpha(s)ds,
\end{align}
where  $\alpha(s) : = \alpha([0, s])$ is the cumulative distribution function of the measure $\alpha$. 

In a similar manner, the functional \eqref{parisieq2} has a natural analogy in the case of $\beta=\infty.$ 
For $(\gl,\gamma) \in \r\times\mathcal{N}^d_u$ or $\r\times\mathcal{N}_u$, 
consider the (classical or weak )solution $\Theta_{\lambda,\gamma}$ to the following PDE, 
  \begin{align}\label{pdeeq1}
\partial_{s} \Theta_{\lambda,\gamma}(s, x) &= - \bigl(\partial_{xx} \Theta_{\lambda,\gamma}(s, x)  +\gamma(s)\bigl(\partial_x \Theta_{\lambda,\gamma}(s, x)  \bigr)^2\bigr), \quad (s,x)\in [0,u)\times \mathbb{R},&   \\
   \Theta_{\lambda,\gamma}(u,x) &= h_{\lambda}(x),\nonumber
\end{align}
where  $\gamma(s) : = \gamma([0, s])$. Define the associated Parisi functional $\mathcal{P}$  on   $ \r \times \mathcal{N}_u$  or $ \r \times \mathcal{N}^d_u$ as
\begin{align}\label{parisieq1}
\mathcal{P}(\lambda,\gamma)=\Theta_{\lambda,\gamma}(0,0)-\lambda u-\int_0^u s\gamma (s)ds.
\end{align}

%

The different boundary conditions $h_\lambda$ that we will  use throughout this paper are all dependent on a collection of functions $(\omega_{\gl}^t)_{(\gl,t)\in \mathbb{R}\times (0,\infty)}$ defined as $\omega^t_\gl(x, r)  =  rx+\lambda r^2-t|r|^p$ for $(x,r)\in \mathbb{R}\times \mathbb{R}.$ The following table collects the relevant cases of the Parisi PDEs and functionals that will be needed for the rest of the paper.
\begin{table}[ht]
\begin{center}
 \begin{tabular}{|c|c|c|c|c|c|} 
  \hline
 
 $\beta$ & measure space &  boundary condition $h_\gl$ &  PDE solution  &    type & Parisi functional\\ 
 
 \hline

 $< \infty$ &  $ \alpha \in  \mathcal{M}^d _{u} $ & $f_\lambda^{\beta,t}(x)=\frac{1}{\beta} \log \int_{-\infty}^\infty  e^{ \beta  \omega^t_\gl(x, r)}  dr$ & $\Phi_{\lambda,\alpha}^{\beta,t,u}(s,x)$& classical  & $\mathcal{P}_{\beta,t,u}(\lambda,\alpha)$ \\
 \hline
 
$< \infty$ &  $\alpha \in  \mathcal{M} _{u} $ & $f_{\lambda}^{\beta,t,M}(x)=\frac{1}{\beta}\log \int_{-M}^M e^{ \beta\omega^t_\gl(x, r)}  dr$ & $\Phi_{\lambda,\alpha}^{\beta,t,u,M} (s, x)$ & weak& $\mathcal{P}_{\beta,t,u,M}(\lambda,\alpha)$  \\ 
 \hline
 $= \infty$ & $\gamma \in \mathcal{N}_u^d$ & $f_{\lambda}^t(x)=\sup_{r\in \mathbb{R} }\omega^t_\gl(r, x)$  & $ \Psi_{\lambda,\gamma}^{t,u}(s, x)$ & classical& $\mathcal{P}_{t,u}(\lambda,\gamma)$  \\
 \hline
 $= \infty$ &  $\gamma \in \mathcal{N}_u$ & $f_{\lambda}^{t,M}(x)=\sup_{r\in [-M,M]}\omega^t_\gl(r, x)$ & $ \Psi_{\lambda,\gamma}^{t,u, M}(s, x)$& weak & $\mathcal{P}_{t,u, M}(\lambda,\gamma)$   \\
 \hline
\end{tabular}
\caption{\label{tab:1} Various Parisi functionals.}
\end{center}
\end{table}
\vspace{-3mm}
\begin{remark}
	\label{rmk3}\rm Obviously $f_{\lambda}^{\beta,t,M}$ and $f_{\lambda}^{t,M}$ are uniformly bounded.  Lemma \ref{lem3} implies that $f_\lambda^{\beta,t}(x)$ and $f_{\lambda}^t(x)$ are bounded by $C(1+|x|^{a+1})$ for some $0<a<1$ and $C>0$. From these, when $\alpha\in \mathcal{M}_u^d$ and $\gamma\in \mathcal{N}_u^d$, we can apply the Hopf-Cole transformation (see Lemma \ref{add:lem2}) to solve the above four PDEs in the classical sense. 
\end{remark}

\begin{remark}\label{rmk2}\rm The functional $\mathcal{P}_{t,u}$ here is of course the same as \eqref{parisifunctional}, defined in the introduction.  In Table \ref{tab:1}, we define the functionals $\mathcal{P}_{t,u,M}$ and $\mathcal{P}_{t,u}$ on different measure spaces $\mathcal{N}_u$ and $\mathcal{N}_u^d$ respectively, mainly due to technical purposes. Indeed, the fact that $f_{\lambda}^{t,M}$ is $M$-Lipschitz allows us to show that for any $0<s_0<u$ and $k\geq 1,$ $\partial_x^k\Psi_{\lambda,\gamma}^{t,u,M}(s,\cdot)$ is also Lipschitz uniformly over all $s\in [0,s_0]$, $\lambda\in \mathbb{R}$, and $\gamma\in \mathcal{N}_u^d$.  Together with a compactness argument, this enables us to show the existence and uniqueness of the weak solution $\Psi_{\lambda,\gamma}^{t,u,M}$ for any $\gamma\in \mathcal{N}_u.$ We refer the reader to check \cite{CHL18,JS17} for details.   It might seem plausible that one can also construct unique weak solutions to $\Psi_{\lambda,\gamma}^{t,u}$ for arbitrary $\gamma\in \mathcal{N}_u$ by similar lines of arguments in \cite{CHL18,JS17}. However,  as the boundary condition $f_{\lambda}^t$ is no longer Lipschitz (see Lemma \ref{lem3}), this  makes the  compactness argument in \cite{CHL18,JS17} technically more tedious. As this part of the analysis is not quite relevant to this work, we do not pursue this extension in this paper.  For the same reason, $\mathcal{P}_{\beta,t,u,M}$ and $\mathcal{P}_{\beta,t,u}$ in Table~\ref{tab:1} are also defined over different measure spaces $\mathcal{M}_u$ and $\mathcal{M}_u^d.$
\end{remark}

\subsection{Parisi-type formula}\label{sec11.2}

Recall from \cite{Pan05} that $\lim_{\varepsilon\downarrow 0}\lim_{n\to\infty}F_{n,t,u,M,\varepsilon}(\beta)$ exists and it can be expressed as the following Paris-type formula associated with the Parisi functional given in the second row of Table~\ref{tab:1}
\begin{align}\label{Pf:freeenergy}
\lim_{\varepsilon\downarrow 0}\lim_{n\to\infty}F_{n,t,u,M,\varepsilon}(\beta)&=\inf_{(\lambda,\alpha)\in \mathbb{R}\times \mathcal{M}_u}\mathcal{P}_{\beta,t,u,M}(\lambda,\alpha).
\end{align}
The following proposition shows that the variational problem in \eqref{Pf:freeenergy} converges to the Parisi-type formula associated to  the functional $\mathcal{P}_{t,u,M}$ (refer to the fourth row of Table~\ref{tab:1}) and gives the existence of a minimizer along with quantitative controls uniformly in $M.$

\begin{proposition}\label{add:prop1} Let $2<p<\infty$ and $t,u>0.$ We have that
	\begin{align}\label{add:prop1:eq0}
		\lim_{\beta\to\infty}\lim_{\varepsilon\downarrow 0}\lim_{n\to\infty}F_{n,t,u,M,\varepsilon}(\beta)=\lim_{\beta\to\infty}\inf_{(\lambda,\alpha)\in \mathbb{R}\times \mathcal{M}_u}\mathcal{P}_{\beta,t,u,M}(\lambda,\alpha)=\inf_{(\lambda,\gamma)\in \mathbb{R}\times \mathcal{N}_u}\mathcal{P}_{t,u,M}(\lambda,\gamma).
	\end{align}
	In addition, there exists a constant $C>0$ depending only on $t,u$ such that for any $M>2u^{1/2}$, there exists some $(\lambda_M,\gamma_M)\in \mathbb{R}\times\mathcal{N}_u$ satisfying 
	\begin{align}\label{add:prop1:eq2}
	\inf_{(\lambda,\gamma)\in \mathbb{R}\times \mathcal{N}_u}\mathcal{P}_{t,u,M}(\lambda,\gamma)=\mathcal{P}_{t,u,M}(\lambda_M,\gamma_M).
	\end{align}
	such that 
	\begin{align}\label{add:prop1:eq1}
	|\lambda_M|\leq C\,\,\mbox{and}\,\,\int_0^u \gamma_M(s)ds\leq C
	\end{align}
	
\end{proposition}

\begin{remark}\rm  Though we do not need this here,  it can be checked by the same argument in Lemma \ref{lem2} that  $\lim_{n\to\infty}\max_{x\in[-M,M]^n:\opnorm{x}_2^2=u}H_{n,p,t}(x)/n$ is equal to the left-hand side of \eqref{add:prop1:eq0} and the right-hand side of \eqref{add:prop1:eq0} established a Parisi-type formula for this ground state energy. 
\end{remark}

\begin{remark}
	\rm The assertion \eqref{add:prop1:eq0} is indeed a special case of \cite[Theorem 5]{CP18}, but it does not include the existence of a minimizer and the bound \eqref{add:prop1:eq1}, which are the most crucial results we need in the proof of Theorem \ref{thm0}.
\end{remark}

The rest of this section is devoted to establishing  Proposition \ref{add:prop1}.

\subsection{Properties of the Parisi PDEs}

We state a number of fundamental properties for $\Phi_{\lambda,\alpha}^{\beta,t,u,M}$ and $\Psi_{\lambda,\gamma}^{t,u,M}$ that will be of great use throughout this section.  Their proofs will be omitted as they follow directly from the same arguments from \cite{AC15,JT16} with no essential changes. First of all,  their regularities are summarized in the following proposition.

\begin{proposition}
	[Regularity] \label{property0} Let $\alpha\in \mathcal{M}_u$ and $\gamma\in \mathcal{N}_u$. For any $s\in [0,u)$, $\Phi_{\lambda,\alpha}^{\beta,t,u,M}(s,\cdot)$ and $\Psi_{\lambda,\gamma}^{t,u,M}(s,\cdot)$ are  twice partially differentiable in $x$. Their first partial derivatives in $x$ are uniformly bounded over $[0,u)\times \mathbb{R}$ by $M$. In addition, $\partial_x\Phi_{\lambda,\alpha}^{\beta,t,u,M}(s,\cdot)$ is uniformly Lipschitz over all $s\in [0,u]$, while for any $0<u_0<u$,  $\partial_x\Psi_{\lambda,\gamma}^{t,u,M}(s,\cdot)$ is also uniformly Lipschitz over all $s\in [0,u_0].$ 
\end{proposition}

Note that
$\Phi_{\lambda,\alpha}^{\beta,t,u,M}$ and $\Psi_{\lambda,\gamma}^{t,u,M}$ are special cases of the Hamilton-Jacobi-Bellman equation, induced by a linear problem of diffusion control. In this case, they can be expressed in terms of stochastic optimal control problems. More precisely, let $0<u<\sqrt{M}$ and $\mathcal{D}_u^M$ be the collection of all progressively measurable processes $v=(v(s))_{0\leq s\leq u}$ with respect to a standard Brownian motion $W=(W(s))_{0\leq s\leq u}$ and satisfying $\sup_{0\leq s\leq u}|v(s)|\leq M$ . 

\begin{proposition}[Stochastic Optimal Control Representation]\label{property1}
	We have
	\begin{align}
	\begin{split}\label{rep}
	\Phi_{\lambda,\alpha}^{\beta,t,u,M}(0,0)&=\sup_{v\in \mathcal{D}_u^M}\e\Bigl[f_{\lambda}^{\beta,t,M}\Bigl(2\int_0^u\beta \alpha(s)v(s)ds+\sqrt{2}W(u)\Bigr)-\int_0^u \beta \alpha(s)v(s)^2ds\Bigr],
	\end{split}\\
	\begin{split}\label{rep2}
	\Psi_{\lambda,\gamma}^{t,u,M}(0,0)&=\sup_{v\in \mathcal{D}_u^M}\e\Bigl[f_{\lambda}^{t,M}\Bigl(2\int_0^u  \gamma(s)v(s)ds+\sqrt{2}W(u)\Bigr)-\int_0^u  \gamma(s)v(s)^2ds\Bigr].
	\end{split}
	\end{align}
	Here, the first supremum is achievable by $v(s)=\partial_x\Phi_{\lambda,\alpha}^{\beta,t,u,M}(s,X(s))$ and the second supremum is attained by $v(s)=\partial_x\Psi_{\lambda,\gamma}^{t,u,M}(s,X'(s))$, where $X=(X(s))_{0\leq s\leq u}$ and $X'=(X'(s))_{0\leq s\leq u}$ are the strong solutions to the following SDEs,
	\begin{align*}
	dX(s)&=2\beta \alpha(s)\partial_x\Phi_{\lambda,\alpha}^{\beta,t,u,M}(s,X(s))ds+\sqrt{2}dW(s),\\
	dX'(s)&=2\gamma(s)\partial_x\Psi_{\lambda,\gamma}^{t,u,M}(s,X'(s))ds+\sqrt{2}dW(s)
	\end{align*}
	for $0\leq s\leq u$ with $X(0)=X'(0)=0.$ 
\end{proposition}

Next, note that the boundary condition $f_{\lambda}^{\beta,t,M}$ and $f_\lambda^{t,M}$ are Lipschitz and convex in $(\lambda,x)\in \mathbb{R}^2.$ Using this together with Proposition \ref{property1}, it can be shown that these properties propagate throughout the entire solutions. 

\begin{proposition}[Lipschitiz property and convexity]\label{property2}
	$(\lambda,\alpha)\mapsto\Phi_{\lambda,\alpha}^{\beta,t,u,M}(0,0)$ and $(\lambda,\gamma)\mapsto \Psi_{\lambda,\gamma}^{t,u,M}(0,0)$ are Lipschitz and  convex respectively with respect to the norms \begin{align}\label{property2:eq1}
	|\lambda-\lambda'|+\int_0^u|\alpha(s)-\alpha'(s)|ds\,\,\ \ \mbox{and}\,\, \ \ 
	|\lambda-\lambda'|+\int_0^u|\gamma(s)-\gamma'(s)|ds.
	\end{align} 
\end{proposition}

\subsection{Optimality of the Parisi functional}

We  investigate some properties of the variational formula of $\mathcal{P}_{\beta,t,u,M}.$

\begin{lemma}\label{lem0}
	Let $t,u>0$. For any $M>u^{1/2}$ and $\beta>0,$ the functional $\mathcal{P}_{\beta,t,u,M}$ has a unique minimizer in $\mathbb{R}\times\mathcal{M}_u$.
\end{lemma}

\begin{proof}
	Let $\nu_{t,\beta,M}(dr)=e^{-t|r|^p}dr$ and $\Delta=\int_{[-M,-\sqrt{u}]\cup [\sqrt{u},M]}d\nu_{t,\beta,M}(r).$
	Note that
	\begin{align}
	\begin{split}\notag
	f_{\lambda}^{\beta,t,M}(x)
	&\geq \lambda u + \frac{1}{\beta}\log \frac{1}{\Delta}\int_{[-M,-\sqrt{u}]\cup [\sqrt{u},M]}\exp \beta\bigl(rx+\lambda (r^2-u)\bigr)d\nu_{t,\beta,M}(r)+\frac{1}{\beta}\log \Delta\\
	&\geq   \lambda u + \frac{1}{\Delta}\int_{[-M,-\sqrt{u}]\cup [\sqrt{u},M]}\bigl(rx+\lambda (r^2-u)\bigr)d\nu_{t,\beta,M}(r)+\frac{1}{\beta}\log\Delta
	\end{split}\\
	\begin{split}\notag
	&= \lambda u + \frac{\lambda}{\Delta}\int_{[-M,-\sqrt{u}]\cup [\sqrt{u},M]}(r^2-u)d\nu_{t,\beta,M}(r)+\frac{1}{\beta}\log\Delta,
	\end{split}
	\end{align}
	where the second inequality used Jensen's inequality.
	Since the right-hand side of this inequality is independent of $x$,
	it follows from \eqref{rep} that $\Phi_{\lambda,\alpha}^{\beta,t,u,M}(0,0)$ is bounded by this lower bound and hence,
	\begin{align*}
	\inf_{\alpha\in \mathcal{M}_u}\mathcal{P}_{\beta,t,u,M}(\lambda,\alpha)&\geq \frac{\lambda}{\Delta}\int_{[-M,-\sqrt{u}]\cup [\sqrt{u},M]}(r^2-u)d\nu_{t,\beta,M}(r)+\frac{1}{\beta}\log\Delta-\int_0^us\beta ds\stackrel{\lambda\to\infty}{\to} \infty.
	\end{align*}
	Similarly, let $\Delta'=\int_{[-\sqrt{u},\sqrt{u}]}d\nu_{t,\beta,M}(r).$ We also have
	\begin{align*}
	f_{\lambda}^{\beta,t,M}(x)
	&= \lambda u + \frac{1}{\beta}\log \frac{1}{\Delta'}\int_{[-\sqrt{u},\sqrt{u}]}\exp \beta\bigl(rx+\lambda (r^2-u)\bigr)d\nu_{t,\beta,M}(r)+\frac{1}{\beta}\log \Delta'\\
	&\geq  \lambda u +  \frac{1}{\Delta'}\int_{[-\sqrt{u},\sqrt{u}]}\bigl(rx+\lambda (r^2-u)\bigr)d\nu_{t,\beta,M}(r)+\frac{1}{\beta}\log\Delta'\\
	&=  \lambda u + \frac{\lambda}{\Delta'}\int_{[-\sqrt{u},\sqrt{u}]}(r^2-u)d\nu_{t,\beta,M}(r)+\frac{1}{\beta}\log\Delta'.
	\end{align*}
	Again, since this lower bound is independent of $x$, it follows that by using \eqref{rep},
	\begin{align*}
	\inf_{\alpha\in \mathcal{M}}\mathcal{P}_{\beta,t,u,M}(\lambda,\alpha)&\geq \frac{\lambda}{\Delta'}\int_{[-\sqrt{u},\sqrt{u}]}(r^2-u)d\nu_{t,\beta,M}(r)+\frac{1}{\beta}\log\Delta'-\int_0^us\beta ds\stackrel{\lambda\to-\infty}{\to} \infty.
	\end{align*}
	In addition, note that from Proposition \ref{property2}, $\lambda\mapsto \inf_{\alpha\in \mathcal{M}_u}\mathcal{P}_{\beta,t,u,M}(\lambda,\alpha)$ is continuous.
	From these, we see that there exists some $\lambda_0\in \mathbb{R}$ such that
	\begin{align*}
	\inf_{(\lambda,\alpha)\in \mathbb{R}\times\mathcal{M}_u}\mathcal{P}_{\beta,t,u,M}(\lambda,\alpha)&=\inf_{\lambda\in \mathbb{R}}\inf_{\alpha\in \mathcal{M}_u}\mathcal{P}_{\beta,t,u,M}(\lambda,\alpha)
=\inf_{\alpha\in \mathcal{M}_u}\mathcal{P}_{\beta,t,u,M}(\lambda_0,\alpha).
	\end{align*}
	Finally, by using Proposition \ref{property2} and noting that $\mathcal{M}_u$ is a compact space with respect to the metric $\int_0^u|\alpha(s)-\alpha'(s)|ds$, we see that $\mathcal{P}_{\beta,t,u,M}(\lambda_0,\cdot) $ is minimized by some $\alpha_0\in \mathcal{M}_u.$ Hence, $(\lambda_0,\alpha_0)$ is a minimizer of $\mathcal{P}_{\beta,t,u,M}$. This completes our proof.
\end{proof}

For the rest of this section, we denote the minimizer of $\mathcal{P}_{\beta,t,u,M}$ by $(\lambda_{M,\beta},\alpha_{M,\beta})$. In next two lemmas, we show that $\beta\alpha_{M,\beta}$ is uniformly integrable and $\lambda_{M,\beta}$ is uniformly bounded in $(M,\beta).$

\begin{lemma}\label{add:lem1}
	Let $t,u>0.$ For any $M>u^{1/2}$ and $\beta>0,$	we have
	\begin{align}\label{add:lem1:eq1}
	\int_0^u s\beta \alpha_{M,\beta}(s)ds\leq \frac{u}{\sqrt{2}}.
	\end{align}
\end{lemma}

\begin{proof}
	Our proof consists of three major steps.
	First of all, we introduce an auxiliary free energy. For $b>0,$ define 
	\begin{align*}
	F_{n}(b)&=\frac{1}{n}\log \int_{x\in[-M,M]^n:\opnorm{x}_2^2\in[u-\varepsilon,u+\varepsilon]}\exp\Bigl(\frac{b}{n^{1/2}}\la G_nx,x\ra-\beta t \|x\|_p^p\Bigr)dx.
	\end{align*}
	Here, to keep the notation light, we ignore the dependence on $\beta,t,u,M,$ and $\varepsilon.$ 
	In a similar manner as \eqref{Pf:freeenergy}, the general result in \cite{Pan05} ensures that this free energy can also be expressed as a similar Parisi-type formula,
	\begin{align}\label{add:eq4}
	\lim_{\varepsilon\downarrow 0}\lim_{n\to\infty}F_{n}(b)&=\inf_{(\lambda,\alpha)\in \mathbb{R}\times \mathcal{M}_u}\mathcal{P}_b (\lambda,\alpha)
	\end{align}
	for
	\begin{align*}
	\mathcal{P}_b (\lambda,\alpha)&:=\Phi_{\lambda,\alpha}^b(0,0)-\lambda u-b^2\int_0^u s\alpha(s)ds,
	\end{align*}
	where $\Phi_{\lambda,\alpha}^b(0,0)$ is defined through the weak solution (see \cite{JT16}) to the following PDE,
	\begin{align}\label{pde2}
	\partial_s\Phi_{\lambda,\alpha}^b(s,x)&=-b^2\bigl(\partial_{xx}\Phi_{\lambda,\alpha}^b(s,x)+\alpha(s)\bigl(\partial_x\Phi_{\lambda,\alpha}^b(s,x)\bigr)^2\bigr)
	\end{align}
	for $(s,x)\in [0,u]\times \mathbb{R}$ with boundary condition
	\begin{align*}
	\Phi_{\lambda,\alpha}^{b}(u,x)&=f_{\lambda}(x):=\log \int_{-M}^M\exp \bigl(sx+\lambda s^2-\beta t|s|^p\bigr)ds.
	\end{align*}
	Note that  $F_{n,t,u,M,\varepsilon}(b)=b^{-1}F_n(b)$ and consequently, \eqref{add:eq4} is equal to \eqref{Pf:freeenergy} modulo a multiplicative factor $b^{-1}$. Here, it can be checked that the two solutions, $\Phi_{\lambda,\alpha}^{b}$ and $\Phi_{\lambda,\alpha}^{\beta,t,u,M}$ for $b=\beta$, are essentially the same up to a transformation,
	\begin{align*}
	\Phi_{b^{-1}\lambda ,\alpha}^{b,t,u,M}(s,\beta^{-1}x)=b^{-1}\Phi_{\lambda,\alpha}^b(s,x),\,\,(s,x)\in[0,u]\times\mathbb{R}
	\end{align*}
and therefore,
	\begin{align}\label{relation}
   \mathcal{P}_{b,t,u,M}(\lambda,\alpha)=b^{-1}\mathcal{P}_b (b\lambda,\alpha),\,\,\forall (\lambda,\alpha)\in \mathbb{R}\times\mathcal{M}_u.
	\end{align}
	From these, Proposition \ref{property2}, and Lemma \ref{lem0}, we see that for ever $b>0,$ $\mathcal{P}_b$ is convex with respect to the metric \eqref{property2:eq1} and it has a minimizer, $(\lambda_b,\alpha_b)$, satisfying $(\lambda_b,\alpha_b)=(b^{-1}\lambda_{M,b},\alpha_{M,b})$.

	Next we continue to bound $b\int_0^us\alpha_b(s)ds$ through the  $b$-derivative of the Parisi formula in \eqref{add:eq4}.
	First, it can be shown that $\Phi_{\lambda,\alpha}^b(0,0)$ also admits an analogous stochastic optimal control representation as that for $\Phi_{\lambda,\alpha}^{\beta,t,u,M}$ in Proposition \ref{property1}. More precisely, let $W$ and $\mathcal{D}_u^M$ be defined as in Proposition \ref{property1}. We can write
	\begin{align*}
	\Phi_{\lambda,\alpha}^b(0,0)&=\sup_{v\in \mathcal{D}_u^M}\e\Bigl[f_{\lambda}\Bigl(2\int_0^u b^2\alpha(s)v(s)ds+\sqrt{2}bW(u)\Bigr)-\int_0^u b^2\alpha(s)v(s)^2ds\Bigr].
	\end{align*}
	Here, the optimal process attaining the maximum is given by 
	$
	v_{\lambda,\alpha}(s)=\partial_x\Phi_{\lambda,\alpha}^b(s,X_{\lambda,\alpha}(s)),
	$
	where $X_{\lambda,\alpha}$ is the strong solution to the following SDE,
	\begin{align*}
	dX_{\lambda,\alpha}(s)&=2b^2\alpha(s)\partial_x\Phi_{\lambda,\alpha}^b(s,X_{\lambda,\alpha}(s))ds+b\sqrt{2}dW(s),\,\,0\leq s\leq u, \ \ \text{ and }  \ \ X_{\lambda,\alpha}(0)=0.
	\end{align*}
	Using this representation, the same computation as  \cite[Proposition 4]{AC2016} yields that for any $(\lambda,\alpha),$
	\begin{align}\label{add:eq7}
	\frac{d}{db}\Phi_{\lambda,\alpha}^b(0,0)&=2b\Bigl(u\e v_{\lambda,\alpha}(u)^2+u\e \partial_{xx}\Phi_{\lambda,\alpha}^b(u,X_{\lambda,\alpha}(u))-\int_0^u s\e v_{\lambda,\alpha}(s)^2\alpha(ds)\Bigr).
	\end{align}
	To handle this derivative, note that $\mathcal{P}_b$ is also a convex functional. It can be argued (see, e.g., \cite{Chen17}) that the directional derivative of $\mathcal{P}_b$ at $(\lambda_b,\alpha_b)$ can be explicitly computed. In particular, the optimality of $(\lambda_b,\alpha_b)$ ensures that  	
	\begin{align}\label{derivative}
	\e \bigl(\partial_x\Phi_{\lambda_b,\alpha_b}^b(s,X_{\lambda_b,\alpha_b}(s))\bigr)^2=\e v_{\lambda_b,\alpha_b}(s)^2=s
	\end{align}
	for any point $s$ in the support of $\alpha_b$ and
 \begin{align}\label{derivative2}
	\e \partial_\lambda\Phi_{\lambda_b,\alpha_b}^b(u,X_{\lambda_b,\alpha_b}(u))=u.
\end{align}
	From \eqref{derivative2} and noting that
	\begin{align*}
	\partial_{xx}\Phi_{\lambda,\alpha}^b(u,x)=\partial_{xx}f_{\lambda}(x)&=\partial_\lambda f_{\lambda}(x)-\bigl(\partial_xf_{\lambda}(x)\bigr)^2=\partial_{\lambda}\Phi_{\lambda,\alpha}^b(u,x)-\bigl(\partial_{x}\Phi_{\lambda,\alpha}^b(u,x)\bigr)^2,
	\end{align*}
	it follows that
	\begin{align*}
	\e \partial_{xx}\Phi_{\lambda_b,\alpha_b}^b(u,X_{\lambda_b,\alpha_b}(u))&=\e \partial_\lambda\Phi_{\lambda_b,\alpha_b}^b(u,X_{\lambda_b,\alpha_b}(u))-\e v_{\lambda_b,\alpha_b}(u)^2\\
	&=u-\e v_{\lambda_b,\alpha_b}(u)^2.
	\end{align*}
	Plugging this equation and \eqref{derivative} into \eqref{add:eq7} leads to
	\begin{align*}
	\frac{d}{db}\Phi_{\lambda,\alpha}^b(0,0)\Big|_{(\lambda,\alpha)={(\lambda_b,\alpha_b)}}&=2b\Bigl(u^2-\int_0^u s^2\alpha_b(ds)\Bigr)=4b\int_0^u s\alpha_b(s)ds,
	\end{align*}
	where the last equality used the integration by parts.
	As a result, we arrive at
	\begin{align}
	\begin{split}\label{add:eq-2}
	\frac{d}{db}\mathcal{P}_b(\lambda,\alpha)\Big|_{(\lambda,\alpha)=(\lambda_b,\alpha_b)}&=4b\int_0^u s\alpha_b(s)ds-2b\int_0^u s\alpha_b(s)ds=2b\int_0^u s\alpha_b(s)ds,
	\end{split}
	\end{align}
	Now from \eqref{add:eq4}, we have that for any $0<b'<b<\infty,$
	\begin{align*}
	\lim_{\varepsilon\downarrow 0}\lim_{n\to\infty}\frac{F_n(b)-F_n(b')}{b-b'}&=\frac{1}{b-b'}\Bigl(\inf_{(\lambda,\alpha)\in \mathbb{R}\times \mathcal{M}_u}\mathcal{P}_b (\lambda,\alpha)-\inf_{(\lambda,\alpha)\in \mathbb{R}\times \mathcal{M}_u}\mathcal{P}_{b'}(\lambda,\alpha)\Bigr)\\
	&\geq \frac{1}{b-b'}\bigl(\mathcal{P}_b (\lambda_b ,\alpha_b )-\mathcal{P}_{b'}(\lambda_b ,\alpha_b )\bigr),
	\end{align*}
	which combining with  the convexity of $F_n(b)$ and \eqref{add:eq-2}, after sending $b'\uparrow b,$
	\begin{align}\label{add:eq---2}
\lim_{\varepsilon\downarrow 0}\liminf_{n\to\infty}F_n'(b)\geq 2b\int_0^u s\alpha_b(s)ds.
	\end{align}
	In the last step,  note that a direct computation gives
	\begin{align}
	F_{n}'(b)=\e\Bigl\la \frac{\la G_nx,x\ra}{n^{3/2}}\Bigr\ra_b
	&\leq \e\sup_{x\in [-M,M]^n: \opnorm{x}_2^2 \in [u-\varepsilon,u+\varepsilon]}\frac{\la G_nx,x\ra}{n^{3/2}}
	\leq \frac{(u+\varepsilon)}{\sqrt{2}}\frac{\e\|\bar G_n\|_2}{n^{1/2}},\label{add:eq8}
	\end{align}
	where $\la \cdot\ra_b$ is the Gibbs expectation associated to the partition function,
	$$
	\int_{x\in [-M,M]:\opnorm{x}_2^2\in[u-\varepsilon,u+\varepsilon]}\exp\Bigl(\frac{b}{n^{1/2}}\la G_nx,x\ra-\beta t\|x\|_p^p\Bigr)dx.
	$$
	It follows from \eqref{add:eq---2}, \eqref{add:eq8}, and the convexity of $F_n$  that
	\begin{align*}
	2b\int_0^u s\alpha_b(s)ds\leq \lim_{\varepsilon\downarrow 0}\liminf_{n\to\infty}F_n'(b)\leq\sqrt{2}u.
	\end{align*}
To translate this inequality back to that for $\alpha_{M,\beta}$, we recall that $\alpha_b=\alpha_{M,\beta}$ for $b=\beta$. This completes our proof of \eqref{add:lem1:eq1}.
\end{proof}

\begin{lemma}\label{add:lem3}
	Let $t,u>0$. There exists a constant $K>0$ depending only on $t,u$ such that 
	\[ 
	|\lambda_{M,\beta}|\leq K \quad \text{for any }  \ \ \beta >M>2u^{1/2}. \]
\end{lemma}

\begin{proof}
	Since $f_{\lambda}^{\beta,t,M}$ is convex,  we can take $v\equiv 0$ in  \eqref{rep} and then apply Jensen's inequality to get
	\begin{align}\label{add:eq---3}
	\Phi_{\lambda,\alpha}^{\beta,t,u,M}(0,0)&\geq \e f_{\lambda}^{\beta,t,M}(\sqrt{2}W(u))\geq f_{\lambda}^{\beta,t,M}(0).
	\end{align}
	Fix any $M  > 2 u^{1/2}$. For $0<D\leq M,$ we have 
	\begin{align*}
	f_{\lambda}^{\beta,t,M}(0)
	&\geq \frac{1}{\beta}\log\int_{-D}^D e^{\beta(\lambda r^2-t|r|^p)}dr \geq \frac{1}{\beta}\log\frac{1}{2D}\int_{-D}^D e^{\beta(\lambda r^2-t|r|^p)}dr+\frac{1}{\beta}\log 2D\\
	&\geq \frac{1}{2D}\int_{-D}^D (\lambda r^2-t|r|^p)dr+\frac{1}{\beta}\log 2D,
	\end{align*}
	where the last inequality used Jensen's inequality.
	Now from \eqref{add:eq---3},
	\begin{align*}
	\mathcal{P}_{\beta,t,u,M}(\lambda,\alpha)&\geq \lambda\Bigl(\frac{1}{2D}\int_{-D}^Dr^2dr-u\Bigr)- \frac{t}{2D}\int_{-D}^D|r|^pdr+\frac{\log 2D}{\beta}-\int_0^u s\beta \alpha(s)ds
	\end{align*}
	and therefore, from Lemma \ref{add:lem1},
	\begin{align*}
	\lambda\Bigl(\frac{1}{2D}\int_{-D}^Dr^2dr-u\Bigr)&\leq \inf_{\lambda,\alpha} \mathcal{P}_{\beta,t,u,M}(\lambda,\alpha)+\frac{t}{2D}\int_{-D}^D|r|^pdr-\frac{\log 2D}{\beta}+\frac{u}{\sqrt{2}}.
	\end{align*}
	Now take $\lambda=\lambda_{M,\beta}.$ If $\lambda_{M,\beta}>0,$ we set $D=\sqrt{u} $ and if $\lambda_{M,\beta}<0,$ set $D= 2 \sqrt{u}$ so that $(2D)^{-1}\int_{-D}^Dr^2dr-u$ shares the same sign with $\lambda_{M,\beta}$. Therefore, 
	\begin{align*}
	|\lambda_{M,\beta}|&\leq  \frac{1}{\bigl|\frac{1}{2D}\int_{-D}^Dr^2dr-u\bigr|}\Bigl(\inf_{\lambda,\alpha} \mathcal{P}_{\beta,t,u,M}(\lambda,\alpha)+\frac{t}{2D}\int_{-D}^D|r|^pdr-\frac{\log 2D}{\beta}+\frac{u}{\sqrt{2}}\Bigr).
	\end{align*}
	Finally, from \eqref{add:lem2:proof:eq1}, 
	\begin{align*}
	\inf_{\lambda,\alpha} \mathcal{P}_{\beta,t,u,M}(\lambda,\alpha)&= \lim_{\varepsilon\downarrow 0}\lim_{n\to\infty}\e F_{n,t,u,M,\varepsilon}(\beta)
	\leq \limsup_{n\to\infty}\e L_{n,p}(t)+\frac{\log 2M}{\beta}
	\leq C(t)^{1/2}+\frac{\log 2M}{\beta}.
	\end{align*}
	These yield the desired bound as long as $\beta>M>2u^{1/2}.$
\end{proof}

\subsection{Proof of Proposition \ref{add:prop1}}

Recall that $(\lambda_{M,\beta},\alpha_{M,\beta})$ is the unique minimizer of $\mathcal{P}_{\beta,t,u,M}$. From Lemma \ref{add:lem1} and the monotinicity of $\beta \alpha_{M,\beta}(s)$ in $s$,
\begin{align}
\int_0^u\beta\alpha_{M,\beta}(s)ds&= \int_0^{u/2}\beta\alpha_{M,\beta}(s)ds+\int_{u/2}^u\beta\alpha_{M,\beta}(s)ds\nonumber\\
& \le  2 \int_{u/2}^{u}\beta\alpha_{M,\beta}(s)d  \leq 2 \cdot \frac{2}{u}\int_{u/2}^{u}s\beta\alpha_{M,\beta}(s)ds \leq 2\sqrt{2}. \label{add:eq11}
\end{align} 
Define a positive measure $\nu_{M,\beta}$ on $[0,u]$ by $\nu_{M,\beta}(A)=\int_A\beta \alpha_{M,\beta}(s)ds.$  From the inequality \eqref{add:eq11} and noting that $\alpha_{M,\beta}(s)$ is nondecreasing in $s$, we see that for any $0<r<u$, $\sup_{\beta>0,s\in [0,r]}\beta \alpha_{M,\beta}(s)<\infty$ and hence, we can pass to a subsequence $(\beta_n)_{n\geq 1}$ such that $(\beta_n\alpha_{M,\beta_n})_{n\geq 1}$ converges to some $\gamma_{M,\infty}\in \mathcal{N}_{u}$ almost everywhere on $[0,u).$ In addition, we can assume that along the same subsequence, $L_M:=\lim_{n\to\infty}\int_0^u \beta \alpha_{M,\beta_n}(s)ds$ exists. Consequently, we conclude that $\nu_{M,\beta_n}$ converges vaguely on $[0,u]$ to some $\nu_{M,\infty}$, defined as $\nu_{M,\infty}(A)=\int_A\gamma_{M,\infty}(s)ds+\Gamma_{M,\infty}\delta_u$, where $0\leq \Gamma_{M,\infty}:=L_M-\int_0^u\gamma_{M,\infty}(s)ds<\infty.$
On the other hand, Lemma \ref{add:lem3} also allows us to pass to a subsequence of $\lambda_{M,\beta}$, along which it is convergent. From now on, without loss of generality, we assume that these are convergent without passing to a subsequence. To proceed,  from the above discussion, combined with  Lemma \ref{add:lem3} and the inequality~\eqref{add:eq11}, we note that  there exists a constant $K$ depending only on $t,u$ such that for any $M>2u^{1/2}$, \begin{align}\label{add:eq9}
\max\Bigl(|\lambda_{M,\infty}|,\Gamma_{M,\infty},\int_0^u\gamma_{M,\infty}(s)ds,\sup_{s\in[0,u)}(u-s)\gamma_{M,\infty}(s)\Bigr)&\leq K.
\end{align}

\begin{lemma}\label{add:lem4}
	For any $0<u<M^2,$ we have that
	\begin{align*}
	\liminf_{\beta\to\infty}\Phi_{\lambda_{M,\beta},\alpha_{M,\beta}}^{\beta,t,u,M}(0,0)&\geq \Psi_{\lambda_{M},\gamma_{M,\infty}}^{t,u,M}(0,0),
	\end{align*}
	where $
	\lambda_M:=\lambda_{M,\infty}+\Gamma_{M,\infty}.$
\end{lemma}

\begin{proof}
	Recall the representations in Proposition \ref{property1}. Consider any  progressively measurable process $v\in \mathcal{D}_u^M$ on $[0,u]$, which is  left-continuous  at $u$, i.e., $\lim_{s\to u-}v(s)=v(u)$. Note that $\max_{s\in[0,u]}|v(s)|\leq M.$ From the vague convergence of $\nu_{M,\beta}$ to $\nu_{M,\infty}$ on $[0,u]$, we have that
	\begin{align}
	\liminf_{\beta\to\infty}\Phi_{\lambda_{M,\beta},\alpha_{M,\beta}}^{\beta,t,u,M}(0,0)
	&\geq \liminf_{\beta\to\infty}\e\Bigl[f_{\lambda_{M,\beta}}^{\beta,t,M}\Bigl(2\int_0^uv(s)\beta \alpha_{M,\beta} (s)ds+\sqrt{2}W(u)\Bigr)-\int_0^u v(s)^2\beta \alpha_{M,\beta}(s) ds\Bigr]\nonumber\\
	&=\e \Bigl[f_{\lambda_{M,\beta}}^{t,M}\Bigl(2\int_0^uv(s)\nu_{M,\infty}(ds)+\sqrt{2}W(u)\Bigr)\Bigr]-\int_0^uv(s)^2\nu_{M,\infty}(ds)\Bigr].\label{add:eq---4}
	\end{align}	
	From the definition of $\nu_{M,\infty},$ the above lower bound can be written as
	\begin{align}
	\begin{split}
	&\e \Bigl[f_{\lambda_{M,\infty}}^{t,M}\Bigl(2\int_0^uv(s)\nu_{M,\infty}(ds)+\sqrt{2}W(u)\Bigr)-\int_0^uv(s)^2\nu_{M,\infty}(ds)\Bigr]\\
	&=\e \Bigl[f_{\lambda_{M,\infty}}^{t,M}\Bigl(2\int_0^uv(s)\gamma_{M,\infty}(s)ds+2v(u)\Gamma_{M,\infty}+\sqrt{2}W(u)\Bigr)-\int_0^uv(s)^2\gamma_{M,\infty}(s)ds-\Gamma_{M,\infty} v(u)^2\Bigr].
	\end{split}\label{add:lem4:proof:eq1}
	\end{align}
	On the other hand, recall from \eqref{rep2},
	\begin{align}\label{add:lem4:proof:eq2}
	\Psi_{\lambda_M,\gamma_{M,\infty}}^{t,u,M}(0,0)=\sup_{v\in \mathcal{D}_u^M}\e \Bigl[f_{\lambda_M}^{t,M}\Bigl(2\int_0^uv(s)\gamma_{M,\infty}(s)ds+\sqrt{2}W(u)\Bigr)-\int_0^u v(s)^2\gamma_{M,\infty}(s)ds\Bigr].
	\end{align}
	For $0\leq s<u,$ denote by $r^*(s)\in [-M,M]$ the largest (random) maximizer in the definition of 
	\begin{align*}
	f_{\lambda_M}^{t,M}\Bigl(2\int_0^sv(l)\gamma_{M,\infty}(l)dl+\sqrt{2}W(s)\Bigr).
	\end{align*}
	Using the compactness of $[-M,M]$, we can pick a sequence $(s_k)\subset (0,u)$ with $\lim_{k\to\infty}s_k=u$ such that $r^*(u):=\lim_{k\to\infty}r^*(s_k)$ is a maximizer to 
	\begin{align*}
	f_{\lambda_M}^{t,M}\Bigl(2\int_0^uv(l)\gamma_{M,\infty}(l)dl+\sqrt{2}W(u)\Bigr).
	\end{align*}
	
	Next,  fix an arbitrary $v\in \mathcal{D}_u^M$. Define a progressively measurable process $v_k\in \mathcal{D}_u^M$ by
	\begin{align*}
	v_k(s)&=v(s)1_{[0,s_k)}(s)+r^*(s_k)1_{[s_k,u]}(s).
	\end{align*}
	Note that $v_k$ is left-continuous at $u.$ Consequently, we can plug this $v_k$ into \eqref{add:eq---4} and \eqref{add:lem4:proof:eq1} to obtain that
	\begin{align*}
	\liminf_{\beta\to\infty}\Phi_{\lambda_{M,\beta},\alpha_{M,\beta}}^{\beta,t,u,M}(0,0)
	&\geq \e \Bigl[f_{\lambda_{M,\infty}}^{t,M}\Bigl(2\int_0^{s_k}v(s)\gamma_{M,\infty}(s)ds+2r^*(s_k)\int_{s_k}^u \gamma_{M,\infty}(s)ds+2r^*(s_k)\Gamma_{M,\infty}+\sqrt{2}W(u)\Bigr)\\
	&\qquad\qquad -\int_0^{s_k}v(s)^2\gamma_{M,\infty}(s)ds-r^*(s_k)^2\int_{s_k}^u\gamma_{M,\infty}(s)ds-r^*(s_k)^2\Gamma_{M,\infty}\Bigr].
	\end{align*}
	Sending $k\to\infty$ and noting that $\int_0^u\gamma_{M,\infty}(s)ds<\infty$ imply that
	\begin{align}
	\liminf_{\beta\to\infty}\Phi_{\lambda_{M,\beta},\alpha_{M,\beta}}^{\beta,t,u,M}(0,0) &\geq\e \Bigl[f_{\lambda_{M,\infty}}^{t,M}\Bigl(2\int_0^{u}v(s)\gamma_{M,\infty}(s)ds+2r^*(u)\Gamma_{M,\infty}+\sqrt{2}W(u)\Bigr) \nonumber\\
	&\qquad\qquad-\int_0^uv(s)^2\gamma_{M,\infty}(s)ds-r^*(u)^2\Gamma_{M,\infty}\Bigr].\label{add:eq---5}
	\end{align}
	Here, from the definition of $f_{\lambda_{M,\infty}}^{t,M},$
	\begin{align*}
	&f_{\lambda_{M,\infty}}^{t,M}\Bigl(2\int_0^{u}v(s)\gamma_{M,\infty}(s)ds+2r^*(u)\Gamma_{M,\infty}+\sqrt{2}W(u)\Bigr)\\
	&\geq \Bigl(2\int_0^{u}v(s)\gamma_{M,\infty}(s)ds+2r^*(u)\Gamma_{M,\infty}+\sqrt{2}W(u)\Bigr)r^*(u)+\lambda_{M,\infty}r^*(u)^2-t|r^*(u)|^p\\
	&=\Bigl(2\int_0^{u}v(s)\gamma_{M,\infty}(s)ds+\sqrt{2}W(u)\Bigr)r^*(u)+\bigl(\lambda_{M,\infty}+2\Gamma_{M,\infty}\bigr)r^*(u)^2-t|r^*(u)|^p.
	\end{align*}
	Plugging this lower bound into the right-hand side of \eqref{add:eq---5} yields that 
	\begin{align*}
	&\liminf_{\beta\to\infty}\Phi_{\lambda_{M,\beta},\alpha_{M,\beta}}^{\beta,t,u,M}(0,0)\\
	&\geq \e\Bigl[\Bigl(2\int_0^{u}v(s)\gamma_{M,\infty}(ds)+\sqrt{2}W(u)\Bigr)r^*(u)+\bigl(\lambda_{M,\infty}+\Gamma_{M,\infty}\bigr)r^*(u)^2-t|r^*(u)|^p -\int_0^uv^2\gamma_{M,\infty}(s)ds\Bigr]\\
	&=\e\Bigl[f_{\lambda_M}^{t,M}\Bigl(2\int_0^{u}v(s)\gamma_{M,\infty}(s)ds+\sqrt{2}W(u)\Bigr)-\int_0^uv(s)^2\gamma_{M,\infty}(s)ds\Bigr].
	\end{align*}
	Since this is valid for all $v\in \mathcal{D}_u^M$, taking supremum over this set completes our proof by using \eqref{add:lem4:proof:eq2}.
	
\end{proof}

We now proceed to establish the proof of Proposition \ref{add:prop1}.
Assume that $\gamma\in \mathcal{N}_u$ satisfies $\gamma(u-)<\infty$ and $\lambda\in \mathbb{R}.$ Define $
\alpha_\beta(s)=\beta^{-1}\gamma(s)1_{[0,u)}(s)+1_{\{u\}}(s).
$
Here $\alpha_\beta\in \mathcal{M}_u$ as long as $\beta$ is large enough. Since $$f_\lambda^{\beta,u,M}(x)\leq f_\lambda^{t,M}(x)+\beta^{-1}\log 2M,$$
we have
\begin{align*}
\lim_{\beta\to\infty}\Phi_{\lambda,\alpha_\beta}^{\beta,t,u,M}(0,0)&=\lim_{\beta \to\infty}\sup_{v\in \mathcal{D}_u^M}\e\Bigl[f_\lambda^{\beta,t,M}\Bigl(2\int_0^u v(s)\gamma(s) ds+\sqrt{2}W(u)\Bigr)-\int_0^u v(s)^2\gamma(s) ds\Bigr]\\
&\leq\sup_{v\in \mathcal{D}_u^M}\e\Bigl[f_\lambda^{t,M}\Bigl(2\int_0^u v(s)\gamma(s) ds+\sqrt{2}W(u)\Bigr)-\int_0^u v(s)^2\gamma(s) ds\Bigr].
\end{align*}
Also note
$
\int_0^u s\beta\alpha_\beta (s)ds=\int_0^u s\gamma (s)ds.
$
From these,
\begin{align*}
\limsup_{\beta\to\infty}\inf_{(\lambda,\alpha)\in \mathbb{R}\times\mathcal{M}_u}\mathcal{P}_{\beta,t,u,M}(\lambda,\alpha)&\leq \mathcal{P}_{t,u,M}(\lambda,\gamma).
\end{align*}
The continuity of $\mathcal{P}_{t,u,M}$ in  Proposition \ref{property2} then implies that 
\begin{align*}
\limsup_{\beta\to\infty}\inf_{(\lambda,\alpha)\in \mathbb{R}\times\mathcal{M}_u}\mathcal{P}_{\beta,t,u,M}(\lambda,\alpha)\leq \inf_{(\lambda,\gamma)\in \mathbb{R}\times \mathcal{N}_u}\mathcal{P}_{t,u,M}(\lambda,\gamma).
\end{align*}
To prove the reverse inequality, we use Lemma \ref{add:lem4} to obtain 
\begin{align*}
\liminf_{\beta\to\infty}\inf_{(\lambda,\alpha)\in \mathbb{R}\times\mathcal{M}_u}\mathcal{P}_{\beta,t,u,M}(\lambda,\alpha)&=\liminf_{\beta\to\infty}\mathcal{P}_{\beta,t,u,M}(\lambda_{M,\beta},\alpha_{M,\beta})\\
&\geq \Psi_{\lambda_{M},\gamma_{M,\infty}}^{t,u,M}(0,0)-\lambda_{M,\infty}u-\int_0^us\nu_{M,\infty}(ds)\\
&=\Psi_{\lambda_M,\gamma_{M,\infty}}^{t,u,M}(0,0)-\lambda_Mu-\int_0^us\gamma_{M,\infty}(s)ds=\mathcal{P}_{t,u,M}(\lambda_M,\gamma_{M,\infty}).
\end{align*}
Let $\gamma_M:=\gamma_{M,\infty}.$ Combining this lower bound with the upper bound obtained above, we deduce 
\begin{align*}
\liminf_{\beta\to\infty}\inf_{(\lambda,\alpha)\in \mathbb{R}\times\mathcal{M}_u}\mathcal{P}_{\beta,t,u,M}(\lambda,\alpha)=\inf_{(\lambda,\gamma)\in \mathbb{R}\times \mathcal{N}_u}\mathcal{P}_{t,u,M}(\lambda,\gamma)=\mathcal{P}_{t,u,M}(\lambda_M,\gamma_{M}),
\end{align*}
which gives \eqref{add:prop1:eq2}. Consequently, we  have \eqref{add:prop1:eq0} by using \eqref{Pf:freeenergy}. Finally, \eqref{add:eq9} and the fact $\lambda_M=\lambda_{M,\infty}+\Gamma_{M,\infty}$ validate \eqref{add:prop1:eq1}. This completes our proof.

\section{Proof of Theorem \ref{thm0}}\label{sec12}

We present the proof of Theorem \ref{thm0} in this section. With the help of \eqref{thm-1:eq2}, the assertion \eqref{thm0:eq2} follows immediately from \eqref{thm0:eq1}, so we only need to establish \eqref{thm0:eq1}. Our argument consists of matching the upper and lower bounds.

\subsection{Lower bound}

We verify that
\begin{align*}
L_{p,u}(t)\geq \inf_{(\lambda,\gamma)\in \mathbb{R}\times\mathcal{N}_u^d}\mathcal{P}_{t,u}(\lambda,\gamma).
\end{align*}
From Propositions \ref{extra:prop1} and \ref{add:prop1}, 
it remains to show that
\begin{align}\label{add:eq---6}
 \limsup_{M\to\infty}\inf_{(\lambda,\gamma)\in \mathbb{R}\times\mathcal{N}_u}\mathcal{P}_{t,u,M}(\lambda,\gamma)\geq \inf_{(\lambda,\gamma)\in \mathbb{R}\times\mathcal{N}_u^d}\mathcal{P}_{t,u}(\lambda,\gamma).
\end{align}
Our argument relies on a stochastic optimal control representation for $\Psi_{\lambda,\gamma}^{t,u}(0,0)$ in the same spirit as that of Proposition \ref{property1}, whose proof is deferred to the appendix.


\begin{proposition}\label{rep3}
	Let $(\lambda,\gamma)\in \mathbb{R}\times\mathcal{N}_u^d$. There exist a probability space $(\Omega,\mathbb{P},\mathcal{F})$, a filtration $(\mathcal{F}_s)_{0\leq s\leq u}$, a continuous adapted process $X=(X(s))_{0\leq s\leq u}$, and a standard Brownian motion $W=(W(s))_{0\leq s\leq u}$ such that together they form a weak solution to the following SDE,
	\begin{align}
	\begin{split}\label{weak}
	dX(s)&=2\gamma(s)\partial_x\Psi_{\lambda,\gamma}^{t,u}(s,X(s))ds+\sqrt{2}dW(s),\,\,0\leq s\leq u,\,\,X(0)=0.
	\end{split}
	\end{align} 
	Furthermore, we have
	\begin{align}\label{rep4}
	\Psi_{\lambda,\gamma}^{t,u}(0,0)=\sup_{v\in \mathcal{D}_u}\e\Bigl[f_\lambda^{t,u}\Bigl(2\int_0^uv(s)\gamma(s) ds+\sqrt{2}W(u)\Bigr)-\int_0^u v(s)^2\gamma(s) ds\Bigr],
	\end{align}
	where  $\mathcal{D}_u$ is  the collection of all progressively measurable processes $v=(v(s))_{0\leq s\leq u}$ with respect to the filtration $(\mathcal{F}_s)_{0\leq s\leq u}$ with $\e \int_0^u |v(s)|^2ds<\infty.$ Here the supremum is attained by the process $v\in \mathcal{D}_u$ defined by $v(s)=\partial_x\Psi_{\lambda,\gamma}^{t,u}(s,X(s))$ for $0\leq s\leq u.$
\end{proposition}

\begin{remark}\label{rmk4}
	\rm Recall from Proposition \ref{property0} that $\partial_x\Psi_{\lambda,\gamma}^{t,u,M}(s,\cdot)$ and $\partial_x\Phi_{\lambda,\alpha}^{\beta,t,u,M}(s,\cdot)$ are uniformly Lipschitz over all $x\in [0,u)$. This ensures the existence of the strong solutions of $X$ and $X'$ in Proposition \ref{property1}. In \eqref{weak}, $\partial_x\Psi_{\lambda,\gamma}^{t,u}(s,\cdot)$ is not Lipschitz, but it can be shown from the Hopf-Cole transformation \eqref{hopfcole} that it is of moderate growth, namely, there exists some $0<a<1$ such that $|\partial_x\Psi_{\lambda,\gamma}^{t,u}(s,x)|\leq C(1+|x|^a)$ for  any $(s,x)\in \bigl([0,u]\times \mathbb{R}\bigr)\setminus \{(u,0)\}$ (see Lemma \ref{appendix:lem1}), where $C>0$ depends on $\gamma$. Hence,
	\begin{align*}
	|\partial_x\Psi_{\lambda,\gamma}^{t,u}(s,x)|\leq 2C(1+|x|),\,\,\forall (s,x)\in \bigl([0,u]\times \mathbb{R}\bigr)\setminus \{(u,0)\}.
	\end{align*}This ensures the existence of the weak solution of \eqref{weak}, see \cite[Propositio 5.3.6]{KS91}. Additionally, from this upper bound, it is a standard application of the Gronwall inequality (see, e.g., \cite[Problem 5.3.15]{KS91}) that 
	\begin{align*}
    \e \max_{0\leq s\leq u} |X(s)|^2&\leq 4u(C_0^2+1)e^{4u^2C_0^2},\,\,C_0:=2\gamma(u)C.
	\end{align*}
	However, this upper bound is not good enough for us since in the Parisi formula we take infimum over all $\gamma$ and this bound might diverge.
\end{remark}

The following technical lemma  controls $\e|X(u)|^2$ uniformly over bounded  $\lambda$ and $\int_0^u \gamma(s)ds$.

\begin{lemma}\label{lem7}
	Assume that $t,u>0$. There exists a nonnegative continuous function $K$ on $(2,\infty)\times (0,\infty)^2\times \mathbb{R}\times [0,\infty)\to [0,\infty)$ such that for any $(\lambda,\gamma)\in \mathbb{R}\times \mathcal{N}_u^d$,
	we have that
	\begin{align*}
	\e \bigl|X(u)\bigr|^2&\leq K\Bigl(p,t,u,\lambda,\int_0^u\gamma(s)ds\Bigr).
	\end{align*}
\end{lemma}

\begin{proof} 
	From \eqref{lem3:eq2},
	\begin{align}
	\begin{split}\label{thm0:proof:eq3}
	\Psi_{\lambda ,\gamma }^{t,u}(u,X(u))&\leq 
	\Bigl(\frac{2}{t}\Bigr)^{1/(p-1)}|X(u)|^{1+1/(p-1)}+\Bigl(\frac{2|\lambda |}{t}\Bigr)^{1/(p-2)}|X(u)|\\
	&+ \Bigl(\frac{2}{t}\Bigr)^{2/(p-1)}|\lambda ||X(u)|^{2/(p-1)}+\Bigl(\frac{2|\lambda |}{t}\Bigr)^{2/(p-2)}|\lambda |.
	\end{split}
	\end{align}
	Note that the assumption $p>2$ ensures that
	\begin{align}
	\label{add:eq---8}
	\max\Bigl(1+\frac{1}{p-1},\frac{2}{p-1}\Bigr)<2.
	\end{align}
	Using Jensen's inequality yields
	\begin{align}
	\begin{split}\label{add:eq5}
	\e \Psi_{\lambda ,\gamma }^{t,u}(u,X(u))
	&\leq K_1\bigl(\bigl(\e|X(u)|^2\bigr)^{1/2+1/2(p-1)}+\bigl(\e|X (u)|^2\bigr)^{1/2}+\bigl(\e|X(u)|^2\bigr)^{1/(p-1)}+1\bigr)
	\end{split}
	\end{align}
	for $$
	K_1:= \Bigl(\frac{2}{t}\Bigr)^{1/(p-1)}+\Bigl(\frac{2|\lambda |}{t}\Bigr)^{1/(p-2)}+\Bigl(\frac{2}{t}\Bigr)^{2/(p-1)}|\lambda |+\Bigl(\frac{2|\lambda |}{t}\Bigr)^{2/(p-2)}|\lambda|.
	$$
	Let
$
	v(s):=\partial_x\Psi_{\lambda,\gamma}^{t,u}(s,X(s))
$
	and
	$
	L:=\int_0^u\gamma (s)v(s)ds.
	$
	Since
	\begin{align*}
	\e|X(u)|^2&\leq 16\bigl(\e L^2+\e|W(u)|^2\bigr)=16\bigl(\e L^2+u\bigr),
	\end{align*}
	it follows that by using the trivial bound $(x+y)^a\leq 2^a(x^a+y^a)$ for all $a,x,y\geq 0,$
$
	\bigl(\e|X(u)|^2\bigr)^a\leq 32^a(\e L^2)^a+32^au^a.
$
	From this and \eqref{add:eq5}, 
	\begin{align}\label{add:eq6}
	\e \Psi_{\lambda ,\gamma }^{t,u}(u,X(u))&\leq K_2\bigl(\bigl(\e L^2\bigr)^{1/2+1/2(p-1)}+\bigl(\e L^2\bigr)^{1/2}+\bigl(\e L^2\bigr)^{1/(p-1)}+1\bigr),
	\end{align}
	where $K_2$ depends on $t,u,p,\lambda$.
	By comparing the representation \eqref{rep4} of $\Psi_{\lambda ,\gamma }^{t,u}(0,0)$ with the optimal control process $v(s)=\partial_x\Psi_{\lambda,\gamma}^{t,u}(s,X(s))$  against the zero control process, we obtain that 
	\begin{align*}
	\e\Psi_{\lambda ,\gamma }^{t,u}(u,X(u))-\int_0^u \gamma (s)\e v(s)^2ds=\Psi_{\lambda ,\gamma }^{t,u}(0,0)\geq \e f_\lambda^{t,u}(\sqrt{2}W(u))\geq 0,
	\end{align*}
	which implies that, by the Cauchy-Schwarz inequality, 
	\begin{align*}
	\e\Psi_{\lambda ,\gamma }^{t,u}(u,X(u))\geq \int_0^u \gamma (s)\e v(s)^2ds\geq \frac{\e L^2}{\int_0^u\gamma(s)ds},
	\end{align*}
	From this  and \eqref{add:eq6}, we arrive at
	\begin{align*}
	\e L^2&\leq K_2\Bigl(\int_0^u\gamma(s)ds \Bigr)\bigl(\bigl(\e L^2\bigr)^{1/2+1/2(p-1)}+\bigl(\e L^2\bigr)^{1/2}+\bigl(\e L^2\bigr)^{1/(p-1)}+1\bigr).
	\end{align*}
	From \eqref{add:eq---8}, $1/2+1/2(p-1)$ and $1/(p-1)$ are strictly less than $1$. Now, if $\e L^2\geq 1,$ dividing $(\e L^2)^\delta$ on the both sides of the above inequality deduces that
	$$
	\bigl(\e L^2\bigr)^{1-\delta}\leq 4K_2\int_0^u\gamma(s)ds 
$$
	for $
	\delta:=\max\bigl(1/2+1/2(p-1),1/(p-1)\bigr).
	$
	Hence, we arrive at the bound,
	$$
	\e L^2\leq K_3:=\max\Bigl(1,\Bigl(4K_2\int_0^u\gamma(s)ds \Bigr)^{1/(1-\delta)}\Bigr).
	$$ 
	Our proof then follows by using
	$
	\e|X(u)|^2\leq 16 (\e L^2+u).
$
\end{proof}

We now establish \eqref{add:eq---6}. Recall the optimizer $(\lambda_M,\gamma_M)$ of $\mathcal{P}_{t,u,M}$ from Proposition \ref{add:prop1}. Note that from the continuity of $\mathcal{P}_{t,u,M}$ in Proposition \ref{property2}, we can approximate $\mathcal{P}_{t,u,M}(\lambda_M,\gamma_M)$ by considering $\mathcal{P}_{t,u,M}(\lambda_M,\gamma)$ for $\gamma\in \mathcal{N}_u^d$ and by letting $\gamma$ to be as close as we want to $\gamma_M$ under the norm $\int_0^u|\gamma(s)-\gamma_M(s)|ds.$ Hence, without loss of generality, we may assume that $\gamma_M\in \mathcal{N}_u^d.$ Recall from \eqref{add:prop1:eq1} that $\lambda_M$ and $\int_0^u\gamma_Mds$ are bounded by a constant $C$ independent of $M.$

As $p>2$, in the definition of $f_\lambda^t$, the term $t|r|^p$ dominates the linear term $xr$ and the quadratic term $\lambda r^2$. So,  for any $L>0,$ there exists some $M_L>0$ such that 
\begin{align*}
f_{\lambda}^t(x)=f_{\lambda}^{t,M}(x),\,\,\ \text{ for all }  |x|\leq L,|\lambda|\leq C,M\geq M_L.	
\end{align*}
Hence, as long as $M\geq M_L,$  we can write that for any $x\in \mathbb{R},$
\begin{align*}
f_{\lambda}^t(x)&=f_{\lambda}^{t,M}(x)1_{\{|x|\leq L\}}+f_{\lambda}^t(x)1_{\{|x|>L\}}\leq f_{\lambda}^{t,M}(x)+f_\lambda^t(x)1_{\{|x|>L\}}.
\end{align*}
Let $X_M$ be the stochastic process $X$ defined in Proposition \ref{rep3} associated to $\gamma=\gamma_M.$ From this inequality, letting $v_M(s):=\partial_x\Psi_{\lambda_M,\gamma_M}^{t,u}(s,X_M(s)),$ and using Propositions \ref{property1} and \ref{rep3}, we have
\begin{align}
\begin{split}\label{add:thm0:proof:eq2}
\Psi_{\lambda_M ,\gamma_M}^{t,u}(0,0)
&=\e\Psi_{\lambda_M ,\gamma_M }^{t,u}(u,X_{M}(u))-\int_0^u \gamma_M (s)\e v_M(s)^2ds\\
&\leq \Psi_{\lambda_M ,\gamma_M }^{t,u,M}(0,0)+\e\Psi_{\lambda_M ,\gamma_M }^{t,u}(u,X_M(u))1_{\{|X_M(u)|>L\}}.
\end{split}
\end{align}
To handle the second term, we need some moment controls.
Note that $p>2$ implies 
$$
\zeta:=\frac{2}{\max\bigl(1+\frac{1}{p-1},\frac{2}{p-1}\bigr)}\in (1,2).
$$ 
From \eqref{thm0:proof:eq3} and Lemma \ref{lem7}, we can apply Jensen's inequality along with the bounds $|\lambda_M|\leq C$ and $\int_0^u\gamma_Mds\leq C$ to get
\begin{align*}
\e\bigl|\Psi_{\lambda_M ,\gamma_M }^{t,u}(u,X_{M}(u))\bigr|^{\zeta}
&\leq 
K\bigl(\bigl(\e |X_M(u)|^2\bigr)^{\frac{\zeta}{2}\bigl(1+\frac{1}{p-1}\bigr)}+\bigl(\e |X_M(u)|^2\bigr)^{\frac{\zeta}{2}}+\bigl(\e |X_M(u)|^2\bigr)^{\frac{\zeta}{p-1}}+1\bigr)\leq K',
\end{align*}
where $K$ and $K'$ are constants depending only on $p,t,u$.
On the other hand, from the Markov inequality and Lemma \ref{lem7},
\begin{align*}
\p\bigl(|X_M(u)|>L\bigr)&\leq \frac{\e |X_M(u)|^2}{L^2}\leq \frac{K''}{L^2},
\end{align*}
where again $K''$ depends only on $p,t,u.$ Applying  the H\"older inequality to \eqref{add:thm0:proof:eq2} and using the above two inequalities yield that
\begin{align*}
\Psi_{\lambda_M ,\gamma_M }^{t,u}(0,0)&\leq \Psi_{\lambda_M ,\gamma_M }^{t,u,M}(0,0)+\bigl(\e\bigl|\Psi_{\lambda_M ,\gamma_M }^{t,u}(u,X_M(u))\bigr|^{\zeta}\bigr)^{1/\zeta}\p\bigl(|X_M(u)|>L\bigr)^{1/\zeta'}\\
&\leq \Psi_{\lambda_M ,\gamma_M }^{t,u,M}(0,0)+\frac{{K'}^{1/\zeta}{K''}^{1/\zeta'}}{L^{2/\zeta'}},
\end{align*}
where $\zeta'$ is the H\"older conjugate exponent of $\zeta.$ As a result, for any $M\geq M_L,$
\begin{align*}
\inf_{(\lambda,\gamma)\in \mathbb{R}\times\mathcal{N}_u^d}\mathcal{P}_{t,u}(\lambda,\gamma)&\leq \mathcal{P}_{t,u}(\lambda_M,\gamma_M)\\
&\leq \mathcal{P}_{t,u,M}(\lambda_M,\gamma_M)+\frac{{K'}^{1/\zeta}{K''}^{1/\zeta'}}{L^{2/\zeta'}}=\inf_{(\lambda,\gamma)\in \mathbb{R}\times\mathcal{N}_u}\mathcal{P}_{t,u,M}(\lambda,\gamma)+\frac{{K'}^{1/\zeta}{K''}^{1/\zeta'}}{L^{2/\zeta'}}.
\end{align*}
Since this is valid for any $L>0$ and $K',K''$ are independent of $L,$ sending $M\to \infty$ and then $L\to \infty$ completes the proof of \eqref{add:eq---6}.

\subsection{Upper bound}\label{12.7}
Next, we verify that 
\begin{align}\label{add:eq---7}
L_{p,u}(t)\leq \inf_{\lambda\in \mathbb{R},\gamma\in \mathcal{N}_u^d}\mathcal{P}_{t,u}(\lambda,\gamma).
\end{align}
From Proposition \ref{extra:prop1}, we have seen that
\begin{align}\label{add:eq--1}
L_{p,u}(t)\leq  \lim_{\beta \to\infty}\limsup_{n\to\infty}\e F_{n,t,u,\beta^{-2}}(\beta).
\end{align}
To control the right-hand side, recall $\mathcal{P}_{\beta,t,u}(\lambda,\alpha)$ for $(\lambda,\alpha)\in \mathbb{R}\times\mathcal{M}_u^d$ from Table \ref{tab:1}. The Replica Symmetry Breaking bound, due to Guerra (see \cite{guerra2003broken} and \cite{Pan05}), states that for any $\beta>0,$ $\varepsilon>0$, and $(\lambda,\alpha)\in \mathbb{R}\times\mathcal{M}_u^d$, we have that
\begin{align}\label{guerra}
\e F_{n,t,u,\varepsilon}(\beta)\leq \mathcal{P}_{\beta,t,u}(\lambda,\alpha)+C_0\beta \varepsilon,
\end{align}
where $C_0>0$ is a universal constant independent of $n,\beta,$ and $\varepsilon$. Recall $f_{\lambda}^{\beta,t}(x)$ and $f_{\lambda}^{t}(x)$ from Table \ref{tab:1}.
To relate $\mathcal{P}_{\beta,t,u}(\lambda,\alpha)$ to $\mathcal{P}_{t,u}(\lambda,\gamma),$ fix $(\lambda,\gamma)\in \mathbb{R}\times \mathcal{N}_u^d$ and let $\alpha(s)=\beta^{-1}\gamma(s)1_{[0,u)}(s)+1_{\{u\}}(s).$ Note that $\alpha\in \mathcal{M}_u^d$ as long as $\beta$ is large enough.  
For any $0<\delta<t,$ since
\begin{align*}
f_\lambda^{\beta,t}(x)&=\frac{1}{\beta}\log \frac{\int_{-\infty}^\infty e^ {\beta(rx+\lambda r^2-(t-\delta)|r|^p)}e^{-\beta\delta|r|^p}dr}{\int_{-\infty}^\infty e^{-\beta\delta|r|^p}dr}+\frac{1}{\beta}\log \int_{-\infty}^\infty e^{-\beta\delta|r|^p}dr\\
&\leq f_\lambda^{t-\delta}(x)+\frac{1}{\beta}\log \int_{-\infty}^\infty e^{-\beta\delta|r|^p}dr
\end{align*}
and $\int_{-\infty}^\infty e^{-\beta\delta|r|^p}dr=(\beta \delta)^{-1/p}\int_{-\infty}^\infty e^{-|s|^p}ds$, it can be argued by using the Hopf-Cole transformation in Lemma \ref{add:lem2}  iteratively  that
\begin{align*}
\Phi_{\lambda,\alpha}^{\beta,t,u}(0,0)\leq \Psi_{\lambda,\gamma}^{t-\delta,u}(0,0)
-\frac{1}{p\beta}\log \beta \delta+\frac{1}{\beta}\log \int_{-\infty}^\infty e^{-|s|^p}ds.
\end{align*}
Note $\int_0^u s\beta \alpha(s)ds=\int_0^u s\gamma(s)ds$. From \eqref{add:eq--1} and \eqref{guerra}, after sending $\beta$ to infinity,
$
L_{p,u}(t)\leq \mathcal{P}_{t-\delta,u}(\lambda,\gamma).
$
Finally, using \eqref{add:eq---11}, \eqref{appendix:lem1:eq2}, and the dominated convergence theorem yields $\lim_{\delta\downarrow 0}\mathcal{P}_{t-\delta,u}(\lambda,\gamma)=\mathcal{P}_{t,u}(\lambda,\gamma)$ and consequently,
$
L_{p,u}(t)\leq \mathcal{P}_{t,u}(\lambda,\gamma).
$
Taking infimum in this inequality validates \eqref{add:eq---7}.

\section{Open Questions}\label{sec13}

In this section, we mention a few open problems and further directions.

\begin{enumerate}
\item Compute the order and the limiting distribution  of the fluctuation of the $\ell_p$-Grothendieck problem. The asymptotic fluctuation for $p=2$ (the largest eigenvalue of the GOE) is known to follow the Tracy-Widom distribution with the usual $n^{-1/6}$ scaling \cite{TW1, TW2}. The proof of Theorem \ref{prop:p=1} suggests that the limiting fluctuation in
the case $p=1$ could possibly be Gumbel. For the case $2<p<\infty$, since $L_{n,p}(t)$ is essentially a SK model with $\ell_p$ external field, it seems reasonable to believe that the limiting fluctuation of $L_{n,p}$ is Gaussian (see \cite{CDP17}) and in view of \eqref{add:lem5:eq0}, $G_{n,p}$ should also be Gaussian up to a transformation. When $p=\infty$, the $\ell_p$-Grothendieck problem becomes  the ground state energy of the SK model and the numerical studies, see, e.g., \cite{PM08}, indicate that the limiting fluctuation should follow the Gumbel distribution, even though a rigorous proof remains elusive. We do not have a prediction in the case $1<p<2$.
To analyze this case, one would naturally use the fact that the global optimizer lies close to one of the approximate optimizers in $\mathcal{O}$, described in \eqref{add:eq---1}. However,   it may be possible that determining the limiting distribution of fluctuation requires finer approximation for the global optimizer (with better error bound).


    \item For the $\ell_p$-Grothendieck problem for $2<p<\infty$, show that with high probability, there exist exponentially many near optimizers that are nearly mutually orthogonal to each other. 
    

\item  For the $\ell_p$-Grothendieck problem for $2<p<\infty$,  devise a polynomial-time algorithm that outputs an approximate optimizer with high probability. For $p=\infty$ (SK model), \cite{alaoui2020optimization,montanari2019optimization} used the Approximate Message Passing  (AMP) algorithm to find a near ground state in polynomial time with high probability (see also \cite{subag18}). 

\item Improve the delocalization bound  \eqref{thm2:eq3} for $2 < p < \infty$. In this case, we expect that an optimizer $x_*$ in the unit $\ell_p$-ball should satisfy
 $\| x_*\|_\infty = O( (\log n)^C n^{-1/p})$ for some constant $C>0$ with high probability. 

\item Study the $\ell_p$-Grothendieck problem for random i.i.d.\ matrices with more general entry distribution (for example, with subgaussian distribution). Do we have universality if $2 < p< \infty$?

	\item 
	 Find the limit of the operator norm $\| G_n\|_{p \to q}$ for all $1 \leq  p,q \leq \infty$. In view of Proposition~\ref{prop:operator_norm}, this problem is yet to be resolved  for pairs $(p, q)$ satisfying  $(p,q)\in (1,\infty]\times [1,\infty)$ and $\mathfrak{p}(q,p^*)= \mathfrak{p}(p^*,q)$, which is the set of pairs satisfying that
	 $$
	 \mbox{$(i)$ $(p,q)\in  (1,2]\times[2,\infty)$ with $q=p^*$ \ \ or  \ \ $(ii)$ $(p,q)\in [2,\infty]\times [1,2]$.}
	 $$
	 A special case of (i) and (ii) is when $1 < p \leq \infty$ and $q = p^*$, which corresponds to the $p$-to-$p^*$ norm  and it can be expressed as a bipartite $\ell_p$-Grothendieck problem via
	\[ \| G_n\|_{p \to p^*} = \max_{x, y \in \mathbb{R}^n:\|x\|_p \le 1, \|y\|_p \le 1 }\la G_nx,y\ra. \]
For $1 < p <2$, an argument similar to the one presented in Theorem~\ref{eq:thm_p<2} might be useful to find the limit of $\| G_n\|_{p \to p^*}$.  For $p =\infty$, the limit of $\| G_n\|_{p \to p^*} = \| G_n\|_{\infty \to 1} $ is known as the ground state energy of the bipartite SK model and is conjectured to admit a minmax-type  Parisi formula (see \cite{BGG11,mourrat20}). For $2<p<\infty,$ we expect that the same formulation should still be valid.

	\begin{remark}\rm In a recent work, Dhara-Mukherjee-Ramanan \cite{DMR20} studied the $p$-to-$q$-norm for $n \times n$ symmetric random matrix $A_n$ with i.i.d.\ non-negative entries, including the adjacency matrices of Erd\"os-R\'enyi random graphs. The $p$-to-$q$-norm is related to the $\ell_p$-Grothendieck problem in the following way (see \cite[Proposition 2.14]{DMR20}):
	for $p \ge 2$ and for any $n \times n$ matrix $A_n$ such that $A_n^TA_n$ is an irreducible matrix with nonnegative entries, it holds that 
	\[ \|A_n\|_{p \to p^*}= \max_{x\in \mathbb{R}^n:\|x\|_p \le1}\la A_n x,x\ra . \]
	In \cite{DMR20},   Gaussian fluctuation  for $\| A_n\|_{p \to q}$ if   $1 \le q \le p < \infty$ was established, which includes, as a special case, the $\ell_p$-Grothendieck problem for $A_n$ for $p \ge 2$. However, the proof uses a non-linear power iteration method, originally introduced by Boyd \cite{boyd74},  to approximate the $p$-to-$q$-norm. It  relies heavily  on the non-negativity of the entries of $A_n$ and cannot be directly applied to the Gaussian case. Also, for $1 \le q \le p < \infty$, the optimizer for $\| A_n\|_{p \to q}$ is close to the scaled constant vector. Obviously, this does not  hold in the Gaussian case. 
\end{remark}

\end{enumerate}

\section*{Acknowledgement} WKC and AS would like to thank Souvik Dhara for explaining the results of \cite{DMR20} and for helpful discussion. 

\appendix

\section{Stability of H\"older's inequality}\label{appA}

In this appendix, we gather some quantitative results on the inequalities used in the proof of Theorem \ref{extra:thm1}. The first lemma provides a stability bound for H\"older's inequality.

\begin{lemma} \label{lem:holder_stability}
	Let $1 < p \le 2$ and let $w$  be a non-zero vector in $\r^n$. Let $v$ be the vector with unit $\ell_p$-norm such that $\la v, w\ra=\| w\|_{p^*}$, i.e., $v$ satisfies $$\mathrm{sgn}(v_i) |v_i|^p = \frac{\mathrm{sgn}(w_i)  |w_i|^{p^*}}{ \| w\|_{p^*}^{p^*}},\,\,\forall i.$$   Then for any $u \in \r^n$ with $\| u \|_p=1$, 
	\[ \la u, w\ra \le \| w\|_{p^*} \Big (1  -  \frac{p^2}{16p^*}\| u - v \|_p^2\Big).\]
\end{lemma}
\begin{proof} First of all, we claim that for $1 < p \le 2$ and $a, b \in \r$, we have the following stability bound for Young's inequality,
	\[ \frac{|a|^p}{p} +  \frac{|b|^{p^*}}{p^*}  \ge ab + \frac{1}{2p^*} \big(|a|^{p/2} \mathrm{sgn}(a) - |b|^{p^*/2}\mathrm{sgn}(b) \big)^2.   \]
		The case when both $a$ and $b$ have the same sign follows from  \cite[Lemma 2.1]{Aldaz}. If $a$ and $b$ have opposite signs, then we need to show, after ignoring the negative term $ab$, that
	\[ \frac{|a|^p}{p} +  \frac{|b|^{p^*}}{p^*}  \ge  \frac{1}{2p^*} \big(|a|^{p/2}  + |b|^{p^*/2}\big)^2,   \]
	which follows from the inequality $ (x+y)^2 \le 2 (x^2 +y^2)$ and the fact that $p \le p^*$. This completes the proof of our claim.
	
	  To prove our main result, without loss of generality, assume that $\|w\|_{p^*}=1$. Applying our claim to $u_i$ and $w_i$ and then summing over $i$, we obtain
	\begin{align*}
	\la u, w\ra &\le1  - \frac{1}{2p^*} \sum_i \big(|u_i|^{p/2} \mathrm{sgn}(u_i) - |w_i|^{p^*/2} \mathrm{sgn}(w_i)\big)^2 \\
	&=  1  - \frac{1}{2p^*} \sum_i \big(|u_i|^{p/2} \mathrm{sgn}(u_i) - |v_i|^{p/2}\mathrm{sgn}(v_i) \big)^2.  
	\end{align*}
	For $0 \le \theta \le 1$,  we have
	\begin{equation}\label{eq:calc1}
	1 - \theta^{p/2}  = \int_\theta^1 \tfrac{p}{2}  x^{p/2-1} dx \ge \tfrac{p}{2}(1 - \theta) \ \  \text{ and }  \ \ 1 + \theta^{p/2}  \ge 1 + \theta \ge \tfrac{p}{2}(1 + \theta). 
	\end{equation}
	Set $a_i = \min (|u_i|, |v_i|)$ and $b_i = \max(|u_i|, |v_i|)$. Also, let $K = \sum_i b_i^p \in [1, 2]$. We now write
	\begin{align*}
	&\ \ \sum_i \big(|u_i|^{p/2} \mathrm{sgn}(u_i) - |v_i|^{p/2}\mathrm{sgn}(v_i) \big)^2 \\
    &= K \sum_{i:b_i \ne 0} \frac{b_i^p}{K}\Bigl( 1  - \mathrm{sgn}(u_i v_i) \Bigl(  \frac{a_i}{b_i} \Bigr)^{p/2} \Bigr)^2 \ge \frac{Kp^2}{4}  \sum_{i : b_i \ne 0} \frac{b_i^p}{K}\Bigl( 1  - \mathrm{sgn}(u_i v_i) \frac{a_i}{b_i} \Bigr)^2    \\
	&\ge \frac{Kp^2}{4}  \Bigl(  \sum_{i: b_i \ne 0} \frac{b_i^p}{K} \Big | 1  - \mathrm{sgn}(u_i v_i) \frac{a_i}{b_i} \Big|^p  \Bigr)^{2/p}   = \frac{p^2}{4K^{2/p-1}} \Bigl(  \sum_{i: b_i \ne 0}  \Big | b_i  - \mathrm{sgn}(u_i v_i) a_i \Big|^p  \Bigr)^{2/p} \\
	&= \frac{p^2}{4K^{2/p-1}} \Bigl(  \sum_{i: b_i \ne 0}  \Bigl| u_i  - v_i \Bigr |^p  \Bigr)^{2/p} \ge \frac{p^2}{8} \| u - v\|_p^2,
	\end{align*}  
	where the first  inequality follows from \eqref{eq:calc1} and the second one uses Jensen's inequality. This finishes our proof.
\end{proof}

Next we show that if a $\ell_p$-unit vector ($1< p< 2$) stays away from the coordinate vectors, then its $\ell^2$-norm must be strictly less than one. 

\begin{lemma}\label{lem:deficient_l_2_norm}
	Fix $1 < p < 2$. There exists constant  $c>0$ such that if $x \in \r^n$ satisfies  $\|x \|_p=1$ and $\| x  \pm e_i\|_p \ge \delta$ for all $i$ and for some  $\delta>0$, then 
	$ \|x \|_2 \le 1 - c\delta^p.$
\end{lemma}
\begin{proof}
	It is enough to prove the bound for $0 < \delta  \le \delta_0$ for sufficiently small constant $\delta_0$. 
	We first claim that $\|x\|_\infty \le 1 -\delta^p/3$. Suppose, if possible, that $|x_i| \ge 1-\delta^p/3$ for some $i$. For definiteness, assume that $x_i \ge 1-\delta^p/3$ as the case  $x_i \le -(1-\delta^p/3)$ can be dealt similarly. Using the identity $\sum_{j \ne i} |x_j|^p = 1 - |x_i|^p$, we obtain that
	\begin{align*}
	\| x - e_i\|_p^p   = |1- x_i|^p + 1 - |x_i|^p < |1- x_i| + 1 - |x_i|^2 =(1-x_i)(2+x_i)\leq \frac{\delta^p}{3}\cdot 3=\delta^p,
	\end{align*}
	contradicting the assumption that $\|x -e_i\|_p \ge \delta$. Under the constraints $\|x\|_\infty \le 1 -\delta^p/3$ and $\|x\|_p=1$, the $\ell_2$-norm of $x$ is maximized if $x$ is supported only on two coordinates, one of them being equal to $1 -\delta^p/3$ in magnitude and the other non-zero component of $x$ must have absolute value $(1 - (1 -\delta^p/3)^p )^{1/p} $. Hence, 
	\begin{align*}
	\|x\|_2^2 &\le (1 -\delta^p/3)^2 + (1 - (1 -\delta^p/3)^p )^{2/p} \le (1 -\delta^p/3)^2 + (1 - (1 -\delta^p/3)^2 )^{2/p} \\
	&\le 1 -2 \delta^p/3  + (\delta^p/3)^2 + (2 \delta^p/3 )^{2/p} \le 1 - \delta^p/3,
	\end{align*}
	by choosing $\delta_0$ sufficiently small. This implies that $\|x\|_2 \le (1 - \delta^p/3)^{1/2} \le  1 - \delta^p/6$.
\end{proof}

\section{Analytic results of the Parisi PDE}\label{AppB}
For $(\lambda,\gamma)\in \mathbb{R}\times\mathcal{N}_u^d,$ recall the PDE solution $\Psi_{\lambda,\gamma}^{t,u}$ defined in \eqref{hopfcole}. In this appendix, we will gather some regularity properties of this PDE solution and provide the proof of Proposition \ref{rep3}. First of all, the boundary condition $f_\lambda^t$ has the following regularity.

\begin{lemma}\label{lem3}
	If $\lambda =0,$ then for all $x\neq 0,$
	\begin{align}
	\begin{split}\label{add:eq-1}
	f_{\lambda}^t(x)&=\frac{(p-1)}{p}\frac{1}{p^{\frac{1}{p-1}}t^{\frac{1}{p-1}}}|x|^{1+\frac{1}{p-1}},\\
	\Bigl|\frac{d}{dx}f_\lambda^t(x)\Bigr|&=\frac{1}{p^{\frac{1}{p-1}}t^{\frac{1}{p-1}}}|x|^{\frac{1}{p-1}}.
	\end{split}
	\end{align}
	If $\lambda\neq 0,$ then $f_{\lambda}^t$ is twice differentiable on $\mathbb{R}\setminus\{0\}$ and for $x\neq 0,$ 
	\begin{align}
	\begin{split}\label{lem3:eq2}
	0<f_\lambda^t(x)&\leq 
	\Bigl(\frac{2}{t}\Bigr)^{1/(p-1)}|x|^{1+1/(p-1)}+\Bigl(\frac{2|\lambda|}{t}\Bigr)^{1/(p-2)}|x|\\
	&+ \Bigl(\frac{2}{t}\Bigr)^{2/(p-1)}|\lambda||x|^{2/(p-1)}+\Bigl(\frac{2|\lambda|}{t}\Bigr)^{2/(p-2)}|\lambda|
	\end{split}
	\end{align}
	and
	\begin{align}\label{lem3:eq3}
	\Bigl|\frac{d}{dx}f_\lambda^t(x)\Bigr|&\leq \max\Bigl(\Bigl(\frac{2|x|}{t}\Bigr)^{1/(p-1)},\Bigl(\frac{2|\lambda|}{t}\Bigr)^{1/(p-2)}\Bigr).
	\end{align}

\end{lemma}

\begin{proof} When $\lambda=0$, the assertion \eqref{add:eq-1} can be checked by a straightforward computation. For the remainder of the proof, we will only focus on the case $\lambda\neq 0.$ Note that $f_\lambda^t$ is an even function. 
	To prove our assertion, it suffices to assume that $x>0.$ Evidently, $f_\lambda^t(x)>0.$ Note that for any $r>0,$
	\begin{align*}
	rx+\lambda r^2-t|r|^p&>(-r)x+\lambda(-r)^2-t|-r|^p,
	\end{align*}
	which implies that
	$
	f_\lambda^t(x)=\sup_{r>0}\bigl(xr+\lambda r^2-tr^p\bigr).
$
	Here, for $r> 0,$
	\begin{align}\label{zero}
	\frac{d}{dr}(xr+\lambda r^2-tr^p)&=x+2\lambda r-pt r^{p-1}.
	\end{align}
	Since $x>0,$ $0<r\mapsto x+2\lambda r$ is a linear function, and $0<r\mapsto ptr^{p-1}$ is strictly convex, we see that \eqref{zero} has a unique zero, $r(x)>0$. In addition, since \eqref{zero} equals $x>0$ at $r\to 0^+$ and it is negative when $r$ is sufficient large, these  imply that $r(x)$ is an unique optimizer of $f_\lambda^t(x).$ This also implies that $0<x\mapsto r(x)$ must be continuous.
	Now, if
	\begin{align*}
	r(x)\geq \max\Bigl(\Bigl(\frac{2x}{t}\Bigr)^{1/(p-1)},\Bigl(\frac{2|\lambda|}{t}\Bigr)^{1/(p-2)}\Bigr),
	\end{align*}
	then
	\begin{align*}
	xr(x)-t\frac{r(x)^p}{2}<0 \,\,\mbox{and}\,\,	\lambda r(x)^2-t\frac{r(x)^p}{2}<0.
	\end{align*}
	Together they imply that 
$
	f_\lambda^t(x)=xr(x)+\lambda r(x)^2-tr(x)^p<0,
$
	a contradiction. Hence, we must have that
	\begin{align}\label{add:eq3}
	r(x)< \max\Bigl(\Bigl(\frac{2x}{t}\Bigr)^{1/(p-1)},\Bigl(\frac{2|\lambda|}{t}\Bigr)^{1/(p-2)}\Bigr).
	\end{align}
	Consequently, dropping $t|r|^p$ and using this inequality, we arrive at \eqref{lem3:eq2} since
	\begin{align*}
	f_\lambda^t(x)&\leq 
	x\max\Bigl(\Bigl(\frac{2x}{t}\Bigr)^{1/(p-1)},\Bigl(\frac{2|\lambda|}{t}\Bigr)^{1/(p-2)}\Bigr)+|\lambda |\max\Bigl(\Bigl(\frac{2x}{t}\Bigr)^{2/(p-1)},\Bigl(\frac{2|\lambda|}{t}\Bigr)^{2/(p-2)}\Bigr)\\
	&\leq \Bigl(\frac{2}{t}\Bigr)^{1/(p-1)}x^{1+1/(p-1)}+\Bigl(\frac{2|\lambda|}{t}\Bigr)^{1/(p-2)}x+ \Bigl(\frac{2}{t}\Bigr)^{2/(p-1)}|\lambda|x^{2/(p-1)}+\Bigl(\frac{2|\lambda|}{t}\Bigr)^{2/(p-2)}|\lambda|.
	\end{align*}
	
	To show \eqref{lem3:eq3}, note that $$
	\frac{d}{dr}(x+2\lambda r-ptr^{p-1})\Bigl|_{r=r(x)}=2\lambda-pt(p-1)r(x)^{p-2}.
	$$
	If $\lambda<0$, then this derivative is strictly less than $2\lambda<0;$ if $\lambda>0$, plugging 
		\begin{align*}
	x+2\lambda r(x)=ptr(x)^{p-1}
	\end{align*}
	into the above equation gives that
	$$
	\frac{d}{dr}(x+2\lambda r-ptr^{p-1})\Bigl|_{r=r(x)}=-2(p-2)\lambda-\frac{x}{r(x)}(p-1)<0.
	$$
	From these and the implicit function theorem, $r(x)$ is differentiable on $(0,\infty)$ and consequently, 
	\begin{align*}
	\frac{d}{dx}f_\lambda^t(x)&=\frac{d}{dx}\bigl(xr(x)+\lambda r(x)^2-tr(x)^p\bigr)
	=r(x)+r'(x)\bigl(x+2\lambda r(x)-pt r(x)^{p-1}\bigr)=r(x).
	\end{align*} 
	The inequality \eqref{lem3:eq3} then follows from \eqref{add:eq3}. 
\end{proof}


\begin{lemma}\label{add:lem2}
	Let $a\in (0,1)$ and $b\in (0,\infty).$ Let $f$ be a nonnegative, even, and convex function on $\mathbb{R}$.  Assume that there exists a constant $C>0$ such that
	\begin{align}
	\begin{split}\label{lem2:proof:eq1}
	f(x)&\leq C(1+|x|^{a+1}),\,\,x\in \mathbb{R}.
	\end{split}
	\end{align}
	Let $m>0.$ For any $(s,x)\in [0,b]\times \mathbb{R},$ set
	\begin{align}\label{add:eq---10}
	F(s,x)&=\frac{1}{m}\log \e \exp m f(x+\sqrt{s}z),
	\end{align}
	where $z$ is standard normal. The following statements hold:
	\begin{itemize}
		\item[$(i)$] For any $s\in [0,b],$ $F(s,\cdot)$ is nonnegative, even, and convex on $\mathbb{R}$.
		\item[$(ii)$] $F\in C^{1,2}((0,b]\times \mathbb{R}).$
		\item[$(iii)$] $F$ satisfies
		\begin{align*}
		\partial_sF(s,x)&=\frac{1}{2}\bigl(\partial_{xx}F(s,x)+m\bigl(\partial_xF(s,x)\bigr)^2\bigr),\,\,(s,x)\in (0,b]\times \mathbb{R},\\
		\lim_{s\downarrow 0,x\to x_0}F(s,x)&=f(x),\,\,\forall x_0\in \mathbb{R}.
		\end{align*}
			\item[$(iv)$] There exists a constant $C'>0$ depending on $a,b,m,C$ such that for any $(s,x)\in (0,b]\times \mathbb{R},$
		\begin{align}
		\begin{split}\label{lem2:proof:eq3}
		|F(s,x)|&\leq C'(1+|x|^{a+1}).
		\end{split}
		\end{align}
		\item[$(v)$] If additionally, $f$ is differentiable on $\mathbb{R}\setminus\{0\}$ and there exists a constant $C$ such that
		\begin{align}
		\begin{split}\label{lem2:proof:eq2}
		|f'(x)|&\leq C(1+|x|^a),\,\quad x\in \mathbb{R}\setminus\{0\},
		\end{split}
		\end{align}
		then there exists a constant $C'$ depending on $a,b,m,C$ such that for any $(s,x)\in (0,b]\times \mathbb{R},$
		\begin{align}
		\begin{split}\label{lem2:proof:eq4}
		|\partial_xF(s,x)|&\leq C'(1+|x|^a).
		\end{split}
		\end{align}
	\end{itemize}

\end{lemma}

\begin{proof} First of all, note that $F$ is well-defined due to \eqref{lem2:proof:eq1} and $a<1.$ To prove $(i),$ the nonnegativity and the evenness of $F$ follow directly from those of $f$; using the H\"older inequality and the convexity of $f$ yield that  for all $0\leq s\leq b$, $0\leq \ell\leq 1$, and $x,y\in \mathbb{R},$
	\begin{align*}
	F(s,\ell x+(1-\ell)y)&= \frac{1}{m}\log \e \exp m f(\ell x+(1-\ell)y+\sqrt{s} z)\\
	&\leq \frac{1}{m}\log \e \exp m \bigl(\ell f( x+\sqrt{s}z)+(1-\ell)f(y+\sqrt{s}z)\bigr)\leq  \ell F(s,x)+(1-\ell)F(s,y),
	\end{align*}
	establishing the convexity of $F(s,\cdot).$
	To show $(ii)$ and $(iv)$, write 
	\begin{align*}
	F(s,x)&=\frac{1}{m}\log \frac{1}{\sqrt{2\pi s}}\int e^{mf(y)-\frac{(y-x)^2}{2s}}dy.
	\end{align*}
	From this expression, \eqref{lem2:proof:eq1}, and the dominated convergence theorem, we see that $F$ lies in the class $C^{1,2}((0,b]\times \mathbb{R})$. Furthermore, for any $(s,x)\in (0,b]\times \mathbb{R}$,
	\begin{align*}
	\partial_xF(s,x)&=\frac{1}{sm}\frac{\int_{\mathbb{R}}(y-x)\varphi(y)dy}{\int_{\mathbb{R}}\varphi(y)dy},\\
		\partial_{xx}F(s,x)&=-\frac{1}{sm}+\frac{1}{s^2m}\frac{\int_{\mathbb{R}}(y-x)^2\varphi(y)dy}{\int_{\mathbb{R}}\varphi(y)dy}-\frac{1}{s^2m}\Bigl(\frac{\int_{\mathbb{R}}(y-x)\varphi(y)dy}{\int_{\mathbb{R}}\varphi(y)dy}\Bigr)^2,\\
		\partial_sF(s,x)&=-\frac{1}{2sm}+\frac{1}{2s^2m}\frac{\int_{\mathbb{R}}(y-x)^2\varphi(y)dy}{\int_{\mathbb{R}}\varphi(y)dy}
	\end{align*}
	for $\varphi(y):=e^{mf(y)-(y-x)^2/(2s)}.$
	This validates $(ii)$ as well as the PDE in $(iii).$ The continuity of $F$ at the boundary in $(iii)$ is also guaranteed by using \eqref{lem2:proof:eq1}, \eqref{add:eq---10}, and the dominated convergence theorem. The assertion $(iv)$ can follows directly from \eqref{lem2:proof:eq1} and noting $a<1$. 
	
	Finally, we prove $(v).$ Without loss of generality, we assume that $x\geq 0.$ Note that we can also use \eqref{lem2:proof:eq1}, \eqref{lem2:proof:eq2}, and the dominated convergence theorem to get that
	\begin{align*}
\partial_xF(s,x)=\frac{\e f'(x+\sqrt{s}z)e^{ m f(x+\sqrt{s} z)} }{\e e^{ m f(x+\sqrt{s} z)} }.
	\end{align*}
	From this, write
	\begin{align}
	\bigl|\partial_xF(s,x)\bigr|
	&\leq \Bigl|\frac{\e \bigl[f'(x+\sqrt{s} z)e^{ m f(x+\sqrt{s} z)};|z|\geq x\bigr]}{\e e^{ m f(x+\sqrt{s} z)} }\Bigr|+\Bigl|\frac{\e \bigl[f'(x+\sqrt{s}  z)e^{ m f(x+\sqrt{s}  z)};|z|<x\bigr]}{\e e^{ m f(x+\sqrt{s} z)} }\Bigr|.\label{add:eq---9}
	\end{align}
	To bound the first term, note that since $e^{mf(\cdot)}$ is convex, Jensen's inequality and the assumption $f\geq 0$ imply that
	\begin{align*}
	1\leq e^{mf(x)}=e^{mf(x+\sqrt{s}\e z)}\leq \e e^{ m f(x+\sqrt{s}z)}.
	\end{align*}
	On the other hand, from \eqref{lem2:proof:eq1} and \eqref{lem2:proof:eq2},
	\begin{align*}
	&\e \bigl[\bigl|f'(x+\sqrt{s}  z)\bigr|e^{ m f(x+\sqrt{s}  z)};|z|\geq x\bigr]\leq C\e \bigl[\bigl(1+(1+\sqrt{b})^a|z|^a\bigr)e^{ mC\bigl(1+(1+\sqrt{b})^{a+1}|z|^{a+1}\bigr)}\bigr]=:C',
	\end{align*}
		where $C'$ is finite due to the assumption that $a<1.$
    Putting these two inequalities together implies that the first term of \eqref{add:eq---9} is bounded above by $C'.$ For the second term in \eqref{add:eq---9}, note that
	\begin{align*}
	\e \bigl[\bigl|f'(x+\sqrt{s}  z)\bigr|e^{ m f(x+\sqrt{s}  z)};|z|<x\bigr]&\leq C(1+(1+\sqrt{s})^a|x|^a)\e \bigl[e^{ m f(x+\sqrt{s}  z)};|z|<x\bigr]\\
	&\leq C(1+(1+\sqrt{b})^a|x|^a)\e \bigl[e^{ m f(x+\sqrt{s}  z)}\bigr].
	\end{align*}
	From this, the second term in \eqref{add:eq---9} is bounded above by $C''|x|^a$ for some constant $C''>0$. These validate \eqref{lem2:proof:eq4}.
\end{proof}

\begin{lemma}\label{appendix:lem1}
	Let $p>2,t>0,u>0.$ Assume that $\lambda\in \mathbb{R}$ and $\gamma\in \mathcal{N}_u^d$ satisfies \eqref{gamma}. Define $\Psi_{\lambda,\gamma}^{t,u}$ from $[0,u]\times \mathbb{R}$ to $\mathbb{R}$ by letting
$
	\Psi_{\lambda,\gamma}^{t,u}(u,x)=f_\lambda^t(x)
$
	and iteratively for $l = k, k-1, \ldots, 0$
	\begin{align}
	\label{add:eq---11}
	\Psi_{\lambda,\gamma}^{t,u}(s,x)&=\frac{1}{m_l}\log \e\exp m_l \Psi_{\lambda,\gamma}^{t,u}(q_{l+1},x+\sqrt{2(q_{l+1}-s)}z),\,\,s\in [q_l,q_{l+1})\times \mathbb{R},
	\end{align}
	where $z\thicksim N(0,1).$ 
	Then for any $0\leq l\leq k,$ $\Psi_{\lambda,\gamma}^{t,u}\in C^{1,2}([q_l,q_{l+1})\times \mathbb{R})$ and it satisfies the following PDE,
	\begin{align}
	\begin{split}\label{appendix:lem1:eq0}
	\partial_s\Psi_{\lambda,\gamma}^{t,u}(s,x)&=-\bigl(\partial_{xx}\Psi_{\lambda,\gamma}^{t,u}(s,x)+\gamma(s)\bigl(\partial_x\Psi_{\lambda,\gamma}^{t,u}(s,x)\bigr)^2\bigr),\,\,(s,x)\in (q_l,q_{l+1})\times\mathbb{R}
	\end{split}
	\end{align}
   and
	\begin{align*}
		\lim_{s\uparrow q_{l+1},x\to x_0}\Psi_{\lambda,\gamma}^{t,u}(s,x)&=\Psi_{\lambda,\gamma}^{t,u}(q_{l+1},x_0),\,\,\forall x_0\in \mathbb{R}.
	\end{align*}
	Furthermore,
	there exists a constant $C>0$ such that
		\begin{align}\label{appendix:lem1:eq2}
	\bigl|\Psi_{\lambda,\gamma}^{t,u}(s,x)\bigr|&\leq C(1+|x|^{1+1/(p-1)}),\,\,\forall (s,x)\in [0,u]\times \mathbb{R},\\
	\label{appendix:lem1:eq1}
	\bigl|\partial_x\Psi_{\lambda,\gamma}^{t,u}(s,x)\bigr|&\leq C(1+|x|^{1/(p-1)}),\,\,\quad\forall (s,x)\in [0,u)\times \mathbb{R}.
	\end{align}
\end{lemma}

\begin{proof} Note that $\Psi_{\lambda,\gamma}^{t,u}(u,x)=f_\lambda^t(x)$ satisfies \eqref{lem2:proof:eq1} and \eqref{lem2:proof:eq2} with $a=(p-1)^{-1}<1$. Our proof follows by applying Lemma  \ref{add:lem2} iteratively. 
\end{proof}

\begin{proof}[\bf Proof of Proposition \ref{rep3}] Fix $(\lambda,\gamma)\in \mathbb{R}\times\mathcal{N}_u^d.$
	Since $\gamma(u)<\infty,$ from \eqref{appendix:lem1:eq1}, there exists some constant $C>0$ such that
	\begin{align*}
	\bigl|\gamma(s)\partial_x\Psi_{\lambda,\gamma}^{t,u}(s,x)\bigr|&\leq C(1+|x|^{1/(p-1)}),\,\,\forall (s,x)\in \bigl([0,u]\times\mathbb{R}\bigr)\setminus\{(u,0)\}.
	\end{align*}
	Consequently, there exists a probability space $(\Omega,\p,\mathcal{F})$, a filtration $(\mathcal{F}_s)_{0\leq s\leq u}$, a continuous process $X$, and a standard Brownian motion $W$ such that they together form a weak solution of \eqref{weak}, see, e.g., Proposition~5.3.6 in \cite{KS91}. Next, for any $v\in \mathcal{D}_u,$ set
$
	Y(s)=2\int_0^s\gamma(r)v(r)dr+\sqrt{2}W(r).
$
	Using It\^o's formula and the PDE \eqref{appendix:lem1:eq0} leads to
	\begin{align*}
	d\Psi_{\lambda,\gamma}^{t,u}(s,Y(s))=-\gamma(s) \bigl[\bigl(\partial_x\Psi_{\lambda,\gamma}^{t,u}(s,Y(s))\bigr)^2-2 v(s)\partial_x\Psi_{\lambda,\gamma}^{t,u}(s,Y(s))\bigr]ds+\sqrt{2}\partial_x\Psi_{\lambda,\gamma}^{t,u}(s,Y(s))dW(s).
	\end{align*}
	Here, we can further rewrite this equation as
	\begin{align*}
	d\Psi_{\lambda,\gamma}^{t,u}(s,Y(s))&=-\gamma(s) \bigl(\partial_x\Psi_{\lambda,\gamma}^{t,u}(s,Y(s))-v(s)\bigr)^2ds+\sqrt{2}\partial_x\Psi_{\lambda,\gamma}^{t,u}(s,Y(s))dW(s)+\gamma(s) v(s)^2ds.
	\end{align*}
	Consequently, 
	\begin{align*}
	\e\Psi_{\lambda,\gamma}^{t,u}(u,Y(u))&=\Psi_{\lambda,\gamma}^{t,u}(0,0)-\int_0^u\gamma(s)\e\bigl(\partial_x\Psi_{\lambda,\gamma}^{t,u}(s,Y(s))-v(s)\bigr)^2ds+\int_0^u \gamma(s)\e v(s)^2ds,
	\end{align*}
	which implies that
	\begin{align*}
	\Psi_{\lambda,\gamma}^{t,u}(0,0)&=\e\Psi_{\lambda,\gamma}^{t,u}(u,Y(u))-\int_0^u \gamma(s)\e v(s)^2ds+\int_0^u\gamma(s)\e\bigl(\partial_x\Psi_{\lambda,\gamma}^{t,u}(s,Y(s))-v(s)\bigr)^2ds.
	\end{align*}
	Clearly, this implies that 
	\begin{align*}
	\Psi_{\lambda,\gamma}^{t,u}(0,0)&\geq \max_{v\in \mathcal{D}_u}\Bigl[f_\lambda^t\Bigl(2\int_0^u \gamma(s) v(s)ds+\sqrt{2}W(u)\Bigr)-\int_0^u \gamma(s)v(s)^2ds\Bigr]
	\end{align*}
	and the equality is achieved if $v(s)\equiv\partial_x\Psi_{\lambda,\gamma}^{t,u}(s,X(s)),$ where $X$ is defined through \eqref{weak}.  This completes our proof.
	
\end{proof}

\bibliographystyle{plain}
{\footnotesize\bibliography{ref}}

\begin{thebibliography}{10}

\bibitem{Aldaz}
J.~Aldaz.
\newblock A stability version of {H}\"{o}lder's inequality for {$0<p<1$}.
\newblock {\em J. Inequal. Pure Appl. Math.}, 9(2):Article 60, 4, 2008.

\bibitem{ArguinChatterjee13}
L.-P. Arguin and S.~Chatterjee.
\newblock Random overlap structures: properties and applications to spin
  glasses.
\newblock {\em Probab. Theory Related Fields}, 156(1-2):375--413, 2013.

\bibitem{AC15}
A.~Auffinger and W.-K. Chen.
\newblock The {P}arisi formula has a unique minimizer.
\newblock {\em Comm. Math. Phys.}, 335(3):1429--1444, 2015.

\bibitem{AC2016}
A.~Auffinger and W.-K. Chen.
\newblock The {L}egendre structure of the {P}arisi formula.
\newblock {\em Comm. Math. Phys.}, 348(3):751--770, 2016.

\bibitem{AC171}
A.~Auffinger and W.-K. Chen.
\newblock Parisi formula for the ground state energy in the mixed {$p$}-spin
  model.
\newblock {\em Ann. Probab.}, 45(6B):4617--4631, 2017.

\bibitem{AZ20}
A.~Auffinger and Y.~Zhou.
\newblock On properties of the spherical mixed vector $p$-spin model.
\newblock {\em arXiv:2007.06019}, 2020.

\bibitem{BGG11}
A.~Barra, G.~Genovese, and F.~Guerra.
\newblock Equilibrium statistical mechanics of bipartite spin systems.
\newblock {\em J. Phys. A: Math. Theor.}, 44(24):245002, 2011.

\bibitem{biau15}
G.~Biau and D.~Mason.
\newblock High-dimensional {$p$}-norms.
\newblock In {\em Mathematical statistics and limit theorems}, pages 21--40.
  Springer, Cham, 2015.

\bibitem{BLM13}
S.~Boucheron, G.~Lugosi, and P.~Massart.
\newblock {\em Concentration inequalities: A nonasymptotic theory of
  independence}.
\newblock Oxford university press, 2013.

\bibitem{boucheron2012}
S.~Boucheron and M.~Thomas.
\newblock Concentration inequalities for order statistics.
\newblock {\em Electron. Commun. Probab.}, 17:no. 51, 12, 2012.

\bibitem{BovierKlimovsky09}
A.~Bovier and A.~Klimovsky.
\newblock The {A}izenman-{S}ims-{S}tarr and {G}uerra's schemes for the {SK}
  model with multidimensional spins.
\newblock {\em Electron. J. Probab.}, 14:no. 8, 161--241, 2009.

\bibitem{boyd74}
D.~Boyd.
\newblock The power method for $\ell_p$ norms.
\newblock {\em Linear Algebra Appl.}, 9:95--101, 1974.

\bibitem{CW04}
M.~Charikar and A.~Wirth.
\newblock Maximizing quadratic programs: extending {Grothendieck's} inequality.
\newblock In {\em 45th Annual IEEE Symposium on Foundations of Computer
  Science}, pages 54--60. IEEE, 2004.

\bibitem{chatterjee2009}
S.~Chatterjee.
\newblock Disorder chaos and multiple valleys in spin glasses.
\newblock {\em arXiv:0907.3381}, 2009.

\bibitem{Chen13Spherical}
W.-K. Chen.
\newblock The {A}izenman-{S}ims-{S}tarr scheme and {P}arisi formula for mixed
  {$p$}-spin spherical models.
\newblock {\em Electron. J. Probab.}, 18:no. 94, 14, 2013.

\bibitem{Chen17}
W.-K. Chen.
\newblock Variational representations for the {P}arisi functional and the
  two-dimensional {G}uerra-{T}alagrand bound.
\newblock {\em Ann. Probab.}, 45(6A):3929--3966, 2017.

\bibitem{CDP17}
W.-K. Chen, P.~Dey, and D.~Panchenko.
\newblock Fluctuations of the free energy in the mixed {$p$}-spin models with
  external field.
\newblock {\em Probab. Theory Related Fields}, 168(1-2):41--53, 2017.

\bibitem{CHL18}
W.-K. Chen, M.~Handschy, and G.~Lerman.
\newblock On the energy landscape of the mixed even {$p$}-spin model.
\newblock {\em Probab. Theory Related Fields}, 171(1-2):53--95, 2018.

\bibitem{CP18}
W.-K. Chen and D.~Panchenko.
\newblock On the {TAP} free energy in the mixed {$p$}-spin models.
\newblock {\em Comm. Math. Phys.}, 362(1):219--252, 2018.

\bibitem{ChenSen17}
W.-K. Chen and A.~Sen.
\newblock Parisi formula, disorder chaos and fluctuation for the ground state
  energy in the spherical mixed {$p$}-spin models.
\newblock {\em Comm. Math. Phys.}, 350(1):129--173, 2017.

\bibitem{chevet78}
S.~Chevet.
\newblock S{\'e}ries de variables al{\'e}atoires gaussiennesa valeurs dans {$E
  \hat \otimes_\varepsilon F$}. {A}pplication aux produits d’espaces de
  wiener abstraits.
\newblock {\em S{\'e}minaire sur la G{\'e}om{\'e}trie des Espaces de Banach
  (1977-1978)}, Exp. no. 19, 1978.

\bibitem{DMR20}
S.~Dhara, D.~Mukherjee, and K.~Ramanan.
\newblock The $ r $-to-$ p $ norm of non-negative random matrices: Asymptotic
  normality and entry-wise bounds.
\newblock {\em arXiv:2005.14056}, 2020.

\bibitem{DEZ15}
J.~Ding, R.~Eldan, and A.~Zhai.
\newblock On multiple peaks and moderate deviations for the supremum of a
  {G}aussian field.
\newblock {\em Ann. Probab.}, 43(6):3468--3493, 2015.

\bibitem{alaoui2020optimization}
A.~El~Alaoui, A.~Montanari, and M.~Sellke.
\newblock Optimization of mean-field spin glasses.
\newblock {\em arXiv:2001.00904}, 2020.

\bibitem{Eldan20}
R.~Eldan.
\newblock A simple approach to chaos for $p$-spin models.
\newblock {\em arXiv:2004.14885}, 2020.

\bibitem{embrechts2013}
P.~Embrechts, C.~Kl{\"u}ppelberg, and T.~Mikosch.
\newblock {\em Modelling extremal events: for insurance and finance},
  volume~33.
\newblock Springer Science \& Business Media, 2013.

\bibitem{fiedler73}
M.~Fiedler.
\newblock Algebraic connectivity of graphs.
\newblock {\em Czechoslovak Math. J.}, 23(2):298--305, 1973.

\bibitem{gordon85}
Y.~Gordon.
\newblock Some inequalities for gaussian processes and applications.
\newblock {\em Isr. J. Math}, 50(4):265--289, 1985.

\bibitem{guerra2003broken}
F.~Guerra.
\newblock Broken replica symmetry bounds in the mean field spin glass model.
\newblock {\em Comm. Math. Phys.}, 233(1):1--12, 2003.

\bibitem{GT02}
F.~Guerra and F.~Toninelli.
\newblock The thermodynamic limit in mean field spin glass models.
\newblock {\em Comm. Math. Phys.}, 230(1):71--79, 2002.

\bibitem{GRSW12}
V.~Guruswami, P.~Raghavendra, R.~Saket, and Y.~Wu.
\newblock Bypassing {UGC} from some optimal geometric inapproximability
  results.
\newblock In {\em Proceedings of the twenty-third annual ACM-SIAM symposium on
  Discrete Algorithms}, pages 699--717. SIAM, 2012.

\bibitem{JKS18inhomogeneousPotts}
A.~Jagannath, J.~Ko, and S.~Sen.
\newblock Max {$\kappa$}-cut and the inhomogeneous {P}otts spin glass.
\newblock {\em Ann. Appl. Probab.}, 28(3):1536--1572, 2018.

\bibitem{JS17}
A.~Jagannath and S.~Sen.
\newblock On the unbalanced cut problem and the generalized
  {Sherrington-Kirkpatrick} model.
\newblock {\em arXiv:1707:09042}, 2018.

\bibitem{JT16}
A.~Jagannath and I.~Tobasco.
\newblock A dynamic programming approach to the {P}arisi functional.
\newblock {\em Proc. Amer. Math. Soc.}, 144(7):3135--3150, 2016.

\bibitem{KS91}
I.~Karatzas and S.~Shreve.
\newblock {\em Brownian motion and stochastic calculus}, volume 113 of {\em
  Graduate Texts in Mathematics}.
\newblock Springer-Verlag, New York, second edition, 1991.

\bibitem{KS03}
B.~Kashin and S.~Szarek.
\newblock On the {Gram} matrices of systems of uniformly bounded functions.
\newblock {\em Tr. Mat. Inst. Steklova}, 243(Funkts. Prostran., Priblizh.,
  Differ. Uravn.):237--243, 2003.

\bibitem{KN12}
S.~Khot and A.~Naor.
\newblock Grothendieck-type inequalities in combinatorial optimization.
\newblock {\em Comm. Pure Appl. Math.}, 65(7):992--1035, 2012.

\bibitem{KS11}
S.~Khot and M.~Safra.
\newblock A two-prover one-round game with strong soundness.
\newblock In {\em 2011 IEEE 52nd Annual Symposium on Foundations of Computer
  Science}, pages 648--657. IEEE, 2011.

\bibitem{KNS10}
G.~Kindler, A.~Naor, and G.~Schechtman.
\newblock The {UGC} hardness threshold of the {$L_p$} {G}rothendieck problem.
\newblock {\em Math. Oper. Res.}, 35(2):267--283, 2010.

\bibitem{JustinKo19}
J.~Ko.
\newblock The crisanti-sommers formula for spherical spin glasses with vector
  spins.
\newblock {\em arXiv:1911.04355}, 2019.

\bibitem{Ko20multispherical}
J.~Ko.
\newblock Free energy of multiple systems of spherical spin glasses with
  constrained overlaps.
\newblock {\em Electron. J. Probab.}, 25:Paper No. 28, 34, 2020.

\bibitem{M01}
A.~Megretski.
\newblock Relaxations of quadratic programs in operator theory and system
  analysis.
\newblock In {\em Systems, approximation, singular integral operators, and
  related topics}, pages 365--392. Springer, 2001.

\bibitem{MPV87}
M.~M\'{e}zard, G.~Parisi, and M.~A. Virasoro.
\newblock {\em Spin glass theory and beyond}, volume~9 of {\em World Scientific
  Lecture Notes in Physics}.
\newblock World Scientific Publishing Co., Inc., Teaneck, NJ, 1987.

\bibitem{montanari2019optimization}
A.~Montanari.
\newblock Optimization of the {Sherrington-Kirkpatrick Hamiltonian}.
\newblock In {\em 2019 IEEE 60th Annual Symposium on Foundations of Computer
  Science (FOCS)}, pages 1417--1433. IEEE, 2019.

\bibitem{MS65}
T.~Motzkin and E.~Straus.
\newblock Maxima for graphs and a new proof of a theorem of {T}ur{\'a}n.
\newblock {\em Canadian J. Math.}, 17:533--540, 1965.

\bibitem{mourrat20}
J.-C. Mourrat.
\newblock Nonconvex interactions in mean-field spin glasses.
\newblock {\em arXiv:2004.01679}, 2020.

\bibitem{NRT99}
A.~Nemirovski, C.~Roos, and T.~Terlaky.
\newblock On maximization of quadratic form over intersection of ellipsoids
  with common center.
\newblock {\em Math. Program.}, 86(3):463--473, 1999.

\bibitem{PM08}
M.~Palassini.
\newblock Ground-state energy fluctuations in the {Sherrington–Kirkpatrick}
  model.
\newblock {\em J. Stat. Mech. Theory Exp.}, 2008(3):P10005, 2008.

\bibitem{Pan05}
D.~Panchenko.
\newblock Free energy in the generalized {S}herrington-{K}irkpatrick mean field
  model.
\newblock {\em Rev. Math. Phys.}, 17(7):793--857, 2005.

\bibitem{Pan13}
D.~Panchenko.
\newblock {\em The {S}herrington-{K}irkpatrick model}.
\newblock Springer Monographs in Mathematics. Springer, New York, 2013.

\bibitem{Pan14}
D.~Panchenko.
\newblock The {P}arisi formula for mixed {$p$}-spin models.
\newblock {\em Ann. Probab.}, 42(3):946--958, 2014.

\bibitem{Pan15Multiplespecies}
D.~Panchenko.
\newblock The free energy in a multi-species {S}herrington-{K}irkpatrick model.
\newblock {\em Ann. Probab.}, 43(6):3494--3513, 2015.

\bibitem{Pan18vector}
D.~Panchenko.
\newblock Free energy in the mixed {$p$}-spin models with vector spins.
\newblock {\em Ann. Probab.}, 46(2):865--896, 2018.

\bibitem{Pan18Potts}
D.~Panchenko.
\newblock Free energy in the {P}otts spin glass.
\newblock {\em Ann. Probab.}, 46(2):829--864, 2018.

\bibitem{paouris2017}
G.~Paouris, P.~Valettas, and J.~Zinn.
\newblock Random version of {D}voretzky’s theorem in $\ell_n^\infty$.
\newblock {\em Stochastic Process. Appl.}, 127(10):3187--3227, 2017.

\bibitem{Parisi79}
G.~Parisi.
\newblock Toward a mean field theory for spin glasses.
\newblock {\em Phys. Lett.}, 73A:203--205, 1979.

\bibitem{Parisi802}
G.~Parisi.
\newblock The order parameter for spin glasses: a function on the interval
  0--1.
\newblock {\em J. Phys.}, A13:1101--1112, 1980.

\bibitem{Parisi801}
G.~Parisi.
\newblock A sequence of approximated solutions to the sk model for spin
  glasses.
\newblock {\em J. Phys.}, A13:L115--L121, 1980.

\bibitem{schechtman2004}
G.~Schechtman.
\newblock The random version of {D}voretzky’s theorem in $\ell_n^\infty$.
\newblock {\em Geometric Aspects of Functional Analysis}, pages 265--270, 2004.

\bibitem{SK72}
D.~Sherrington and S.~Kirkpatrick.
\newblock Solvable model of a spin glass.
\newblock {\em Phys. Rev. Lett.}, 35:1792--1796, 1972.

\bibitem{subag18}
E.~Subag.
\newblock Following the ground-states of full-rsb spherical spin glasses.
\newblock {\em arXiv preprint arXiv:1812.04588}, 2018.

\bibitem{Tal06}
M.~Talagrand.
\newblock The {P}arisi formula.
\newblock {\em Ann. of Math. (2)}, 163(1):221--263, 2006.

\bibitem{Tal111}
M.~Talagrand.
\newblock {\em Mean field models for spin glasses. {V}olume {I}}.
\newblock Springer-Verlag, Berlin, 2011.
\newblock Basic examples.

\bibitem{Tal112}
M.~Talagrand.
\newblock {\em Mean field models for spin glasses. {V}olume {II}}.
\newblock Springer, Heidelberg, 2011.
\newblock Advanced replica-symmetry and low temperature.

\bibitem{TW1}
C.~Tracy and H.~Widom.
\newblock Level-spacing distributions and the airy kernel.
\newblock {\em Comm. Math. Phys.}, 159(1):151--174, 1994.

\bibitem{TW2}
C.~Tracy and H.~Widom.
\newblock On orthogonal and symplectic matrix ensembles.
\newblock {\em Comm. Math. Phys.}, 177(3):727--754, 1996.

\bibitem{van14}
R.~Van~Handel.
\newblock Probability in high dimension.
\newblock Technical report, Princeton Univ NJ, 2014.
\newblock \url{https://web.math.princeton.edu/$\sim$rvan/APC550.pdf}.

\bibitem{HDP}
R.~Vershynin.
\newblock {\em High-dimensional probability: An introduction with applications
  in data science}, volume~47.
\newblock Cambridge university press, 2018.

\end{thebibliography}

\end{document}